\documentclass[english,10pt]{amsart}

\usepackage[english]{babel}

\usepackage{amscd,amssymb,amsfonts,amsmath}
\usepackage{tikz-cd}

\usepackage{bm}
\usepackage{graphics}
\usepackage{epsfig}
\usepackage{mathrsfs}
\usepackage{xcolor}
\DeclareMathAlphabet{\mathscrbf}{OMS}{mdugm}{b}{n}

\definecolor{violet}{rgb}{0.0,0.2,0.7}
\definecolor{rouge2}{rgb}{0.8,0.0,0.2}
\usepackage{hyperref}
\hypersetup{
    unicode=false,          
    pdftoolbar=true,        
    pdfmenubar=true,        
    pdffitwindow=false,     
    pdfstartview={FitH},    
    pdftitle={},    
    pdfauthor={},     
    colorlinks=true,       
   linkcolor=violet,          
    citecolor=rouge2,        
    filecolor=black,      
    urlcolor=cyan}           

\usepackage[text={6.7in,9.2in},centering]{geometry}

\makeatletter
\renewcommand\subsection{\@startsection{subsection}{2}%
  \z@{.5\linespacing\@plus.7\linespacing}{-.5em}%
  {\normalfont\sffamily}}  
\makeatother

\renewcommand{\div}{\textup{div}}

\newcommand{\NS}{\textup{N}^1}

\renewcommand{\phi}{\varphi}

\newcommand{\map}{\dashrightarrow}

\newcommand{\wt}{\widetilde}

\newcommand{\wb}{\overline}

\renewcommand{\le}{\leqslant}
\renewcommand{\ge}{\geqslant}

\newcommand{\mult}{\textup{mult}}

\newcommand{\sA}{\mathscr{A}}
\newcommand{\sB}{\mathscr{B}}
\newcommand{\sC}{\mathscr{C}}
\newcommand{\sD}{\mathscr{D}}
\newcommand{\sE}{\mathscr{E}}
\newcommand{\sF}{\mathscr{F}}
\newcommand{\sG}{\mathscr{G}}
\newcommand{\sH}{\mathscr{H}}
\newcommand{\sI}{\mathscr{I}}

\newcommand{\sK}{\mathscr{K}}
\newcommand{\sL}{\mathscr{L}}
\newcommand{\sM}{\mathscr{M}}
\newcommand{\sN}{\mathscr{N}}
\newcommand{\sO}{\mathscr{O}}

\newcommand{\sQ}{\mathscr{Q}}
\newcommand{\sR}{\mathscr{R}}

\newtheorem{thm}{Theorem}[section]

\newtheorem{lemma}[thm]{Lemma}
\newtheorem{cor}[thm]{Corollary}

\newtheorem{prop}[thm]{Proposition}

\newtheorem*{thm*}{Theorem}
\theoremstyle{definition}
\newtheorem{defn}[thm]{Definition}

\newtheorem{defn-thm}[thm]{Definition-Theorem} 
\newtheorem{defn-lemma}[thm]{Definition-Lemma}

\theoremstyle{remark}
\newtheorem{claim}[thm]{Claim}

\newtheorem*{not-and-def}{Notation and definitions}
\newtheorem{assumption}[thm]{Assumption} 
\newtheorem{rem}[thm]{Remark}

\newtheorem{exmp}[thm]{Example}
\newtheorem{setup}[thm]{Setup}

\numberwithin{equation}{section}

\def\factor#1.#2.{\left. \raise 2pt\hbox{$#1$} \right/\hskip -2pt\raise -2pt\hbox{$#2$}}

\begin{document} 

\title[Projectively flat foliations]{Projectively flat foliations}

\author{St\'ephane \textsc{Druel}}

\address{St\'ephane Druel: CNRS, Université Claude Bernard Lyon 1, UMR 5208, Institut Camille Jordan, F-69622 Villeurbanne, France} 

\email{stephane.druel@math.cnrs.fr}


\begin{abstract}
We describe the structure of regular codimension $1$ foliations with numerically projectively flat tangent bundle on complex projective manifolds of dimension at least $4$. Along the way, we prove that either the normal bundle of a regular codimension $1$ foliation is pseudo-effective, or its conormal bundle is nef.
\end{abstract}

\maketitle

{\small\tableofcontents}

\section{Introduction}

Let $X$ be a smooth projective algebraic variety over the field of complex numbers, of dimension $n \ge 2$. If $T_X$ is numerically flat, then $X$ is a finite \'etale quotient of an abelian variety as a classical consequence of Yau's theorem on the existence of a K\"ahler-Einstein metric. In \cite{jahnke_radloff_13}, Jahnke and Radloff proved that if the normalized tangent bundle $\textup{S}^nT_X\otimes\sO_X(K_X)$ is numerically flat, then $K_X \equiv 0$, and hence $T_X$ is numerically flat. Thus $X$ is again a finite \'etale quotient of an abelian variety. Let $\sE$ be a vector bundle of rank $r \ge 1$ on $X$. Recall from \cite[Theorem 1.1]{jahnke_radloff_13} that the normalized vector bundle $\textup{S}^n\sE\otimes\det(\sE^*)$ is numerically flat if and only if $\sE$ is semistable with respect to some ample divisor $H$ on $X$ and equality holds in the Bogomolov-Gieseker inequality,
$$c_2(\sE)\cdot H^{n-2}=\frac{r-1}{2r}c_1(\sE)^2\cdot H^{n-2}.$$
The aim of this paper is to generalize the theorem of Jahnke and Radloff to codimension $1$ regular foliations.

\begin{thm}\label{thm_intro:main}
Let $X$ be a complex projective manifold of dimension $n\ge 4$, and let $\sF \subset T_X$ be a regular codimension $1$ foliation. Suppose that the normalized vector bundle $\textup{S}^{n-1}\sF\otimes \det(\sF^*)$ is numerically flat. Then one of the following holds.
\begin{enumerate}
\item There exists a $\mathbb{P}^1$-bundle structure $\phi\colon X \to Y$ onto a finite \'etale quotient of an abelian variety, and $\sF$ induces a flat Ehresmann connection on $\phi$.
\item There exists an abelian variety $A$ as well as a finite \'etale cover $\gamma\colon A \to X$ such that 
$\gamma^{-1}\sF$ is a linear foliation. 
\item There exists an abelian scheme $f\colon B \to C$ onto a smooth complete curve of genus at least $2$ as well as a finite \'etale cover $\gamma\colon B \to X$ such that $\gamma^{-1}\sF$ is induced by $f$.
\item The divisor $K_X$ is ample and $\kappa(K_\sF)=\dim X$. Moreover, the vector bundles $\sN^*$ and $\Omega_X^1$ are nef.
\end{enumerate}
\end{thm}

Notice that no example of foliations satisfying the conclusion of Theorem \ref{thm_intro:main}~(4) is known to the author. In addition, if we relax the condition on the dimension of $X$ then the conclusion of Theorem \ref{thm_intro:main} may be false (see \cite{brunella_aens}). 

\medskip

The proof of Theorem \ref{thm_intro:main} makes use of the following result, which might be of independent interest.

\begin{thm}\label{thm_intro:divisor_zero_square}
Let $X$ be a complex projective manifold of dimension $n \ge 1$, and let $\sF \subset T_X$ be a regular codimension $1$ foliation with normal bundle $\sN$. Then, either $\sN$ is pseudo-effective or $\sN^*$ is nef.
\end{thm}

In fact, a more general statement is true (see Theorem \ref{thm:divisor_zero_square}).

\subsection*{Previous results}
Greb, Kebekus, and Peternell generalized the theorem of Jahnke and Radloff to projective varieties with klt singularities in \cite{GKP_proj_flat_JEP}. More precisely, they proved that torus quotients are the only klt varieties with semistable tangent sheaf and extremal Chern classes. In \cite{druel_proj_flat}, the structure of projective log smooth pairs with numerically flat normalized logarithmic tangent bundle is described. Passing to an appropriate finite cover and up to isomorphism, these are projective spaces and log smooth pairs with numerically flat logarithmic tangent bundles blown-up at finitely many points away from the boundary. In addition, it has been shown in \cite{druel_lo_bianco} that log smooth pairs with numerically flat logarithmic tangent bundle are toric fiber bundles over finite \'etale quotients of abelian varieties.

\subsection*{Structure of the paper} 
Section \ref{section:notation} gathers notation, known results and global conventions that will be used throughout the paper. In Section \ref{section:preparation}, we provide technical tools for the proof of our main results. Section \ref{section:normal_bundle} is mostly taken up by the proof of Theorem \ref{thm_intro:divisor_zero_square}. The proof of Theorem \ref{thm_intro:main} is long and therefore subdivided into numerous steps: Sections \ref{section:non_minimal} to \ref{section:pseff_normal} prepare for it. In section \ref{section:non_minimal}, we describe regular codimension $1$ foliations with numerically projectively flat tangent bundle on non-minimal complex projective manifolds.
Next, we prove abundance for $K_X$ in Section \ref{section:abundance} following the strategy employed in \cite{jahnke_radloff_13}: we use a Shafarevich map construction to prove that it reduces to the special case where $\sF \cong \sL^{\oplus n}$ for some line bundle $\sL$.
In Section \ref{section:nef_conormal}, we address regular codimension $1$ foliations with numerically projectively flat tangent bundle and nef conormal bundle on minimal varieties.
Section \ref{section:pseff_normal} describes regular codimension $1$ foliations with numerically projectively flat tangent bundle and pseudo-effective normal bundle still on minimal varieties. To this end, we show in Section \ref{section:nef_reduction} that the nef reduction map for $K_\sF$ is an abelian scheme up to a finite \'etale cover. With these preparations at hand, the proof of Theorem \ref{thm_intro:main} which we give in Section \ref{section:proof} becomes short.

\subsection*{Acknowledgements} 

This article owns much to the ideas and work of Marco Brunella. We would also like to thank Michel Brion, Beno\^{i}t Claudon, and Adrien Dubouloz for their help and many interesting discussions.

\bigskip

\section{Notation, conventions, and used facts}\label{section:notation}

\subsection{Global conventions} Throughout the paper, all varieties are assumed to be defined over the field of complex numbers. We will freely switch between the algebraic and analytic context. 

\subsection{Projective space bundles}
If $\sE$ is a locally free sheaf of finite rank on a variety $X$, 
we denote by $\mathbb{P}_X(\sE)$ the variety $\textup{Proj}_X(\textup{S}^\bullet\sE)$.

\subsection{Stability}
The word \textit{stable} will always mean \textit{slope-stable with respect to a
given ample divisor}. Ditto for \textit{semistable} and \textit{polystable}.
We refer to \cite[Definition 1.2.12]{HuyLehn} for their precise definitions.

\subsection{Numerically flat vector bundles} For the reader's convenience, we recall the definition of numerically flat vector bundles following \cite[Definition 1.17]{demailly_peternell_schneider94}.

\begin{defn}
A vector bundle $\sE$ of rank $r\ge 1$ on a smooth projective variety is called \textit{numerically flat} if $\sE$ and $\sE^*$ are nef vector bundes.
\end{defn}

\begin{rem}\label{rem:num_flat_zero_first_Chern_class}
Let $X$ be a smooth projective variety and let $\sE$ be a vector bundle of rank $r \ge 1$ on $X$ with $c_1(\sE) \equiv 0$. Then 
$\sE$ is numerically flat if and only if $\sE$ is nef.
\end{rem}

\begin{thm}\label{thm:numerically_flat}
Let $X$ be a smooth projective variety and let $\sE$ be a vector bundle of rank $r \ge 1$ on $X$. Let $H$ be an ample divisor on $X$. Then the following conditions are equivalent.
\begin{enumerate}
\item The vector bundle $\sE$ is numerically flat.
\item The vector bundle $\sE$ is flat and semistable with respect to $H$.
\item The vector bundle $\sE$ admits a filtration 
$$\{0\}=\sE_0 \subset \sE_1 \subset \cdots \subset \sE_m=\sE$$
by vector subbundles such that the quotients $\sQ_k:=\sE_k/\sE_{k-1}$ are hermitian flat.
\end{enumerate}
\end{thm}

\begin{proof}
This follows easily from \cite[Theorem 1.18]{demailly_peternell_schneider94} and \cite{simpson_higgs_flat} together (see \cite[Theorem 2.9]{druel_proj_flat}).
\end{proof}

\begin{rem}
Setting and notation as in Theorem \ref{thm:numerically_flat}. If $\sE$ is numerically flat, then  it is semistable with respect to any ample divisor.
\end{rem}

\begin{lemma}[{\cite[Corollary 1.19]{demailly_peternell_schneider94}}]
Let $\sE$ be a vector bundle on a smooth projective variety. If $\sE$ is numerically flat, then $c_i(\sE)\equiv 0$ for any integer $i \ge 1$. 
\end{lemma}

\subsection{Projectively flat and numerically projectively flat vector bundles} Next, we recall the definition of projectively flat vector bundles.

\begin{defn}\label{def:proj_flat}
A vector bundle $\sE$ of rank $r\ge 1$ on a smooth projective variety $X$ is called \textit{projectively flat} if $\mathbb{P}(\sE)$ comes from a  
representation $\pi_1(X) \to \textup{PGL}(r,\mathbb{C})$.
\end{defn}

\begin{lemma}[{\cite[Proof of Proposition 1.1 (2)]{jahnke_radloff_13}}]\label{lemma:chern_classes}
Let $\sE$ be a vector bundle on a smooth projective variety. If $\sE$ is projectively flat, then $c_i(\sE)\equiv\frac{1}{r^i}\binom{r}{i}c_1(\sE)^i$ for any integer $i \ge 1$. 
\end{lemma}

One key notion is that of a numerically projectively flat vector bundle. We recall the definition.

\begin{defn}\label{def:num_proj_flat}
A vector bundle $\sE$ of rank $r\ge 1$ on a smooth projective variety $X$ is called \textit{numerically projectively flat} if $\textup{S}^r\sE \otimes \det \sE^*$ is numerically flat.
\end{defn}

\begin{rem}
Setting and notation as in Definition \ref{def:num_proj_flat}. Then $\sE$ is numerically projectively flat if and only if $\textup{S}^r\sE \otimes \det \sE^*$ is nef since 
$\det(\textup{S}^r\sE \otimes \det \sE^*)\cong\sO_X$ (see Remark \ref{rem:num_flat_zero_first_Chern_class}).
\end{rem}

\begin{rem}
Setting and notation as in Definition \ref{def:num_proj_flat}. Then $\sE$ is numerically flat if and only if $\sE$ is numerically projectively flat and $c_1(\sE)\equiv 0$.
\end{rem}

We will need the following characterization of numerically projectively flat vector bundles.

\begin{thm}\label{thm:numerically_projectively_flat}
Let $X$ be a smooth projective variety of dimension $n\ge 2$ and let $\sE$ be a vector bundle of rank $r \ge 1$ on $X$. Let $H$ be an ample divisor on $X$. Then the following conditions are equivalent.
\begin{enumerate}
\item The vector bundle $\sE$ is numerically projectively flat.
\item The vector bundle $\sE$ is projectively flat and semistable with respect to $H$.
\item The vector bundle $\sE$ admits a filtration 
$$\{0\}=\sE_0 \subset \sE_1 \subset \cdots \subset \sE_m=\sE$$
by vector subbundles such that the quotients $\sQ_k:=\sE_k/\sE_{k-1}$ are hermitian projectively flat with $$\frac{1}{\textup{rank}\,\sQ_k}
c_1(\sQ_k)\equiv \frac{1}{r}c_1(\sE).$$
\item The vector bundle $\sE$ is semistable with respect to $H$ and equality holds in the Bogomolov-Gieseker inequality,
$$c_2(\sE)\cdot H^{n-2}=\frac{r-1}{2r}c_1(\sE)^2\cdot H^{n-2}.$$
\item For any morphism $C \to X$ from a smooth complete curve $C$, the vector bundle $\sE|_C$ is semistable.
\end{enumerate}
\end{thm}

\begin{proof}
This follows easily from \cite[Theorem 1.1]{jahnke_radloff_13} and \cite[Proposition 1.1]{jahnke_radloff_13} together (see also \cite[Theorem 2.13]{druel_proj_flat}).
\end{proof}

\begin{rem}
Setting and notation as in Theorem \ref{thm:numerically_projectively_flat}. If $\sE$ is numerically projectively flat, then it is semistable with respect to any ample divisor.
\end{rem}

\begin{exmp}
Let $\sL$ be a line bundle on a smooth projective variety $X$. Then $\sL^{\oplus m}$ is numerically projectively flat for any integer $m\ge 1$.
\end{exmp}

The following is an easy consequence of Theorem \ref{thm:numerically_projectively_flat}.

\begin{lemma}\label{lemma:numerically_projectively_flat_elementary properties}
Let $X$ be a smooth projective variety of dimension $n\ge 1$, and let $\sE$ be a vector bundle of rank $r \ge 1$ on $X$.
\begin{enumerate}
\item If $\sE$ is numerically projectively flat, then $\sE\otimes\sL$ is numerically projectively flat as well for any line bundle $\sL$ on $X$.
\item Let $Y$ be a smooth projective variety and let $f \colon Y \to X$ be a morphism. If $\sE$ is numerically projectively flat, then $f^*\sE$ is also numerically projectively flat. In addition, the converse statement is true if $f$ is surjective. 
\item Let $\sG$ be a vector bundle on $X$ of rank $s \ge 1$. Then $\sE\oplus\sG$ is numerically projectively flat if and only if $\sE$ and $\sG$ are numerically projectively flat and 
$$\frac{1}{r}c_1(\sE)\equiv \frac{1}{s}c_1(\sG)\equiv \frac{1}{r+s}c_1(\sE\oplus \sG).$$
\end{enumerate}
\end{lemma}

We will need the following easy observation.

\begin{lemma}\label{lemma:descent_vb}
Let $f\colon X \to Y$ be a smooth projective morphism with connected fibers between smooth algebraic varieties, and let $\sE$ be  a vector bundle of rank $r \ge 1$ on $X$. Suppose that the restriction of $\sE$ to every fiber of $f$ is numerically projectively flat.
Suppose in addition that $\sE|_F\cong \sO_F^{\oplus r}$, where $F$ is a general fiber of $f$. Then there exists a vector bundle $\sG$ on $Y$ such that $\sE \cong f^*\sG$.
\end{lemma}

\begin{proof}
Let $F$ be any fiber of $f$. Because the function
$y \mapsto h^0(X_y,\sE|_{X_y})$ are upper semicontinuous in the Zariski topology on $Y$ (see \cite[Theorem 12.8]{hartshorne77}), we have 
$h^0(F,\sE|_F) \ge r$.  
By \cite[Proposition 1.16]{demailly_peternell_schneider94}, the evaluation map $H^0(F,\sE|_F)\otimes\sO_F \to \sE|_F$ is an injective morphism of vector bundles. This implies that $\sE|_F \cong \sO_F^{\oplus r}$. By \cite[Corollary 12.9]{hartshorne77} together with the base change theorem \cite[Theorem 12.11]{hartshorne77}, the sheaf $\sG:=f_*\sE$ is locally free of rank $r$ and the formation of $f_*\sE$ commutes with arbitrary base change. This easily implies that $\sE \cong f^*\sG$.
\end{proof}

\subsection{Foliation} In this section, we have gathered a number of results and facts concerning foliations which will later be used
in the proofs.

\begin{defn}
A \emph{foliation} on a smooth variety $X$ is a coherent subsheaf $\sF\subseteq T_X$ such that
\begin{enumerate}
\item $\sF$ is closed under the Lie bracket, and
\item $\sF$ is saturated in $T_X$. In other words, the quotient $T_X/\sF$ is torsion-free.
\end{enumerate}

\medskip

The \emph{rank} $r$ of $\sF$ is the generic rank of $\sF$.
The \emph{codimension} of $\sF$ is defined as $q:=\dim X-r$. 

\medskip

The \textit{canonical class} $K_{\sF}$ of $\sF$ is any Weil divisor on $X$ such that 
$\sO_X(-K_{\sF})\cong \det\sF$. 

\medskip

Let $X^\circ \subseteq X$ be the open set where $\sF$ is a subbundle of $T_X$. The foliation is called \emph{regular} if $X^\circ = X$.

\medskip

A \emph{leaf} of $\sF$ is a maximal connected and immersed holomorphic submanifold $L \subseteq X^\circ$ such that
$T_L=\sF_{|L}$.
\end{defn}

The \emph{normal sheaf} of $\sF$ is $\sN:=(T_X/\sF)^{**}$.
The $q$-th wedge product of the inclusion
$\sN^*\subseteq \Omega^1_X$ gives rise to a non-zero global section $\omega\in H^0(X,\Omega^{q}_X\otimes \det\sN)$ whose zero locus has codimension at least $2$ in $X$. 
Moreover, $\omega$ is \emph{locally decomposable} and \emph{integrable}.
To say that $\omega$ is locally decomposable means that, 
in a neighborhood of a general point of $X$, $\omega$ decomposes as the wedge product of $q$ local $1$-forms 
$\omega=\omega_1\wedge\cdots\wedge\omega_q$.
To say that it is integrable means that for this local decomposition one has 
$d\omega_i\wedge \omega=0$ for every  $i\in\{1,\ldots,q\}$. 
The integrability condition for $\omega$ is equivalent to the condition that $\sF$ 
is closed under the Lie bracket. Conversely, let $\sL$ be a line bundle on $X$, and let
$\omega\in H^0(X,\Omega^{q}_X\otimes \sL)$ be a global section
whose zero locus has codimension at least $2$ in $X$.
Suppose that $\omega$  is locally decomposable and integrable.
Then the kernel of the morphism $T_X \to \Omega^{q-1}_X\otimes \sL$ given by the contraction with $\omega$ defines a foliation of codimension $q$ on $X$. 
These constructions are inverse of each other. 

\medskip

Let $X$ and $Y$ be smooth complex algebraic varieties, and let $f\colon X\to Y$ be a dominant morphism that restricts to a smooth morphism $f^\circ\colon X^\circ\to Y^\circ$,
where $X^\circ\subseteq X$ and $Y^\circ\subseteq Y$ are dense open subsets.
Let $\sF$ be a foliation on $Y$. \emph{The pull-back $f^{-1}\sF$ of $\sF$ via $f$} is the foliation on $X$ whose restriction to $X^\circ$ is $(df^\circ)^{-1}\big(\sF_{|Y^\circ}\big)$.
We will need the following easy observation.

\begin{lemma}\label{lemma:finite_etale_cover}
Let $f\colon X \to Y$ be a finite \'etale cover of smooth projective varieties, and let $\sF$ be a regular foliation on $Y$. Then the following holds.
\begin{enumerate}
\item The foliation $f^{-1}\sF$ is regular. Moreover $f^{-1}\sF \cong f^*\sF$, and $\sN_{f^{-1}\sF} \cong f^*\sN_\sF$.
\item We have $\kappa(K_{f^{-1}\sF})=\kappa(K_\sF)$, and $\nu(K_{f^{-1}\sF})=\nu(K_\sF)$.
\end{enumerate}
\end{lemma}

\begin{proof}
The first part of the statement is obvious. For the second part, note that $K_{f^{-1}\sF}=f^*K_\sF$ by (1). Then we have $\kappa(K_{f^{-1}\sF})=\kappa(K_\sF)$ by \cite[Theorem 5.13]{uenoLN439}, and $\nu(K_{f^{-1}\sF})=\nu(K_\sF)$ by \cite[Proposition V.2.22]{nakayama04}.
\end{proof}

\subsection{Ehresmann connection} Let $f \colon X \to Y$ be a surjective proper morphism between smooth varieties. An \textit{Ehresmann connection} on $f$ is a distribution $\sD \subseteq T_X$ such that the restriction of the tangent map 
$Tf\colon T_X \to f^*T_Y$ to $\sD$ induces an isomorphism $\sD \cong f^*T_Y$. The Ehresmann connection $\sD$ is said to be \textit{flat} if $\sD$ is a foliation.

If $f$ has an Ehresmann connection, then $f$ is a locally trivial fibration for the analytic topology since complex flows of vector fields on analytic spaces exist. 

Suppose that $f$ has a flat Ehresmann connection $\sD$. Let $F$ be any fiber of $f$, and set $y:=f(F)$. Let also $\rho\colon \pi_1(Y,y) \to \textup{Aut}(F)$ be the monodromy representation of $\sD$. Then $$X \cong (\wt{Y} \times F)/\pi_1(Y,y),$$
where $\wt{Y}$ denotes the universal cover of $Y$ based at $y$ and $\pi_1(Y,y)$ acts diagonally on the product.

\begin{exmp}\label{example:suspension} Let $n \ge 2$ be an integer.
Let $A=\mathbb{C}^{n-1}/\Lambda$ be a complex abelian variety, and let 
$\rho\colon \pi_1(A) \to \textup{PGL}(2,\mathbb{C})$ be a representation of the fundamental group $\pi_1(A)\cong\Lambda$ of $A$.
Then the group $\pi_1(A)$ acts diagonally on $\mathbb{C}^{n-1}\times \mathbb{P}^1$ by 
$\gamma\cdot(z,p)=\big(\gamma(z),\rho(\gamma)(p)\big)$. Set $X:=(\mathbb{C}^{n-1} \times \mathbb{P}^1)/\pi_1(A)$, 
and denote by $\psi\colon X \to A \cong \mathbb{C}^{n-1}/\pi_1(A)$ the projection morphism, which is $\mathbb{P}^1$-bundle.
The foliation on $\mathbb{C}^{n-1} \times \mathbb{P}^1$ induced by the projection $\mathbb{C}^{n-1} \times \mathbb{P}^1 \to \mathbb{P}^1$ is invariant under the action of $\pi_1(A)$ and gives a flat Ehresmann connection $\sF$ on $\psi$. Then $\rho$ identifies with the monodromy representation of $\sF$.
\end{exmp}

\subsection{The standard setting} Throughout the present paper, we will be working in the following setup and use the notation below.

\begin{setup}\label{setup:main}
Let $X$ be a complex projective manifold of dimension $n\ge 3$, and let $\sF \subset T_X$ be a regular codimension $1$ foliation with normal bundle $\sN$. Suppose that $\sF$ is numerically projectively flat. Let $\rho\colon \pi_1(X) \to \textup{PGL}(n-1,\mathbb{C})$ denote a representation that defines the projectively flat structure on $\sF$.
\end{setup}

\begin{rem}
Setting and notation as in Setup \ref{setup:main}. Let $f \colon Y \to X$ be a finite \'etale cover. Then the pull-back $f^{-1}\sF$ of $\sF$ on $Y$ satisfies all the conditions listed in Setup \ref{setup:main} (see Lemma \ref{lemma:finite_etale_cover}).
\end{rem}

We will need the following observation.

\begin{lemma}\label{lemma:minimal}
Setting and notation as in \ref{setup:main}. Then $\sF^*$ is a nef vector bundle.
\end{lemma}

\begin{proof}In order to prove the lemma, it suffices to show that $K_\sF$ is nef since $\sF$ is numerically projectively flat by assumption.
We argue by contradiction and assume that $K_\sF$ is not nef. Then there exists a smooth complete curve $C \to X$ such that $\deg_C K_\sF <0$. This implies that the vector bundle 
$\sF|_C$ is ample since it is semistable with positive slope by our current assumption.
By \cite[Main Theorem]{bogomolov_mcquillan01} (see also \cite[Theorem 1]{kebekus_sola_conde_07}), the leaf $F$ of $\sF$ through a general point in $C$ is algebraic and rationally connected. It follows that $\sF$ is induced by a morphism with connected fibers $f \colon X \to B$
onto a smooth complete curve $B$. By \cite[Theorem IV.2.10]{kollar96} applied to a general fiber $F$ of $f$, there exists a minimal free rational curve $\mathbb{P}^1 \to F$
which means that there exists an integer $0 \le p \le n-2$ such that 
$$\sF|_{\mathbb{P}^1}\cong T_F|_{\mathbb{P}^1}\cong \sO_{\mathbb{P}^1}(2)\oplus\sO_{\mathbb{P}^1}(1)^{\oplus p} \oplus \sO_{\mathbb{P}^1}^{\oplus n-2-p}.$$
Since $\sF|_{\mathbb{P}^1}$ is projectively flat by assumption, we must have $p=0$ and $n=2$, yielding a contradiction. This finishes the proof of the lemma.
\end{proof}

\section{Preparations for the proof of Theorem \ref{thm_intro:main}}\label{section:preparation}

In this section we provide technical tools for the proof of our main results.

\subsection{Bott (partial) connection and applications} We briefly recall the relevant notions first. Let $X$ be a complex manifold, and let $\sF\subset T_X$ be a regular foliation with normal bundle $\sN$. Let $p\colon T_X\to \sN$ denotes the natural projection. For sections $U$ of $\sN$, $T$ of $T_X$, and $V$ of $\sF$ over some open subset of $X$ with $U=p(T)$, set $D_V U=p([V,U])$. This expression is well defined, $\sO_X$-linear in $V$, and satisfies the Leibnitz rule 
$D_V(fU)=fD_V U+(Vf)U$ so that $D$ is an $\sF$-connection on $\sN$
(see \cite{baum_bott70}). We refer to it as the \textit{Bott connection} on $\sN$.

\begin{lemma}\label{lemma:BB_vanishing}
Let $X$ be a complex manifold of dimension $n\ge 2$, and let $\sF\subset T_X$ be a regular foliation of codimension $1 \le q \le n-1$ with normal bundle $\sN$.
\begin{enumerate}
\item The Bott connection is flat when restricted to the leaves of $\sF$.
\item The cohomology class $c_1(\sN)\in H^1(X,\Omega^1_X)$ lies in the image of the natural map $$H^1(X,\sN^*) \to H^1(X,\Omega^1_X).$$ In particular, we have 
$c_1(\sN)^{q+1}=0 \in H^{q+1}(X,\Omega^{q+1}_X)$.
\end{enumerate}
\end{lemma}
\begin{proof}
Item (1) follows from an easy (local) computation. The arguments of \cite[Proof of Corollary 3.4]{baum_bott70} give a proof of Item (2).
\end{proof}

The following easy consequence of Lemma \ref{lemma:BB_vanishing} above will prove to be crucial.

\begin{prop}\label{prop:vanishing_chern_classes}
Let $X$ be a complex manifold of dimension $n\ge 2$ and let $Y\subseteq X$ be a compact K\"ahler submanifold. Let $\sF \subset T_X$ be a regular codimension $1$ foliation with normal bundle $\sN$. Suppose that $\sF\cong \sL_1\oplus\cdots\oplus\sL_{n-1}$ for some line bundles $\sL_i$
with $c_1(\sL_1)=\cdots=c_1(\sL_{n-1})\in H^1(X,\Omega_X^1)$.  
Then $$c_1(\sL_i|_Y)^2=c_1(\sN|_Y)^2=c_1(\sL_i|_Y)\cdot c_1(\sN|_Y)\equiv 0.$$ Moreover, the linear subspace $$L:=\langle c_1(\sL_1|_Y),\ldots,c_1(\sL_{n-1}|_Y), c_1(\sN|_Y)\rangle \subseteq H^1(Y,\Omega^1_Y)$$ has dimension at most one.
\end{prop}

\begin{proof}
The second assertion follows immediately from the first using the Hodge index theorem.

Set $c:=c_1(\sL_1)=\cdots=c_1(\sL_{n-1})\in H^1(X,\Omega_X^1)$. Notice that $\sL_i \subset T_X$ defines a regular foliation for any index $i\in I$. Denote by $\sN_i$ its normal bundle. Then Lemma \ref{lemma:BB_vanishing} shows that
$c_1(\sN_i)=c_1(\sN)+(n-2)c\in H^1(X,\Omega^1_X)$ lies in the image of the natural map
$$H^1(X,\sN^*_i) \to H^1(X,\Omega^1_X).$$ On the other hand, we have
\begin{align*}
\bigcap_{1 \le i \le n-1} \textup{Im}\left (H^1(X,\sN^*_i) \to H^1(X,\Omega^1_X)\right)  
& = \bigcap_{1 \le i \le n-1} \textup{Ker}\left (H^1(X,\Omega^1_X) \to H^1(X,\sL^*_i) \right) \\
& = \textup{Ker}\left ( H^1(X,\Omega^1_X) \to \bigoplus_{1 \le i \le n-1} H^1(X,\sL^*_i) \right) 
\\
& = \textup{Ker}\left ( H^1(X,\Omega^1_X) \to H^1(X,\sF^*) \right) \\
& = \textup{Im}\left (H^1(X,\sN^*) \to H^1(X,\Omega^1_X)\right),\\
\end{align*}
and hence $c_1(\sN)+(n-2)c\in H^1(X,\Omega^1_X)$ lies in the image of the natural map
$H^1(X,\sN^*) \to H^1(X,\Omega^1_X)$. By Lemma \ref{lemma:BB_vanishing} again, $c_1(\sN)\in H^1(X,\Omega^1_X)$ lies in the image of the natural map $$H^1(X,\sN^*) \to H^1(X,\Omega^1_X).$$

Now, we have a commutative diagram
\begin{center}
\begin{tikzcd}[row sep=large, column sep=huge]
H^1(X,\sN^*) \ar[r] \ar[d] & H^1(X,\Omega^1_X) \ar[d]\\
H^1(Y,\sN^*|_Y) \ar[r] & H^1(Y,\Omega^1_Y). 
\end{tikzcd}
\end{center}
The first assertion follows easily since $\sN$ has rank $1$ by assumption.
\end{proof}

The proof of Proposition \ref{prop:trivial_representation_torus_quotient} below makes use of the following result.

\begin{lemma}\label{lemma:conductor}
Let $X$ be a complex projective variety, and let $\nu\colon \wt{X} \to X$ be the normalization morphism. Suppose that $\omega_X$ is invertible. Then the following holds.
\begin{enumerate}
\item There exists an effective Weil divisor on $\wt{X}$ such that $\omega_{\wt{X}}\cong \nu^*\omega_X(-E)$.
\item We have $E=0$ if and only if $X$ is normal.
\end{enumerate}
\end{lemma}

\begin{proof}
By \cite[Proposition 5.67]{kollar_mori} and \cite[Proposition 5.68]{kollar_mori} together, we have $\nu_*\omega_{\wt{X}}\cong \sC \otimes \omega_X$, where $\sC:=\sH om_{\sO_X}(\nu_*\sO_{\wt{X}},\sO_X)\subseteq \sO_X$ is the conductor ideal of $\wt{X}$ in $X$. This gives $\omega_{\wt{X}}\cong \wt{\sC}\otimes\nu^*\omega_X$, where 
$\wt{\sC}:=\sC\cdot \sO_{\wt{X}}\subseteq \sO_{\wt{X}}$. Recall that $\omega_{\wt{X}}$ is reflexive since $\wt{X}$ is normal. It follows that $\wt{\sC}$ is reflexive as well. Hence, there exists an effective Weil divisor on $\wt{X}$ such that $\wt{\sC}\cong \sO_{\wt{X}}(-E)$, proving (1).

If $E=0$, then $\sC\cong \sO_{X}$. This immediately implies that $\nu$ is an isomorphism, finishing the proof of the lemma. 
\end{proof}

\begin{prop}\label{prop:trivial_representation_torus_quotient}
Let $X$ be a complex manifold of dimension $n\ge 3$, and let $Y\subseteq X$ be a compact K\"ahler submanifold of dimension $m \ge 2$. Let $\sF \subset T_X$ be a regular codimension $1$ foliation with normal bundle $\sN$. Suppose that $\sF\cong \sL_1\oplus\cdots\oplus\sL_{n-1}$ for some line bundles $\sL_i$ with $c_1(\sL_1)=\cdots=c_1(\sL_{n-1})\in H^1(X,\Omega_X^1)$. 
Let $\sM$ be a line bundle on $Y$ such that $c_1(\sM)\in \langle c_1(\sL_1|_Y),\ldots,c_1(\sL_{n-1}|_Y), c_1(\sN|_Y)\rangle \subseteq H^1(Y,\Omega^1_Y)$. Suppose that 
$\sM$ is nef and $\sM \equiv \sum_{i\in I}b_i B_i$, where $b_i\in\mathbb{Q}_{>0}$ and the $B_i$ are distinct prime divisors. 
Suppose in addition that $\sN_{Y/X}^*|_{B_i}$ is nef if $m<n$. Then the following holds.
\begin{enumerate}
\item For every index $i\in I$, $B_i$ is a finite \'etale quotient of a complex torus.
\item We have $B_i\cap B_j =\emptyset$ if $i\neq j$.
\item The vector bundles $\sN_{B_i/Y}$ and $\sN_{B_i/X}$ are numerically flat. In particular $B_i\subset Y$ is a nef divisor.
\end{enumerate}
\end{prop}

\begin{proof}
Suppose that $m<n$. If $m=n$, the proof of the proposition is very similar and therefore omitted.

We may assume without loss of generality that $c_1(\sM)\not\equiv 0$. In other words, assume that $I$ is not empty. By Proposition \ref{prop:vanishing_chern_classes}, we have $c_1(\sM)^2 \equiv 0$. It follows that $\sM|_{B_i}\equiv 0$ for any $i \in I$ since 
$\sM$ is nef by assumption. By Proposition \ref{prop:vanishing_chern_classes} again, 
$$L:=\langle c_1(\sL_1|_Y),\ldots,c_1(\sL_{n-1}|_Y), c_1(\sN|_Y)\rangle = \langle c_1(\sM|_Y)\rangle \subseteq H^1(Y,\Omega^1_Y).$$ 
Thus $\sL_j|_{B_i}\equiv 0$ and $\sN|_{B_i}\equiv 0$.
As a consequence, the vector bundle $\Omega_X^1|_{B_i}$ is numerically flat.
Consider the exact sequence
$$0 \to \sN_{Y/X}^*|_{B_i} \to \Omega_X^1|_{B_i} \to \Omega_Y^1|_{B_i}\to 0.$$
Since $\sN_{Y/X}^*|_{B_i}$ is nef by assumption, we infer that $\sN_{Y/X}^*|_{B_i}$ and
$\Omega_Y^1|_{B_i}$ are numerically flat as well.

Next, consider the exact sequence of sheaves
$$0 \to \sO_{B_i}(-B_i) \to \Omega_Y^1|_{B_i} \to \Omega_{B_i}^1 \to 0.$$
Since $-B_i|_{B_i}\equiv \frac{1}{b_i}\sum_{j\neq i}b_j B_j|_{B_i}$, we conclude that $-B_i|_{B_i}\equiv 0$ and that $\sO_{B_i}(-B_i)$ is a subbundle of $\Omega_Y^1|_{B_i}$ (see \cite[Proposition 1.16]{demailly_peternell_schneider94}). As a consequence, the sheaf 
$\Omega_{B_i}^1$ is locally free and numerically flat. Moreover, $B_i \cap B_j =\emptyset$ if $i\neq j$.

Let $\nu_i\colon \wt{B}_i \to B_i$ be the normalization morphism. By the adjunction formula, 
$$\omega_{B_i} \cong \sO_Y(K_Y+B_i)|_{B_i}\equiv 0.$$ 
In particular, $\omega_{B_i}$ is locally free. By Lemma \ref{lemma:conductor} above, there exists an effective Weil divisor $E_i$ on $\wt{B}_i$ such that $\omega_{\wt{B}_i}\cong \nu_i^*\omega_{B_i}(-E_i)$.

Suppose that $E_i \neq 0$. Then $K_{\wt{B}_i}$ is not pseudo-effective, and hence $B_i$ is uniruled by \cite[Corollary 0.3]{bdpp} applied to a resolution of $\wt{B}_i$.
Let $C \subseteq B_i$ be a rational curve. Then $\Omega^1_Y|_C$ is the trivial vector bundle since $\Omega^1_Y|_C$ is numerically flat, yielding a contradiction. This shows that $E_i=0$, and hence $B_i$ is normal by Lemma \ref{lemma:conductor}. Then \cite[Corollary 0.3]{bdpp} applies to show that $B_i$ has canonical singularities since $K_{B_i}\equiv 0$. By the solution of the Zariski-Lipman conjecture for canonical spaces (see \cite[Theorem 6.1]{greb_kebekus_kovacs_peternell10} or \cite[Theorem 1.1]{druel_zl}), we conclude that $B_i$ is smooth. In addition, we have $c_1(B_i)\equiv 0$ and $c_2(B_i)\equiv 0$ by \cite[Corollary 1.19]{demailly_peternell_schneider94}.
As a classical consequence of Yau's theorem on the existence of a K\"ahler-Einstein metric, $B_i$ is then covered by a complex torus (see \cite[Chapter IV Corollary 4.15]{kobayashi_diff_geom_vb}). This finishes the proof of the proposition.
\end{proof}

\begin{prop}\label{prop:vanishing_chern_classes2}
Let $X$ be a complex manifold of dimension $n\ge 3$, and let $Y\subseteq X$ be a compact K\"ahler submanifold of dimension $m\ge 2$. Let $\sF \subset T_X$ be a regular codimension $1$ foliation. Suppose that $\sF\cong \sL_1\oplus\cdots\oplus\sL_{n-1}$ for some line bundles $\sL_i$
with $c_1(\sL_1)=\cdots=c_1(\sL_{n-1})\in H^1(X,\Omega_X^1)$. Suppose furthermore that $c_k(\sN_{Y/X})=0$ for any $k \in\{1,\ldots,m\}$ if $m<n$. 
Then $c_1(Y)^2=0$ and $c_k(Y)=0$ for any $k \in\{2,\ldots,m\}$. 
Moreover, we have $\chi(Y,\sO_Y)=0$.
\end{prop}

\begin{proof}
Suppose that $m<n$. If $m=n$, the proof of the proposition is very similar and therefore omitted.

The Chern polynomial $c_t(T_X|_Y)$ of $T_X|_Y$ satisfies $$c_t(T_X|_Y)=c_t(\sN|_Y)\cdot\prod_{1\le i\le n-1}c_t(\sL_i|_Y),$$ where $\sN$ denotes the normal bundle of $\sF$. Therefore, by Proposition \ref{prop:vanishing_chern_classes}, we have $c_1(T_X|_Y)^2=0$ and 
$c_k(T_X|_Y)=0$ for any $k \in\{2,\ldots,n\}$. On the other hand, 
$$c_t(T_X|_Y)=c_t(\sN_{Y/X})\cdot c_t(T_Y)=c_t(T_Y)$$ since $c_t(\sN_{Y/X})=1$ by assumption.
This easily implies $c_1(Y)^2=0$ and $c_k(Y)=0$ for any $k \in\{2,\ldots,m\}$. Then 
$\chi(Y,\sO_Y)=0$ by the Hirzebruch-Riemann-Roch theorem. This finishes the proof of the proposition. 
\end{proof}

\subsection{Abelian schemes} The proof of Theorem \ref{thm_intro:main} uses Koll\'ar's
characterisation of \'etale quotients of abelian schemes \cite[Theorem 6.3]{kollar_sh_inventiones}. We recall the relevant notion first. Let $X$ be a normal projective variety and let $Y \subseteq X$ be a closed subvariety. We say that X has \textit{generically large fundamental group on $Y$} if for any very general point $y\in Y$ and for every closed and positive-dimensional subvariety $y\in Z \subseteq Y$ with normalization $\wb{Z}$, the image of the natural morphism $\pi_1(\wb{Z}) \to \pi_1(X)$ is infinite. The following observation will prove to be crucial.

\begin{lemma}\label{lemma:injectivity_fundamental_groups}
Let $f \colon X \to Y$ be a projective morphism with connected fibers between smooth algebraic varieties, and let $F$ be a general fiber of $f$. Suppose that every fiber of $f$ $($with its reduced structure$)$ is a finite \'etale quotient of an abelian variety.
Then the homomorphism $\pi_1(F) \to \pi(X)$ is injective. In particular, $X$ has generically large fundamental group on $F$. 
\end{lemma}

\begin{proof}
By \cite[Proof of Lemma 2.2]{hoering_nef_cotan} and \cite[Lemma 7.3]{nakayama_tori},
there exists an open subset $Y^\circ\subseteq Y$ with complement of codimension at least $2$ 
such that the restriction $f^\circ$ of $f$ to $X^\circ:=f^{-1}(Y^\circ)$ is locally homotopically  $Q$-smooth (we refer the reader to \cite[Definition 7.2]{nakayama_tori} for this notion). Then 
\cite[Theorem 7.8]{nakayama_tori} applies to show that $f^\circ$ is birational to a $Q$-smooth
$Q$-torus fibration $g^\circ \colon Z^\circ \to Y^\circ$ over $Y^\circ$. In other words, 
there exists a finite morphism $Y_1^\circ \to Y^\circ$ such that the normalization 
$Z_1^\circ$ of $Y_1^\circ \times_{Y^\circ} Z^\circ$ is \'etale over $Z^\circ$ and such that the morphism $Z_1^\circ \to Y_1^\circ$ is smooth with fibers being finite \'etale quotients of abelian varieties. This immediately implies that any fiber of $g^\circ$ (with its reduced structure) is a finite \'etale quotient of an abelian variety. It then follows that $X^\circ/Y^\circ \cong Z^\circ/Y^\circ$. Let $X_1^\circ \to X^\circ$ be the induced finite \'etale cover and denote by $f_1^\circ\colon X_1^\circ \to Y_1^\circ$ the corresponding 
smooth fibration. By \cite[Theorem 3.14]{nakayama_tori}, there exists a finite \'etale cover $X_2^\circ \to X_1^\circ$ such that any fiber of the composition $f_2^\circ \colon X_2^\circ \to X_1^\circ \to Y_1^\circ=:Y_2^\circ$ is an abelian variety.
Let $F_2$ be a general fiber of $f_2^\circ$. By \cite[Proof of Lemma 2.3]{claudon_invariance}, 
the homomorphism $\pi_1(F_2) \to \pi_1(X_2^\circ)$ is injective. Moreover, there is a commutative diagram
\begin{center}
\begin{tikzcd}[row sep=large, column sep=huge]
  \pi_1(F_2)  \ar[d, hook]\ar[r, hook] & \pi_1(X_2^\circ)\ar[d, hook]\\
  \pi_1(F)\ar[r] & \pi_1(X^\circ),   
\end{tikzcd}
\end{center}
where $F$ denotes the image of $F_2$ in $X^\circ$. By \cite[Corollary 3.7]{nakayama_tori}, 
$\pi_1(F)$ is torsion free. This easily implies that the homomorphism $\pi_1(F) \to \pi(X)$ is also injective since $\pi_1(F_2)$ has finite index in $\pi(F)$. Finally, the inclusion induces an isomorphism $\pi_1(X^\circ)\cong \pi_1(X)$ of fundamental groups since $X\setminus X^\circ$ is a closed subset of codimension at least $2$. This finishes the proof of the lemma.
\end{proof}

We will also need the following auxiliary result.

\begin{lemma}\label{lemma:smooth_deformation_torus_quotient}
Let $f \colon X \to Y$ be a smooth projective morphism with connected fibers between smooth algebraic varieties. Suppose that there exists $y_0 \in Y$ such that the fiber $f^{-1}(y_0)$ of $f$ over $y_0$ is a finite \'etale quotient of an abelian variety. Then every fiber of $f$ is a finite \'etale quotient of an abelian variety. Moreover, there exist a dense open set $Y^\circ\subseteq Y$ and a finite \'etale cover $X_1^\circ \to X^\circ:=f^{-1}(Y^\circ)$ such that the fibration $X_1^\circ \to Y_1^\circ$ obtained as the Stein factorization of the composition $X_1^\circ \to X^\circ \to Y^\circ$ is an abelian scheme equipped with a level three structure.
\end{lemma}

\begin{proof}
By \cite[Theorem 3.14]{nakayama_tori}, there exists a finite \'etale cover $X_1 \to X$ such that the fiber $f_1^{-1}(y_0)$ of the composition $f_1 \colon X_1 \to X \to Y=:Y_1$ is an abelian variety. Notice that $f_1$ is a smooth morphism. By Ehresmann's theorem, any fiber $F_1$ of $f_1$ is homeomorphic to an abelian variety. Then \cite[Theorem 4.8]{catanese_deformation_types} applies to show that $F_1$ is an abelian variety. Finally, shrinking $Y_1$ if necessary, we can find a suitable finite \'etale cover $Y_2\to Y_1$ 
such that the projection morphism $Y_2 \times_{Y_1} X_1 \to Y_2$ has a section.
Then $Y_2 \times_{Y_1} X_1$ is an abelian scheme over $Y_2$. Replacing $Y_2$ by a further finite \'etale cover, we may assume that $Y_2 \times_{Y_1} X_1/Y_2$ admits a level three structure, finishing the proof of the lemma.
\end{proof}

The proof of Theorem \ref{thm_intro:main} makes use of the following result, which might be of independent interest.

\begin{lemma}\label{lemma:contraction_abelian_scheme}
Let $f\colon X \to Y$ be an abelian scheme over a projective base equipped with a level three structure, and let $\beta\colon X \to M$ be a birational morphism onto a smooth projective variety. Then $M$ is an abelian scheme over a smooth projective base $N$ and there exists a birational morphism $Y \to N$ such that $X/Y \cong Y \times_N M/Y$.   
\end{lemma}

\begin{proof}
Let $\Sigma\subseteq X$ be the neutral section of $f$, and let $N$ be the normalization of $\Lambda:=\beta(\Sigma)$. Let also $\sA(3)$ be the fine moduli space of polarized abelian varieties with a level three structure.

Let $C\subseteq X$ be a rational curve such that $\dim \beta(C)=0$, and set $B:=f(C)$. By \cite[Lemma 5.9.3]{kollar_sh_inventiones}, $\big(f^{-1}(B)/B,\Sigma|_{f^{-1}(B)}\big)\cong (B \times A/B,  B\times\{0_A\})$, where $(A,0_A)$ is an abelian variety. Notice that $C=B 
\times \{a_0\}$ for some point $a_0 \in A$. In particular, either $C$ is disjoint from $\Sigma$ or $C$ is contained in $\Sigma$. This immediately implies that $\Sigma=\beta^{-1}(\Lambda)$
since every positive-dimensional fiber of $\beta$ is rationally chain connected. By the rigidity lemma (see \cite[Lemma 1.15]{debarre}), the morphism $Y \to \sA(3)$ corresponding to $f$ factorizes through the map $\eta\colon Y\cong \Sigma \to N$ induced by $\beta|_\Sigma$. In other words, the abelian scheme $f$ is the pull-back of an abelian scheme $M_1$ over $N$ via $\eta$. Let $\beta_1\colon X \to M_1$ be the induced morphism. 

To prove the statement, it suffices to show that there exists an isomorphism $\tau \colon M \to M_1$ such that $\beta_1=\tau\circ\beta$ or equivalently that $\beta$ (resp. $\beta_1$) factorizes through $\beta_1$ (resp. $\beta$).

Let $F$ be a positive-dimensional fiber of $\beta$. Recall that $F$ is rationally chain connected. Set $G:=f(F)$. By \cite[Lemma 5.9.3]{kollar_sh_inventiones} again, $\big(f^{-1}(G)/G,\Sigma|_{f^{-1}(G)}\big)\cong (G \times A_2/G,  G\times\{0_{A_2}\})$, where $(A_2,0_{A_2})$ is an abelian variety. Moreover, $F=G \times \{a_2\}$ for some point $a_2 \in A_2$. Since $G$ is a fiber of $\eta$, $F$ is contracted by $\beta_1$. By the rigidity lemma (see \cite[Lemma 1.15]{debarre}), $\beta_1$ factorizes through $\beta$. Let now $F_1$ be a fiber of $\beta_1$.
Set $G_1:=f(F_1)$. Then $G_1$ is a fiber of $\eta$. Moreover, $\big(f^{-1}(G_1)/G_1,\Sigma|_{f^{-1}(G_1)}\big) \cong (G_1 \times A_1/G_1, G_1 \times \{0_{A_1}\})$ for some abelian variety $(A_1,0_{A_1})$, and $F_1\cong G_1 \times \{a_1\}$ for some point $a_1 \in A_1$. Since $G_1 \times \{0_{A_1}\}\cong \Sigma\cap f^{-1}(G_1)$ is contracted by $\beta$, $F_1$ is contracted by $\beta$ as well. By the rigidity lemma (see \cite[Lemma 1.15]{debarre}), $\beta$ factorizes through $\beta_1$. This finishes the proof of the lemma.
\end{proof}

\subsection{Descent of line bundles} Let $f \colon X \to Y$ be a morphism of algebraic varieties, and let $\sL$ be a line bundle on $X$. We give sufficient conditions that guarantee that there exist a line bundle $\sM$ on $Y$ and a positive integer $m$ such that $\sL^{\otimes m} \cong f^*\sM$. The following observation is rather standard. We include a proof here for the reader's convenience.  

\begin{lemma}\label{lemma:pull_back}
Let $f \colon X \to Y$ be a flat projective morphism with integral fibers between smooth complex algebraic varieties. Let $\sL$ be a line bundle on $X$, and let $F$ be a general fiber of $f$. Suppose that $\sL|_F$ is a torsion line bundle. Then there exists a line bundle $\sM$ on $Y$ and a positive integer $m$ such that
$\sL^{\otimes m} \cong f^*\sM$.
\end{lemma}

\begin{proof}
Let $m$ be a positive integer such that the set of points $y\in Y$ with 
$\sL^{\otimes m}|_{X_y} \cong \sO_{X_y}$ is dense.
Because the functions $y \mapsto h^0\big(X_y,\sL^{\otimes \pm m}|_{X_y}\big)$ are upper semicontinuous in the Zariski topology on $Y$ (see \cite[Theorem 12.8]{hartshorne77}), we have 
$h^0\big(X_y,\sL^{\otimes \pm m}|_{X_y}\big)=1$ for every $y\in Y$. This implies that 
$\sL^{\otimes m}|_{X_y} \cong \sO_{X_y}$ since $X_y$ is integral by assumption. By 
\cite[Corollary 12.9]{hartshorne77} together with the base change theorem \cite[Theorem 12.11]{hartshorne77}, $\sM:=f_*\sL^{\otimes m}$ is a line bundle and the formation of $f_*\sL^{\otimes m}$ commutes with arbitrary base change. This easily implies that $\sL^{\otimes m} \cong f^*\sM$.
\end{proof}

\begin{rem}
Setting and notation as in Lemma \ref{lemma:pull_back}. Suppose in addition that $\sL|_F \cong \sO_F$. Then the proof of Lemma \ref{lemma:pull_back} shows that there exists a line bundle $\sM$ on $Y$ such that $\sL\cong f^*\sM$. 
\end{rem}

\begin{prop}\label{prop:pull-back}
Let $X$ be a complex projective manifold and let $f \colon X \to B$ be a surjective morphism with connected fibers onto a smooth complete curve. Let $\sL$ be a nef line bundle on $X$ such that $\sL|_F$ is torsion, where $F$ is a general fiber of $f$. Then there exist a line bundle $\sM$ on $B$ and a positive integer $m$ such that $\sL^{\otimes m}\cong f^*\sM$. In particular, $\sL\equiv_B 0$.
\end{prop}

\begin{proof}
Let $B^\circ \subseteq B$ be a dense open set such that $f^\circ:=f|_{X^\circ}$ is a smooth morphism, where $X^\circ:=f^{-1}(B^\circ)$. Let $m$ be a positive integer. Because the functions $b \mapsto h^0\big(X_b,\sL^{\otimes \pm m}|_{X_b}\big)$ are upper semicontinuous in the Zariski topology on $B$ (see \cite[Theorem 12.8]{hartshorne77}), the set of points $b$ on $B^\circ$ such that $\sL^{\otimes m}|_{X_b}\cong \sO_{X_b}$ is closed. As a consequence, there exists a positive integer $m_1$ such that $\sL^{\otimes m_1}|_{X_b}\cong \sO_{X_b}$ for all $b \in B^\circ$.
From \cite[Corollary 12.9]{hartshorne77}, we see that $f^\circ_*(\sL^{\otimes m_1}|_{X^\circ})$ is a line bundle. Thus, there exists a divisor $D$ on $X$ with $f(\textup{Supp}\,D)\subsetneq B$
such that $\sL^{\otimes m_1}\cong \sO_X(D)$.

Let $S \subset X$ be a $2$-dimensional complete intersection of general elements of a very ample linear system on $X$. This is a smooth surface in $X$ and the restriction $f|_S$ of $f$ to $S$ is a surjective morphism with connected fibers onto $B$. By Zariski's Lemma, we have $D|_S^2 \le 0$. Since $D$ is nef by assumption, we must have $D|_S^2 = 0$. By Zariski's Lemma again, there exist a positive integer $m_2$ and a divisor $G$ on $B$ such that $m_2 D|_S=f|_S^*G$. The claim now follows from the Grothendieck-Lefchetz hyperplane theorem.
\end{proof}

The following result is an easy consequence of Proposition \ref{prop:pull-back} above.

\begin{cor}\label{cor:numerically trivial}
Let $f\colon X \to Y$ be an equidimensional morphism with connected fibers between normal projective varieties, and let $\sL$ be a nef line bundle on $X$ such that $\sL|_F$ is torsion, where $F$ is a general fiber of $f$. Then $\sL\equiv_Y 0$.
\end{cor}

\begin{proof}
Notice that a general fiber of $f$ is irreducible by \cite[5.5]{debarre}. By a theorem of Chevalley (see \cite[Corollaire 14.4.4]{ega28}, the morphism $f$ is universally open.

Let $y \in Y$ be a point and let $B \to Y$ be a morphism from a smooth complete curve $B$ whose image contains $y$ and a general point in $Y$. Observe that product $B \times_Y X$ is irreducible since the projection $B \times_Y X \to B$ is open with irreducible general fibers. Let $Z$ be a resolution of $(B \times_Y X)_{\textup{red}}$. Our claim follows from Proposition \ref{prop:pull-back} applied to the pull-back of $\sL$ to $Z$. 
\end{proof}

We will also need the following elementary observation.

\begin{lemma}\label{lemma:hodge_index_theorem}
Let $f \colon X \to Y$ be a surjective morphism of normal projective varieties with $n:=\dim X \ge 2$ and $m:=\dim Y  \ge 1$. Let $D$ be a prime $\mathbb{Q}$-Cartier divisor on $X$ such that $\dim f(D)=0$. Suppose in addition that $D^2\cdot H^{n-2} \ge 0$ for some very ample divisor $H$ on $X$. Then $\dim Y = \kappa(D) = 1$, and $f$ is induced by the linear system $|mD|$ for $m$ sufficiently large.
\end{lemma}

\begin{proof}
Let us first show that $\dim Y =1$. We argue by contradiction and assume that $\dim Y \ge 2$. Let $S \subseteq X$ be a $2$-dimensional complete intersection of general elements of $|H|$, and set $C:= S \cap D$. By Bertini's theorem (see \cite[Th\'eor\`eme I.6.10]{jouanolou_bertini}), we may assume without loss of generality that $S$ and $C$ are reduced and irreducible. 
Then $D^2\cdot S = (D|_S)^2 <0$ since $C$ is contracted by the generically finite morphism 
$f|_S$, yielding a contradiction since $D^2\cdot S = D^2\cdot H^{n-2} \ge 0$ by assumption. This shows $\dim Y =1$. 

Zariski's Lemma then applies to show that some non-zero multiple of $D|_S=C$ is a fiber of 
$f|_S$. This easily implies that some non-zero multiple of $D$ is a fiber of $f$, and hence
$\kappa(D) = 1$. This finishes the proof of the lemma.
\end{proof}

\subsection{Abundance}

We will later use the following special case of the abundance
conjecture, whose proof relies on an extension theorem proved in \cite{demailly_hacon_paun}.

\begin{prop}\label{prop:log_abundance}
Let $X$ be a smooth complex projective manifold of dimension $n \ge 2$, and let $B$ be a non-zero divisor on $X$. Let $B:=\sum_{i \in I} B_i$ be its decomposition into irreducible components. Suppose that $B_i$ is smooth with $K_{B_i} \equiv 0$ and $B_i^2 \equiv 0$. Suppose moreover that $B_i \cap B_j = \emptyset$ if $i \neq j$. Suppose finally $\kappa(X)\ge 0$ and $K_X$ nef. Then there exists a morphism with connected fibers $f \colon X \to C$ onto a smooth complete curve such that $B_i$ is a fiber of $f$ for every $i\in I$. In addition, $K_X$ is semiample, $K_X|_F\sim_\mathbb{Q}0$ for a general fiber $F$ of $f$ and 
$\kappa(X)=\nu(K_X)\le 1$. In particular, $K_X+B$ is semiample with $\kappa(K_X+B)=\nu(K_X+B)\le 1$.
\end{prop}

\begin{proof}
Let $D$ be an effective $\mathbb{Q}$-divisor such that $K_X \sim_\mathbb{Q} D$. 
Notice that $D \cdot B_i = K_X\cdot B_i \equiv 0$ since $B_i^2\equiv 0$ and $K_{B_i}\equiv 0$ by assumption. Thus, there is a decomposition $D=D_1+D_2$, where $D_1$ and $D_2$ are effective $\mathbb{Q}$-divisor with 
$\textup{Supp}\, D_1 \cap \textup{Supp}\, D_2 = \emptyset$ and $\textup{Supp}\, D_2 \subseteq \textup{Supp}\, B$.  

Let us show that $D_1$ is nef. We argue by contradiction and assume that there exists a curve $C \subset X$ such that $D_1 \cdot C < 0$. Then $C \subseteq \textup{Supp}\, D_1$  and hence $C \cap \textup{Supp}\, D_2 =\emptyset$. Thus $K_X \cdot D =D \cdot C = D_1 \cdot C <0$, yielding a contradiction. Notice also that $B_i$ is nef since $B_i^2 \equiv 0$ by assumption.

Pick $i_0 \in I$. Then $K_X+B_{i_0}+\varepsilon \big(D_1+\sum_{i \in I\setminus \{i_0\}}B_i\big)$ is nef as well for every $\varepsilon >0$. Let $0<\varepsilon \ll 1$ be a rational number such that 
the pair $\big(X,B_{i_0}+\varepsilon \big(D_1+\sum_{i \in I\setminus \{i_0\}}B_i\big)\big)$ is plt. Then \cite[Corollary 1.8]{demailly_hacon_paun} applies to show that there exists an effective $\mathbb{Q}$-divisor $G_1$ such that 
$K_X+B_{i_0}+\varepsilon \big(D_1+\sum_{i \in I \setminus \{i_0\}}B_i\big)\sim_\mathbb{Q} G_1$ and $B_{i_0} \cap G_1 =\emptyset$ since
$\big(K_X+B_{i_0}+\varepsilon \big(D_1+\sum_{i \in I \setminus \{i_0\}}B_i\big)\big)|_{B_{i_0}}\sim_\mathbb{Q} K_{B_{i_0}}$ is torsion. Set $G_2:=D+B_{i_0}+\varepsilon \big(D_1+\sum_{i \in I \setminus \{i_0\}}B_i\big)$. Then $G_2 \sim_\mathbb{Q} G_1$ and $\mult_{B_{i_0}}\, G_2 \ge 1$.  

Let $m$ be a positive integer such that $mG_1$ and $m G_2$ are $\mathbb{Z}$-divisors, and let $X \map \mathbb{P}^1$ be the corresponding rational map. Let $p\colon Y \to X$ and $q\colon Y \to \mathbb{P}^1$ be a common resolution of singularities. We may assume without loss of generality that $p$ restricts to an isomorphism over some Zariski open neighborhood of $B_{i_0}$ since $B_{i_0} \cap G_1 =\emptyset$ by construction. Let $g \colon Y \to C$ be the Stein factorization of $q$. By Zariski's Lemma, some positive multiple of the strict transform of $B_{i_0}$ in $Y$ is a fiber of $g$ since $B_{i_0}^2\equiv 0$. This immediately implies that the rational map $f \colon X \map C$ is a morphism. By Zariski's Lemma again, for every $i\in I$, some positive multiple of $B_i$ is a fiber of $f$. Let now $F$ be a general fiber of $f$. Then, by construction, $\big(K_X+B_{i_0}+\varepsilon \big(D_1+\sum_{i \in I \setminus \{i_0\}}B_i\big)\big)|_F = K_X|_F$ is torsion. But $K_X|_F \sim_\mathbb{Z} K_F$ by the adjunction formula. The statement easily follows from Proposition \ref{prop:pull-back}.
\end{proof}

\subsection{Foliations and rational curves} In order to study foliations on non-minimal varieties it is useful to understand their behavior with respect to rational curves.
The following result is a minor generalization of \cite[Lemma 2.18]{figueredo_rcc}. 

\begin{lemma}\label{lemma:rational_surface}
Let $X$ be a complex projective manifold of dimension $n\ge 2$, and let $\sF \subset T_X$ be a regular codimension $1$ foliation. Let $S$ be a rational ruled surface, and let $\gamma\colon S \to X$ be a generically finite morphism. Then there exists a rational curve $C \subset S$ which is not contracted by $\gamma$ such that $\gamma(C)$ is tangent to $\sF$.
\end{lemma}

\begin{proof}
Let $\sN$ denote the normal bundle of $\sF$, and let $f\colon S \to \mathbb{P}^1$ denote the natural morphism. Recall that $\NS(S)\cong \mathbb{Z}[C_0]\oplus \mathbb{Z}[F]$, where $F$ is any fiber of $f$ and $C_0$ is a section of $f$ with $C_0^2=-e\le 0$.

Set $\sN_S:=\sN|_S$. By Lemma \ref{lemma:BB_vanishing}, we have $c_1(\sN_S)^2 = 0$. 
Write $c_1(\sN_S)\equiv a C_0 + b F$. Then 
$$0=c_1(\sN_S)^2=(a C_0 + b F)^2=a(-ae+2b),$$ 
$$c_1(\sN_S)\cdot F=a$$ and $$c_1(\sN_S)\cdot C_0=-ae+b.$$

We may obviously assume that $F$ is not tangent to $\sF$. Then the composition
$T_{\mathbb{P}^1}\cong T_F \to T_X|_F \to \sN|_F=\sN_S|_F$
is non-zero, and hence $a=c_1(\sN_S)\cdot F\ge 2>0$. It follows that $2b=ae$. In particular, we have $b \ge 0$.

If $C_0$ is contracted by $\gamma$, then $c_1(\sN_S)\cdot C_0=0$. This immediately implies $b=e=0$. Then $S \cong \mathbb{P}^1 \times \mathbb{P}^1$ and $f$ identifies with the projection onto the first factor. Moreover, the image of a general fiber of the projection onto the second factor is tangent to $\sF$.

Suppose from now on that $C_0$ is not contracted by $\gamma$ and that $\gamma(C_0)$ is not tangent to $\sF$. Then, as above, we must have $c_1(\sN_S) \cdot C_0 \ge 2$. It follows that $b \ge ae+2$, and hence $b+2 \le 0$, yielding a contradiction. This finishes the proof of the lemma.
\end{proof}

We will also need the following auxiliary results.

\begin{lemma}\label{lemma:extremal_contraction_fol_transverse}
Let $X$ be a complex projective manifold of dimension $n\ge 2$, and let $\sF \subset T_X$ be a regular codimension $1$ foliation. Let $\phi\colon X \to Y$ be an elementary Fano-Mori contraction. Suppose that there exists an irreducible component $F_1$ of a positive-dimensional fiber $F$ of $\phi$ which is not tangent to $\sF$. Then $Y$ is smooth, and either $\phi$ is the blow-up of a smooth codimension $2$ submanifold or $\phi$ is a conic bundle.
\end{lemma}

\begin{proof}
Let $\sN$ denote the normal bundle of $\sF$, and let $C \subseteq F$ be a rational curve. Suppose that $\sN\cdot C=0$. Then $C$ is tangent to $\sF$ and moreover $\sN|_F\equiv 0$ since $\phi$ is elementary by assumption. This in turn implies that any irreducible component of $F$ is tangent to $\sF$ since $F$ is rationally chain connected by \cite[Corollary 1.4]{hacon_mckernan}, yielding a contradiction. This shows that $\sN\cdot C \neq 0$. On the other hand, $c_1(\sN)^2\equiv 0$ by Lemma \ref{lemma:BB_vanishing}. This immediately implies that any positive-dimensional fiber of $\phi$ is $1$-dimensional. The claim now follows from \cite[Theorem 1.2]{wisn_crelle}. 
\end{proof}

\begin{lemma}\label{lemma:rational_curve_uniruled}
Let $X$ be a complex projective manifold of dimension $n\ge 2$, and let $\sF$ be a regular codimension $1$ foliation. Suppose that $\sF$ is projectively flat, and let $C \subset X$ be a rational curve such that $K_\sF \cdot C \le 0$. Then $T_X|_C$ is nef. In particular, $X$ is uniruled.
\end{lemma}

\begin{proof}
Let $\sN$ denote the normal bundle of $\sF$, and let $\mathbb{P}^1 \to C \subset X$ be the normalization morphism. Since $\sF$ is projectively flat, there exists $a \in \mathbb{Z}$ such that $\sF|_{\mathbb{P}^1}\cong \sO_{\mathbb{P}^1}(a)^{\oplus n-1}$. Notice that $a \ge 0$ since $K_\sF \cdot C \le 0$ by assumption. 

If $C$ is tangent to $\sF$, then $\sN|_{\mathbb{P}^1}$ is a flat line bundle by Lemma \ref{lemma:BB_vanishing}, and hence $\sN|_{\mathbb{P}^1}\cong \sO_{\mathbb{P}^1}$. If  
$C$ is not tangent to $\sF$, then the composition
$T_{\mathbb{P}^1} \to T_X|_{\mathbb{P}^1} \to \sN|_{\mathbb{P}^1}$
is non-zero, and hence $\sN \cdot C \ge 2$. In either case, the exact sequence
$$0 \to \sF|_{\mathbb{P}^1} \to T_X|_{\mathbb{P}^1} \to \sN|_{\mathbb{P}^1} \to 0$$
shows that $T_X|_{\mathbb{P}^1}$ is nef. But then $X$ is uniruled by \cite[Theorem IV.1.9]{kollar96}. 
\end{proof}

\subsection{Special cases} In this subsection we treat a few easy special cases of Theorem \ref{thm_intro:main}. 

\begin{prop}\label{prop:classification_zero_canonical_class}
Setting and notation as in \ref{setup:main}. Suppose in addition that $K_\sF\equiv 0$. Then one of the following holds.
\begin{enumerate}
\item There exists a $\mathbb{P}^1$-bundle structure $\phi\colon X \to Y$ onto a finite \'etale quotient of an abelian variety, and $\sF$ induces a flat Ehresmann connection on $\phi$.
\item There exists an abelian variety $A$ and a finite \'etale cover $\gamma\colon A \to X$ such that $\gamma^{-1}\sF$ is a linear foliation on $A$.
\item There exists a smooth complete curve $B$ of genus at least $2$ as well as an abelian variety $A$, and a finite \'etale cover $\gamma \colon B \times A \to X$ such that $\gamma^{-1}\sF$ is
induced by the projection morphism $B \times A \to B$.
\end{enumerate}  
\end{prop}

\begin{proof} This follows easily from \cite[Th\'eor\`eme 1.2]{touzet} together with \cite[Lemma 5.9]{bobo}.
\end{proof}

\begin{rem}
In the setup of Proposition \ref{prop:classification_zero_canonical_class}, there exists a finite \'etale cover $\eta\colon Z \to X$ such that $\eta^{-1}\sF \cong \sO_Z^{\oplus n-1}$. 
\end{rem}

\begin{prop}\label{prop:special_case_algebraically_integrable}
Setting and notation as in \ref{setup:main}. Suppose in addition that $\sF$ is algebraically integrable. Then one of the following holds.
\begin{enumerate}
\item There exists an abelian variety $A$ as well as a smooth complete curve $C$ of genus at most $1$ and a finite \'etale cover $\gamma\colon  A \times C \to X$ such that $\gamma^{-1}\sF$ is induced by the projection $A \times C \to C$.
\item There exists a smooth abelian scheme $f\colon B \to C$ over a smooth complete curve $C$ of genus at least $2$ as well as a finite \'etale cover $\gamma\colon  B \to X$ such that $\gamma^{-1}\sF$ is induced by $f$.
\end{enumerate}
In particular, $K_\sF$ is semiample with $\kappa(\sF)=\nu(\sF)\in\{0,1\}$.
\end{prop}

\begin{proof}By assumption, $\sF$ is induced by a morphism with irreducible fibers $f\colon X \to C$ onto a smooth complete curve $C$. By \cite[Theorem 0.1]{jahnke_radloff_13}, any leaf of $\sF$ is a finite \'etale quotient of an abelian variety. By Lemma \ref{lemma:injectivity_fundamental_groups} and \cite[Theorem 6.3]{kollar_sh_inventiones} together, after replacing $X$ by a finite \'etale cover, the fibration $f$ is birational to an abelian scheme over $C$. This easily implies that $f$ is an abelian scheme.
Replacing $X$ by a further finite \'etale cover, we may also assume that $f$ is polarized with a level three structure. Let $\sA(3)$ be the corresponding fine moduli space of polarized abelian varieties with a level three structure, and let $C \to \sA(3)$ be the moduli map. By \cite[Lemma 5.9.3]{kollar96}, either the moduli map is the constant morphism or $C$ has genus at least $2$. The lemma now follows easily.
\end{proof}

\begin{rem}
In the setup of Proposition \ref{prop:special_case_algebraically_integrable} (2), we have
$\Omega^1_{B/C} \cong f^*\sE$ with $\sE:=f_*\Omega^1_{B/C}$. Then $\sF$ is numerically projectively flat if and only if $\sE$ is semistable by Lemma \ref{lemma:numerically_projectively_flat_elementary properties}.
\end{rem}

\begin{lemma}\label{lemma:abelian_variety}
Setting and notation as in \ref{setup:main}. Suppose in addition that $X$ is a finite \'etale quotient of an abelian variety. Then $K_\sF \equiv 0$.
\end{lemma}

\begin{proof}
Replacing $X$ by a finite \'etale cover, if necessary, we may assume without loss of generality that $X$ is an abelian variety. Then $\sN$ is generated by its global section. On the other hand, we have $\sN^2\equiv 0$ by Lemma \ref{lemma:BB_vanishing}. If $\sN\cong \sO_X$, then $K_\sF\equiv 0$ by the adjunction formula. Suppose that $\sN \not\equiv 0$. Then, after a further finite \'etale cover, we have $X \cong A \times E$, where $A$ is an abelian variety of dimension $n-1$ and $E$ is an elliptic curve by Poincar\'e's complete reducibility theorem. Moreover, $\sN\cong f^*\sM$ for some line bundle $\sM$ on $E$, where $f\colon X \to E$ denotes the morphism induced projection onto $E$. 
Then Lemma \ref{lemma:product} below applies to show that $K_\sF \equiv 0$ since $n \ge 3$ by assumption. This implies $\sN\equiv 0$, yielding a contradiction. This finishes the proof of the lemma.
\end{proof}

\begin{lemma}\label{lemma:product}
Let $C$ be a smooth complete curve, and let $A$ be an abelian variety of dimension $n-1 \ge1$. Set $X:=C \times A$ and $f:=\textup{pr}_C$. Let $\sF$ be a codimension $1$ foliation on $X$ with normal bundle $\sN\cong f^*\sL$ for some line bundle $\sL$ on $C$. Then $\sF$ is regular and
one of the following conditions holds.
\begin{enumerate}
\item We have $\sF\cong \sO_X^{\oplus n-1}$.
\item We have $\sF\cong \sO_X^{\oplus n-2}\oplus f^*(T_C \otimes \sL^*)$.
\item We have $\sL\cong\sO_C$ and there is an exact sequence
$$0 \to \sO^{\oplus n-2} \to \sF \to f^*T_C\to 0.$$
\end{enumerate}
\end{lemma}

\begin{proof}
Let $\omega\in H^0(X,\Omega_X^1\otimes\sN)$ be a twisted $1$-form defining $\sF$.
Let $F \cong A$ be a fiber of $f$. Then $T_X|_F\cong \sO_F^{\oplus n}$ and $\sN|_F \cong \sO_F$. This immediately implies that $\sF|_F\cong \sO_F^{\oplus n-1}$ and that $\sF$ is regular since
the zero set of $\omega$ has codimension at least $2$ in $X$. Moreover, there exists a vector bundle $\sE$ on $C$ such that $\sF \cong f^*\sE$.

We may obviously assume that $\sF\neq T_{X/C}$. If $F$ is a general fiber of $f$, then $\sF\cap T_{X/C}|_F \subset T_F$ is a linear codimension $1$ foliation on $F \cong A$. Hence there exists a subbundle $\sK \subset \sE$ of rank $n-2$ such that $\sF\cap T_{X/C} = f^*\sK \subset \sF$. 
Set $g:=\textup{pr}_A$. Write $\omega=\alpha + \beta \in H^0(X,\Omega_X^1\otimes\sN) \cong H^0(X,f^*\omega_C\otimes \sN) \oplus H^0(X,g^*\Omega_A^1\otimes \sN)$ 
with
$\alpha \in H^0(X,f^*\omega_C\otimes \sN)$ and $\beta \in H^0(X,g^*\Omega_A^1\otimes \sN)\cong H^0(A,\Omega_A^1) \otimes H^0(C,\sL)$. By construction, there is an effective divisor $D_1$ on $C$ and an exact sequence
$$0 \to \sK \to H^0(A,\Omega_A^1)^*\otimes \sO_C \to \sL(-D_1) \to 0,$$
where the map $H^0(A,\Omega_A^1)^*\otimes \sO_C \to \sL(-D_1)$ is induced by $\beta$.
Since $\sF\neq T_{X/C}$, the natural map $\sO_X^{\oplus n-1}\cong T_{X/C} \to \sN$ is non-zero, and hence $h^0(X,\sL) \ge 1$. Let $D_2$ be any effective divisor on $C$ such that $\sL \cong \sO_C(D_2)$. We view $\omega$, $\alpha$ and $\beta$ as rational $1$-forms on $X$ with poles along the divisor $f^*D_2$. We have $\textup{d}\alpha=0$ and $\textup{d}\beta \wedge \alpha=0$.
But then $\textup{d}\beta \wedge \beta=\textup{d}\omega\wedge \omega=0$.
Let $\gamma_1,\ldots,\gamma_{n-1}$ be a basis of $H^0(A,\Omega_A^1)$ and let $U \subset C\setminus \textup{Supp}\, D_2$ be an open affine subset. Then
$$\beta|_{U \times A}=\sum_{1 \le i \le n-1} t_i \gamma_i$$
for some regular functions $t_i$ on $U$. Moreover, $\textup{d}\beta \wedge \beta=0$ if and only if $\frac{\textup{d}t_i}{t_i}=\frac{\textup{d}t_j}{t_j}$ for all indices $i$ and $j$. 
Suppose that $t_j\neq 0$. Then $t_i=c_{ij}t_j$ for some $c_{ij}\in \mathbb{C}$.
This implies that there exists $s \in H^0(Y,\sL)$ such that $\beta = \gamma \otimes s$, where $\gamma:= \sum_{1 \le i \le n-1} c_{ij} \gamma_i \in H^0(A,\Omega_A^1)$. Then $\div \, s = D_1$ and $\sK\cong \sO_C^{\oplus n-2}$.

Set $\sQ:=\sE/\sK$. Then $\sQ \cong T_C\otimes\sL^*$. If $h^0(C,\sL^*)=0$, then 
$h^1(C,\sQ^*)=h^1(C,\omega_C\otimes \sL)=h^0(C,\sL^*)=0$ by Serre duality, and hence the exact sequence
$$0 \to \sK \to \sE \to \sQ \to 0$$
splits. Then $\sF\cong \sO_X^{\oplus n-2}\oplus f^*(T_C \otimes \sL^*)$.

Suppose finally that $h^0(C,\sL^*) \ge 1$. Then $\sL\cong\sO_C$ since we also have $h^0(C,\sL) \ge 1$. Then $\sQ\cong T_C$, finishing the proof of the lemma.
\end{proof}

\section{Normal bundles of regular foliations}\label{section:normal_bundle}

In this section we prove Theorem \ref{thm_intro:divisor_zero_square} and give a few applications to codimension $1$ foliations on varieties of general type. If $C$ is a smooth complete curve, then $\pm K_C$ is pseudo-effective. Theorem \ref{thm_intro:divisor_zero_square} can be seen as the analogous statement for the canonical divisor of the space of leaves of a codimension $1$ regular foliation. 

\medskip

We will need the following Bertini-type result.

\begin{lemma}\label{lemma:bertini}
Let $X$ be a projective variety of dimension $n \ge 2$, and let $Y \subset X$ be a subvariety of codimension $c \ge 2$. Let $1\le k \le c-1$ be an integer, and let $H_1,\ldots,H_k$ be very ample divisors on $X$. Suppose that $X$ is Cohen-Macaulay. Then there exist large enough positive integers $m_1,\ldots,m_k$ and dense open subsets $V_i \subseteq  |\sI_Y(m_iH_i)|$ such that for any tuple 
$(A_1,\ldots,A_k)\in V_1\times \cdots \times V_k$, the associated complete intersection 
$S:=A_1\cap \cdots \cap A_k$ is reduced and irreducible. In addition, $S$ is Cohen-Macaulay.  
\end{lemma}

\begin{proof}
To prove the statement, we argue by induction on $k$.

Suppose $k=1$. Let $\beta \colon Z \to X$ be the blow-up of $X$ along $Y$, and let $E$ be the effective divisor on $Z$ corresponding to the invertible sheaf $\beta^{-1}\sI_Y\cdot\sO_Z$. Notice that $Z$ is reduced and irreducible. Let $m_1$ be a positive integer such that $\sI_Y(m_1H_1)$ is generated by its global sections. Then the linear system $|m_1\beta^*H_1-E|$ induces a morphism $\phi_1\colon Z \to \mathbb{P}^{N_1}$, where $N_1:=\dim |\sI_Y(m_1H_1)|$. We may also assume without loss of generality that $\dim \phi_1(Z) \ge 2$. By Bertini's theorem (see \cite[Th\'eor\`eme I.6.10]{jouanolou_bertini}), a general member of the linear system $|m_1\beta^*H_1-E|$ is reduced and irreducible. This immediately implies that a general member $A_1$ of $|\sI_Y(m_1H_1)|$ is generically reduced and irreducible. Since $A_1$ is Cohen-Macaulay (see \cite[Proposition 18.13]{eisenbud}), we see that $A_1$ is reduced by Serre's criterion. 

Suppose $k\ge 2$. Then the induction hypothesis applied to $A_1$ and $H_2|_{A_1},\ldots,H_k|_{A_1}$ show that there exist positive integers $m_2,\ldots,m_k$ and dense open subsets $V_i \subseteq  |\sI_Y(m_iH_i|_{A_1})|$ such that for any tuple 
$(B_2,\ldots,B_k)\in V_1\times \cdots \times V_k$, the associated complete intersection 
$S:=B_2\cap \cdots \cap B_k$ is reduced and irreducible. In addition, $S$ is Cohen-Macaulay.
We may also assume without loss of generality that 
$h^1(X,\sO_X(m_iH_i-m_1H_1))=0$ for any $i \ge 2$. This implies
that $|\sI_Y(m_iH_i)|_{A_1}=|\sI_Y(m_iH_i|_{A_1})|$. The claim now follows easily.
\end{proof}

The proof of Theorem \ref{thm:divisor_zero_square} below makes use of the following elementary results.

\begin{lemma}\label{lemma:divisor_zero_square_1}
Let $S$ be a smooth projective surface, and let $D$ be a divisor on $S$ such that $D^2=0$. Then $\pm D$ is pseudo-effective.
\end{lemma}

\begin{proof}
Let $H$ be an ample divisor on $S$. Replacing $D$ by $-D$ if necessary, we may assume without loss of generality that $D\cdot H \ge 0$.
Let $\varepsilon>0$ be a rational number, and set
$D_\varepsilon:=D+\epsilon H$. Then 
$$D_\varepsilon\cdot H = D\cdot H + \varepsilon H^2 >0 \quad\text{and}\quad D_\varepsilon^2=\varepsilon(2D\cdot H + \varepsilon H^2)>0.$$ Applying \cite[Corollary V.1.8]{hartshorne77}, we see that $D_\varepsilon$ is $\mathbb{Q}$-effective. This shows that $D$ is pseudo-effective, proving the lemma.
\end{proof}

\begin{lemma}\label{lemma:divisor_zero_square_2}
Let $S$ be a smooth projective surface, and let $D$ be a pseudo-effective divisor on $S$ with $D^2=0$. Suppose in addition that $\kappa(D) \le 1$. Then $D$ is nef.
\end{lemma}

\begin{proof}
Let $D=P+N$ be the Zariski decomposition of $D$. Then $P$ is a nef $\mathbb{Q}$-divisor while $N$ is an effective $\mathbb{Q}$-divisor. In addition, $P\cdot N=0$ and, either $N=0$ or the intersection matrix $(N_i\cdot N_j)_{i,j \in I}$ is negative definite, where the $N_i$ with $i \in I$ are the irreducible components of $N$. Now $\kappa(P)\le \kappa(D) \le 1$ by assumption, and hence $P^2=0$. It follows that  $0=D^2=N^2$, and hence $N=0$, proving the lemma.
\end{proof}

\begin{thm}\label{thm:divisor_zero_square}
Let $X$ be a complex projective manifold of dimension $n \ge 1$, and let $D$ be a divisor on $X$ such that $D^2\equiv 0$. Suppose in addition that $D$ is not pseudo-effective and that 
$h^0(X,\Omega^1_X(D))\ge 1$. Then $-D$ is nef.
\end{thm}

\begin{proof}
If $n=1$, then $\deg D \le 0$ since $D$ is not pseudo-effective by assumption, and hence $-D$ is $\mathbb{Q}$-effective. Suppose from now on that $n\ge 2$.
By \cite[Theorem 1.5]{bdpp}, there exists a smooth projective variety $Z$, as well as a birational morphism $\beta\colon Z \to X$, and very ample divisors $H_1,\ldots,H_{n-1}$ on $Z$ such that 
$$\beta^*D\cdot H_1\cdots H_{n-1}<0.$$ 

Let $C \subset Z$ be a curve. By Lemma \ref{lemma:bertini} 
there exist positive integers $m_1,\ldots,m_{n-2}$ and a tuple 
$(A_1,\ldots,A_{n-2})\in |\sI_C(m_1H_1)|\times \cdots \times|\sI_C(m_{n-2}H_{n-2})|$
such that the surface $S:=A_1\cap \cdots \cap A_{n-2}$ is reduced and irreducible.

We have $\beta^*D|_S\cdot H_{n-1}|_S<0$ and $\beta^*D|^2_S=0$. Let $S_1 \to S$ be a resolution of singularities. Applying Lemma \ref{lemma:divisor_zero_square_1} to the pull-back 
$\beta^*D|_{S_1}$ of $\beta^*D|_S$ on $S_1$, we conclude that $-\beta^*D|_{S_1}$ is pseudo-effective. 

Let $\omega \in H^0(X,\Omega^1_X(D))\setminus \{0\}$. The twisted $1$-form $\omega$ yields an injective map $\sO_X(-D) \to \Omega^1_X$ of sheaves. We may assume without loss of generality that the composition 
$\sO_X(-D)|_{S_1} \to \Omega^1_X|_{S_1} \to \Omega^1_{S_1}$ is non-zero. By the Bogomolov-Sommese vanishing theorem, we then have $\kappa(\sO_X(-D)|_{S_1})\le 1$. Applying Lemma \ref{lemma:divisor_zero_square_2}, we conclude that $-\beta^*D|_{S_1}$ is nef. In particular,
$-\beta^* D \cdot C \ge 0$. This shows that $-D$ is nef, completing the proof of the proposition.
\end{proof}

\begin{proof}[Proof of Theorem \ref{thm_intro:divisor_zero_square}]This is an immediate consequence of Theorem \ref{thm:divisor_zero_square} since $\sN^2\equiv 0$ by Lemma \ref{lemma:BB_vanishing}.
\end{proof}

The same argument used in the proof of Theorem \ref{thm:divisor_zero_square} above shows that the following holds.

\begin{thm}
Let $X$ be a complex projective manifold of dimension $n \ge 2$, and let $\sL\subseteq \Omega_X^1$ be a locally free subsheaf of rank $1$. Suppose $\sL\cdot H^{n-1}=0$
and $\sL^2\cdot H^{n-2}=0$ for some ample divisor $H$ on $X$. Then $\sL$ is nef.
\end{thm}

The following is an immediate consequence of Theorem \ref{thm:divisor_zero_square}. Note that Touzet described the structure of codimension $1$ foliations with pseudo-effective conormal bundle in \cite{touzet_conpsef}. 

\begin{prop}\label{prop:ambient_general_type_conormal_nef}
Let $X$ be a smooth projective variety of general type, and let $\sF$ be a codimension $1$ foliation on $X$ with normal bundle $\sN$. Suppose $\sN^2\equiv 0$ and $\nu(K_\sF)\le \dim X - 1$. Then $\sN^*$ is nef.
\end{prop}

\begin{proof}
If $\sN$ is pseudo-effective, then $\nu(K_\sF)=\dim X$ since $K_\sF=K_X+c_1(\sN)$
by the adjunction formula and $\nu(K_X)=\dim X$ by assumption, yielding a contradiction. 
The claim now follows from Theorem \ref{thm:divisor_zero_square}.
\end{proof}

The following result extends \cite[Lemme 9]{brunella_aens} to arbitrary regular codimension $1$ foliations on complex projective varieties of general type.

\begin{cor}\label{cor:ambient_general_type_conormal_nef}
Let $X$ be a smooth projective variety of general type, and let $\sF$ be a regular codimension $1$ foliation on $X$ with normal bundle $\sN$. Suppose $\nu(K_\sF)\le \dim X - 1$. Then $\sN^*$ is nef.
\end{cor}

\begin{proof}
By Lemma \ref{lemma:BB_vanishing}, we have $\sN^2\equiv 0$. The statement follows from
Proposition \ref{prop:ambient_general_type_conormal_nef} above.
\end{proof}

\section{Projectively flat foliations on non-minimal varieties}\label{section:non_minimal}

In this section, we address regular codimension $1$ foliations with numerically projectively flat tangent bundle on non-minimal complex projective manifolds. The following is the main result of this section. 

\begin{thm}\label{thm:suspension_complex_torus}
Setting and notation as in \ref{setup:main}. Suppose in addition that $n \ge 4$, and that $K_X$ is not nef. Then there exists a $\mathbb{P}^1$-bundle structure $\phi\colon X \to Y$ onto a finite \'etale quotient of an abelian variety, and $\sF$ induces a flat Ehresmann connection on $\phi$.
\end{thm}

\begin{proof}
For the reader's convenience, the proof is subdivided into a number of steps. 

\medskip

\noindent\textit{Step 1.} We will need the following easy observation.

\begin{claim}\label{claim:step_1}
Any rational curve on $X$ is generically transverse to $\sF$.
\end{claim}

\begin{proof}Let $C$ be a rational curve on $X$, and let 
$\mathbb{P}^1\to C$ be the normalization morphism. Suppose that $T_{\mathbb{P}^1} \subset \sF|_{\mathbb{P}^1}$. Then $\sF|_{\mathbb{P}^1}$ is an ample vector bundle since $\sF$ is numerically projectively flat by assumption. But this contradicts Lemma \ref{lemma:minimal}, proving the claim. 
\end{proof}

Let $\phi\colon X \to Y$ be an elementary Fano-Mori contraction. By Lemma \ref{lemma:extremal_contraction_fol_transverse} and Claim \ref{claim:step_1} together, 
$Y$ is a smooth variety, and either $\phi$ is the blow-up of a smooth codimension $2$ submanifold or $\phi$ is a conic bundle. The proof of Theorem \ref{thm:suspension_complex_torus} uses the following notation.

\begin{setup}
Suppose that $\phi$ is a singular conic bundle and let $\Delta:=\{y \in Y \, | \, \phi^{-1}(y)\textup{ is not smooth}\}$ be the discriminant locus of $\phi$. The same argument used in the proof of \cite[Proposition 1.2]{beauville_prym} shows that $\Delta$ has pure codimension $1$. Moreover, $\Delta$ has normal crossings in codimension $1$. Let $Z_{-2} \subseteq \Delta$ be an irreducible component, and let $E_{-1}$ be the normalization of $E_{-2}:=\phi^{-1}(Z_{-2})$. Then there exists a normal projective variety $Z_{-1}$ as well as morphisms $E_{-1} \to Z_{-1}$ and $Z_{-1} \to Z_{-2}$, and a commutative diagram
\begin{center}
\begin{tikzcd}[row sep=large, column sep=large]
  E_{-1}  \ar[d, "{p_{-1}}"']\ar[r, "{\nu}"] & E_{-2} \ar[d, "{p_{-2}}"]\\
  Z_{-1}\ar[r, "\iota"] & Z_{-2}.   
\end{tikzcd}
\end{center}
In addition, $p_{-1}$ is a $\mathbb{P}^1$-bundle and $\iota$ is a finite morphism of degree $2$. By \cite[Lemme 1.5.2 ]{beauville_prym}, there exists an open set 
$Z_{-2}^\circ\subseteq Z_{-2}$ with complement of codimension at least $2$ such that 
$Z_{-1}^\circ:=\iota^{-1}(Z_{-2}^\circ)$ is smooth and such that the morphisms $E_{-1}^\circ \to X$ and
$Z_{-1}^\circ \to Y$ induced by $\nu|_{E_{-1}^\circ}$ and $\iota|_{Z_{-1}^\circ}$ are unramified, where $E_{-1}^\circ:=p_{-1}^{-1}(Z_{-1}^\circ)$. Let $Z \to Z_{-1}$ be a quasi-\'etale cover with $Z$ normal and projective, and let $E:=Z \times_{Z_{-1}} E_{-1}$ with natural morphism $p\colon E \to Z$. Let $Z^\circ$ be the inverse image of $Z_{-1}^\circ$ and set $E^\circ:=p^{-1}(Z^\circ)$ and $p^\circ:=p|_{E^\circ}$.
By the Nagata-Zariski purity theorem, the morphism $Z^\circ \to Z_{-1}^\circ$ is \'etale. As a consequence, the natural morphisms $E^\circ \to X$ and $Z^\circ \to Y$ are unramified.

Suppose now that $\phi$ is a $\mathbb{P}^1$-bundle. Set $Z_{-1}:=Y$. Let $Z \to Z_{-1}$ be a finite \'etale cover and set $E:=Z\times_{Z_{-1}} E_{-1}$ with natural morphism $p\colon E \to Z$. Set also $E^\circ:=E$ and $Z^\circ:=Z$.

Suppose finally that $\phi$ is the blow-up of a smooth subvariety $Z_{-1}$ of codimension $2$ in $Y$ with exceptional locus $E_{-1}$. Let $Z \to Z_{-1}$ be a finite \'etale cover, and set $E:=Z\times_{Z_{-1}} E_{-1}$ with natural morphism $p\colon E \to Z$. Set also
$E^\circ:=E$ and $Z^\circ:=Z$.  

We obtain a diagram

\begin{center}
\begin{tikzcd}[row sep=large, column sep=large]
  E^\circ  \ar[d, "{p^\circ}"']\ar[r, hook]\ar[rr, bend left, "{\textup{unramified}}"] 
  & E \ar[d, "{p,\,\mathbb{P}^1\textup{-bundle}}"] 
  \ar[r] & X \ar[d, "{\phi}"] \\
  Z^\circ \ar[r, hook] & Z \ar[r] & Y.   
\end{tikzcd}
\end{center}
By choice of $E^\circ$, $T_{E^\circ}$ is a subbundle of $T_X|_{E^\circ}$.

\end{setup}

\noindent\textit{Step 2.} By Claim \ref{claim:step_1}, the composition $T_{E/Z} \to T_X|_{E} \to \sN|_{E}$ is non-zero.

\begin{claim}
To prove Theorem \ref{thm:suspension_complex_torus}, it suffices to show that the map $T_{E/Z} \to \sN|_{E}$ is an isomorphism. 
\end{claim}

\begin{proof}
Suppose that the map $T_{E/Z} \to \sN|_{E}$ is an isomorphism, and let $F$ be a fiber of $p$. Since $\sF$ is projectively flat, there exists $a \in \mathbb{Z}$ such that $\sF|_{\mathbb{P}^1}\cong \sO_{\mathbb{P}^1}(a)^{\oplus n-1}$. In particular, we have 
$K_\sF \cdot F=(n-1)a$.
If $\phi$ is a singular conic bundle or a blow-up, then $K_X\cdot F = -1$ by choice of $F$. The adjunction formula $K_\sF=K_X+c_1(\sN)$
then gives 
$$(n-1)a = K_\sF \cdot F = K_X\cdot F + c_1(\sN)\cdot F = -1 +2 =1.$$
This is impossible and hence $\phi$ is a $\mathbb{P}^1$-bundle. Moreover, $\sF$ is a flat Ehresmann connection on $\phi$. It follows that $T_Y$ is numerically projectively flat. But then \cite[Theorem 0.1]{jahnke_radloff_13} applies to show that $Y$ is a finite \'etale quotient of an abelian variety.
\end{proof}

We argue by contradiction and assume the following.

\begin{assumption}\label{assumption}
The map $T_{E/Z} \to \sN|_{E}$ is not an isomorphism.
\end{assumption}

There are finitely many prime divisors $(\Sigma_i)_{i\in I}$ on $E$ and positive integers $m_i$ for $i\in I$ such that $$\sN|_{E}\cong T_{E/Z}(\sum_{i\in I}m_i\Sigma_i).$$
By Claim \ref{claim:step_1} again, the restriction of $p$ to $\Sigma_i$ is a finite morphism $\Sigma_i \to Z$.
\begin{claim}\label{claim:step_2}
Let $i \in I$. Then $\Sigma_i$ is smooth in codimension $1$, and
the finite morphism $\Sigma_i \to Z$ is quasi-\'etale. Moreover, the vector bundles $T_{E/Z}|_{\Sigma_i}$, $T_X|_{\Sigma_i}$, $\sN_{\Sigma_i/E}$, $\sF|_{\Sigma_i}$ and 
$\sN|_{\Sigma_i}$ are numerically flat. In addition, $\sO_E(\Sigma_j)|_{\Sigma_i}\equiv 0$
for all indices $j \in I$.
\end{claim}

\begin{proof}
Set $\Sigma_i^\circ:=\Sigma_i \cap E^\circ$. Let $C \subset Z$ be a complete intersection curve of general elements of very ample linear systems on $Z$. By general choice of $C$, we may assume without loss of generality that $C \subset Z^\circ$. Set $S:=p^{-1}(C)\subset E^\circ$ and $f:=p|_S$. Write $T_f:=T_{S/C}=T_{E/Z}|_S$. By Bertini's theorem (see \cite[Th\'eor\`eme I.6.10]{jouanolou_bertini}), we may assume that $S$ is smooth and that the complete curve $B_i:=\Sigma_i\cap S$ is reduced and irreducible.

\medskip 

We show that $B_i \to C$ is \'etale and that 
$$T_f\cdot B_i=\sN\cdot B_i=B_j\cdot B_i =0$$
for all indices $j \in I$. Let $\sE_C$ be a normalized vector bundle on $C$ such that $S\cong \mathbb{P}_C(\sE_C)$.
Recall that $\NS(S)\cong \mathbb{Z}[C_0]\oplus \mathbb{Z}[F]$, where 
$F\cong \mathbb{P}^1$ is any fiber of $f$ and
$C_0$ is a section of $f$ with $C_0^2=-e=\deg \sE_C$. In addition, $T_f\equiv 2C_0+eF$, and hence $T_f^2=0$. By Lemma \ref{lemma:BB_vanishing}, we have $(\sN|_S)^2 = 0$. This easily implies that $$\sN|_S\equiv \frac{\sN|_S\cdot F}{2}T_f$$ so that $(\sum_{i\in I}m_iB_i)^2=0$.

Let $a_i$ and $b_i$ be integers such that $B_i\equiv a_i C_0+b_i F$. Then $a_i=B_i\cdot F > 0$.
By construction, we have $T_f|_{B_i}\subset \sF|_{B_i}$. By Theorem \ref{thm:numerically_projectively_flat} and Lemma \ref{lemma:minimal} together, we have 
$$T_f \cdot B_i = (2C_0+eF)\cdot(a_iC_0+b_iF)=-ea_i+2b_i \le 0.$$

Suppose first that $e \ge 0$. 

\medskip

If $I=\{i\}$ and $B_i=C_0$, then $e=0$ since $0=(\sum_{i\in I}m_iB_i)^2=-m_i^2 e$. Moreover $T_f\cdot B_i=\sN\cdot B_i=B_i\cdot B_i =0$. Otherwise, there exists $j \in I$ such that $B_j \neq C_0$. Then $b_j \ge ea_j$ by \cite[Proposition 2.20]{hartshorne77}, and hence $e=0$, $b_j=0$ and $B_j\equiv a_j C_0$. Hence we have $T_f\cdot B_i=\sN\cdot B_i=B_j\cdot B_i =0$ for all indices $j \in I$. Moreover, the morphism $B_i \to C$ is \'etale by \cite[Proof of Proposition 2.21]{hartshorne77}.

\medskip

Suppose now that $e<0$. 

\medskip

Then $2b_i \le ea_i<0$. Moreover, $B_i\neq C_0$ since $T_f\cdot C_0=-e>0$. Therefore, we have $a_i\ge 2$ and $2b_i \ge ea_i$ by \cite[Proposition 2.21]{hartshorne77}, and hence $2b_i=ea_i$. This implies $B_i\cdot B_j=0$ for all $j \in I$ and $T_f\cdot B_i=\sN\cdot B_i=0$. In addition, the morphism $B_i \to C$ is \'etale by \cite[Proof of Proposition 2.21]{hartshorne77}. 

\medskip

Since the morphism $B_i \to C$ is \'etale, we see that $\Sigma_i$ is smooth in codimension $1$ and that $\Sigma_i \to Z$ is quasi-\'etale. Let $\wt{\Sigma}_i$ denote the normalization of $\Sigma_i$.
Replacing $E$ by $\wt{\Sigma}_i\times_Z E$, if necessary, we may assume without loss of generality that $\Sigma_i$ is a section of $p$.
Then Kleiman's criterion for numerical triviality (see \cite[Lemma 4.1]{gkp_bo_bo}) applies to show that the vector bundles $T_{E/Z}|_{\Sigma_i}$, $\sN_{\Sigma_i/E}$ and  $\sN|_{\Sigma_i}$ are numerically flat. In addition, $\sO_E(\Sigma_j)|_{\Sigma_i}\equiv 0$ for all indices $j \in I$.
Recall that $T_f|_{B_i}\subset \sF|_{B_i}$. By Theorem \ref{thm:numerically_projectively_flat} and Lemma \ref{lemma:minimal} together, we see that $\sF|_{B_i}$ is numerically flat. As a consequence, 
$K_\sF\cdot B_i =0$. By Kleiman's criterion for numerical triviality again, we conclude that $K_\sF|_{\Sigma_i}\equiv 0$. This in turn implies that $\sF|_{\Sigma_i}$ is numerically flat by Theorem \ref{thm:numerically_projectively_flat}. Since $\sN|_{\Sigma_i}\equiv 0$, the vector bundle 
$T_X|_{\Sigma_i}$ is numerically flat as well, finishing the proof of Claim
\ref{claim:step_2}.
\end{proof}

\noindent\textit{Step 3.} Let us show $Z$ is a finite quasi-\'etale quotient of an abelian variety. 

\medskip

By Claim \ref{claim:step_2}, the morphisms $\Sigma_i \to Z$ are quasi-\'etale covers. Let $i\in I$. To lighten notation, set $\Sigma:=\Sigma_i$ and $\Sigma^\circ:=\Sigma_i^\circ$. 
Let $\wt{\Sigma}$ denote the normalization of $\Sigma$.
Replacing $E$ by $\wt{\Sigma}\times_Z E$, if necessary, we may assume without loss of generality that $\Sigma$ is a section of $p$. Set $\sE:=p_*\sO_E(\Sigma)$. Then $E\cong \mathbb{P}_Z(\sE)$ and there is an exact sequence 
$$0 \to \sO_Z \to \sE \to \sL \to 0,$$
where $\sL \cong \sN_{\Sigma/E}$ is a flat line bundle on $Z$ by Claim \ref{claim:step_2}. Notice that $T_{E/Z}\cong \sO_E(2\Sigma)\otimes p^*\sL^*$. 

Let $F$ be any fiber of $p$. Notice that $\sN|_E\equiv c \, T_{E/Z}\equiv 2c \Sigma$ with $c:=\frac{\sN \cdot F}{2}$ since 
$\sN|_E\cdot F = c \, T_{E/Z} \cdot F$ and $\sN|_{\Sigma} \equiv T_{E/Z}|_\Sigma\equiv 0$ by Claim \ref{claim:step_2}.

Let now $k$ be an integer such that $\sF(k\Sigma)|_F \cong \sO_F^{\oplus n-1}$. Then $\sF|_E \cong p^*\sK\otimes \sO_E(-k\Sigma)$ for some vector bundle $\sK$ on $Z$. Since $\sF|_\Sigma$ and $\sN_{\Sigma/E}$ are numerically flat by Claim \ref{claim:step_2}, the locally free sheaf $\sK$ is numerically flat as well. This immediately implies that 
\begin{align*}
c_1(\sF|_E)) & = p^*c_1(\sK) -k(n-1)\Sigma\\
& \equiv -k(n-1)\Sigma
\end{align*}
and
\begin{align*}
c_2(\sF|_E)) & = p^*c_2(\sK) -(n-2)k\Sigma\cdot p^*c_1(\sK)+\frac{(n-1)(n-2)}{2}k^2\Sigma^2\\
& \equiv 0
\end{align*}
since $\Sigma^2\equiv 0$. The exact sequence 
$$ 0 \to \sF|_E \to T_X|_E \to \sN|_E \to 0$$
then gives
\begin{align*}
c_1(T_X|_E) & = c_1(\sF|_E)+c_1(\sN|_E)\\
& \equiv \Sigma
\end{align*}
using $K_X\cdot F=-1$, and 
\begin{align*}
c_2(T_X|_E) & =c_2(\sF|_E))+c_1(\sF|_E))\cdot c_1(\sN|_E) \\
 & \equiv 0
\end{align*}
using $c_1(\sN)\cdot \Sigma\equiv 0$. In particular, there exists a flat line bundle $\sM$ on $Z$ such that $$\sN|_E \cong p^*\sM\otimes\sO_E((k(n-1)+1)\Sigma).$$
By the projection formula, there exist rational numbers $a_1$, $a_2$ and $a_3$ such that 
\begin{equation}\label{eq:first_computation}
p_*c_2(T_X|_E) = a_1c_1(\sK)+a_2c_1(\sM)+a_3c_1(\sN_{\Sigma/E})
\end{equation}
as cycle classes in $\textup{CH}^1(Z)$.
The exact sequence
$$0 \to T_{E^\circ/Z^\circ} \to T_{E^\circ} \to (p^\circ)^*T_{Z^\circ} \to 0$$
gives
\begin{align*}
c_1(T_{E^\circ}) & = 2\Sigma^\circ-(p^\circ)^*c_1(\sL|_{Z^\circ})+(p^\circ)^*c_1(T_{Z^\circ})\\ 
& \equiv 2\Sigma^\circ+(p^\circ)^*c_1(T_{Z^\circ})
\end{align*}
and 
\begin{align*}
c_2(T_{E^\circ}) & =c_1(T_{E^\circ/Z^\circ})\cdot (p^\circ)^*c_1(T_{Z^\circ})+(p^\circ)^*c_2(T_{Z^\circ})\\ 
& = (2 \Sigma^\circ - (p^\circ)^*c_1(\sL|_{Z^\circ}))\cdot (p^\circ)^*c_1(T_{Z^\circ})+(p^\circ)^*c_2(T_{Z^\circ}).
\end{align*}
Notice that the cokernel $\sN_{E^\circ/X^\circ}$ of the tangent map $T_{E^\circ}\to T_X|_{E^\circ}$ 
is locally free since the morphism $E^\circ \to X$ is unramified by choice of $E^\circ$.
Then we have
\begin{align*}
c_2(T_X|_{E^\circ}) & =  c_2(T_{E^\circ})+c_1(T_{E^\circ})\cdot c_1(\sN_{E^\circ/X^\circ}) \\
& = c_2(T_{E^\circ})+c_1(T_{E^\circ})\cdot (c_1(T_X|_{E^\circ})-c_1(T_{E^\circ}))\\
& \equiv 2 \Sigma^\circ\cdot (p^\circ)^*c_1(T_{Z^\circ})+(p^\circ)^*c_2(T_{Z^\circ})+(2\Sigma^\circ+(p^\circ)^*c_1(T_{Z^\circ}))\cdot (-\Sigma^\circ-(p^\circ)^*c_1(T_{Z^\circ}))\\
& \equiv - \Sigma^\circ\cdot (p^\circ)^*c_1(T_{Z^\circ})+(p^\circ)^*c_2(T_{Z^\circ})-(p^\circ)^*c_1(T_{Z^\circ})^2
\end{align*}
By the projection formula again, there exist rational numbers $b_1$, $b_2$, $b_3$ et $b_4$ such that
\begin{equation}\label{eq:second_computation}
(p^\circ)_*c_2(T_X|_{E^\circ}) = - c_1(T_{Z^\circ}) + 
b_1c_1(\sK|_{Z^\circ})+b_2c_1(\sM|_{Z^\circ})+b_3c_1(\sL|_{Z^\circ})+b_4c_1(\sN_{\Sigma^\circ/E^\circ})
\end{equation}
as cycle classes in $\textup{CH}^1(Z^\circ)$.
Equations \eqref{eq:first_computation} and \eqref{eq:second_computation} together then show that $K_Z$ is $\mathbb{Q}$-Cartier using the fact that $Z \setminus Z^\circ$ has codimension at least $2$ in $Z$. Moreover, $K_Z \cdot C = 0$ for any complete intersection curve $C$ of general elements of very ample linear systems on $Z$ so that Kleiman's criterion for numerical triviality (see \cite[Lemma 4.1]{gkp_bo_bo}) applies to show that $K_Z\equiv 0$.

Suppose that $\phi$ is a $\mathbb{P}^1$-bundle. Then $T_\Sigma$ is numerically projectively flat since both $T_X|_\Sigma$ and $\sN_{\Sigma/X}$ are numerically projectively flat by Claim \ref{claim:step_2} above. This implies that $T_Y$ is numerically projectively flat as well. Then
\cite[Theorem 0.1]{jahnke_radloff_13} applies to show that $Y=Z$ is a finite \'etale quotient of an abelian variety.

Suppose now that $\phi$ is a singular conic bundle or a blow-up. By Lemma \ref{lemma:rational_surface} and Claim \ref{claim:step_1} together, $Z$ contains no rational curve. Then \cite[Corollary 0.3]{bdpp} applied to a resolution of $Z$ shows that $Z$ has canonical singularities. By \cite[Proposition 1.16]{demailly_peternell_schneider94}, the map 
$\det T_\Sigma \to \wedge^{n-2}(T_X|_{\Sigma})$ does not vanish anywhere since $T_X|_{\Sigma}$ is numerically flat by Claim \ref{claim:step_2}. This immediately implies that $T_\Sigma$ is a locally free and numerically flat. By the solution of the Zariski-Lipman conjecture for klt spaces (see \cite[Theorem 6.1]{greb_kebekus_kovacs_peternell10} or \cite[Theorem 1.1]{druel_zl}), we see that $Z\cong \Sigma$ is smooth. Moreover, by \cite[Corollary 1.19]{demailly_peternell_schneider94}, we have $c_2(Z)=0$. As a classical consequence of Yau's theorem on the existence of a K\"ahler-Einstein metric, $Z$ is then covered by a complex torus (see \cite[Chapter IV Corollary 4.15]{kobayashi_diff_geom_vb}). 

\medskip

This shows that $Z$ is a quasi-\'etale quotient of an abelian variety.

\medskip

\noindent\textit{Step 4.} Finally, let us obtain a contradiction.

\begin{claim}\label{claim:step4}
The vector bundle $\sF|_E$ is not numerically flat. 
\end{claim}

\begin{proof}
Suppose first that $\phi$ is a $\mathbb{P}^1$-bundle. Using Claim \ref{claim:step_1}, we then conclude that $\sF$ induces a flat Ehresmann connection on $p$. This immediately implies that the map $T_{E/Z} \to \sN|_{E}$ is an isomorphism, which contradicts Assumption \ref{assumption}.

Suppose now that $\phi$ is a singular conic bundle or a blow-up. Then $T_X|_F \cong \sO_{\mathbb{P}^1}(2)\oplus \sO_{\mathbb{P}^1}^{n-2}\oplus \sO_{\mathbb{P}^1}(-1)$. But this is impossible by Lemma \ref{lemma:rational_curve_uniruled}, proving our claim.
\end{proof}

By Step 3, replacing $Z$ by a finite quasi-\'etale cover, if necessary, we may also assume that $Z$ is an abelian variety. Recall also that $E\cong \mathbb{P}_Z(\sE)$, where $\sE$ is numerically flat. In particular, $\sE$ is flat by the Corlette–Simpson correspondence (see Theorem \ref{thm:numerically_flat}). As a consequence, there is a flat Ehresmann connection $\sH\cong \sO_E^{\oplus n-2}$ on $p$.

Let $\omega\in H^0(X,\Omega_X^1\otimes\sN)$ be a twisted $1$-form defining $\sF$, and let $\sG$ be the foliation on $E$ induced by $\sF$. 
Suppose that $\omega|_E\in H^0(E,\Omega_E^1\otimes\sN|_E)$ vanishes at some point $x \in E$. Then $x \in \Sigma_i$ for some $i\in I$. In addition, $\omega|_E$ vanishes along $\Sigma_i$ since $\Omega_E^1\otimes\sN|_{\Sigma_i}$ is numerically flat (see \cite[Proposition 1.16]{demailly_peternell_schneider94}). Repeating the argument, we see that there exist integers $0 \le k_i \le m_i$ such that $\sN_\sG\cong T_{E/Z}(\sum_{i\in I}k_i\Sigma_i)$ and that $\sG$ is regular.

Suppose first that $\sO_E^{\oplus n-2}\cong\sH=\sG\subset T_E$. Then $\sF|_E$ is numerically flat 
by Theorem \ref{thm:numerically_projectively_flat} and Lemma \ref{lemma:minimal} together. But this contradicts Claim \ref{claim:step4}.

Suppose from now on that $\sH \neq \sG \subset T_E$, and let $x \in X$ be a general point. 
Then there exists $v \in H^0(E,\sH)$ with $v(x) \in \sG_x\setminus \{0\}$ since $n \ge 3$. 
If $h^0(E,\sF|_E)\ge 1$, then $\sF|_E$ is numerically flat 
by Theorem \ref{thm:numerically_projectively_flat} and Lemma \ref{lemma:minimal} together, which is impossible by Claim \ref{claim:step4} again. Thus we have 
$h^0(E,\sF|_E)=0$. But then $v$ induces a non-zero section $s \in H^0(E,\sN|_E)$ vanishing at $x$. 
Let $\sum_{j\in J}d_j \Lambda_j$ be the zero divisor of $s$. Since $\sN|_E\equiv 2c\Sigma$ is nef and $\Sigma^2 \equiv 0$ we see that $\Lambda_j\cdot \Sigma\equiv0$. This easily implies $\Lambda_j^2\equiv 0$ since $\sE$ is numerically flat.
Arguing as in the proof of Claim \ref{claim:step_2}, we conclude that 
$\Lambda_j$ is smooth in codimension $1$ and that $\Lambda_j \to Z$ is quasi-\'etale. 
By Serre's criterion for normality, $\Lambda_j$ is normal and $\Lambda_j \to Z$ is \'etale by Nagata-Zariski purity theorem. Let $j \in J$ such that $x \in \Lambda_j$. Then $\Lambda_j \cap \Sigma = \emptyset$ since $x \not\in \Sigma$. Replacing $Z$ by a finite \'etale cover, if necessary, we may therefore assume that $p$ has $3$ pairwise disjoint sections. Then $E \cong \mathbb{P}^1 \times Z$.
In addition, $\sN_\sG\cong q^*\sR$ for some line bundle $\sR$ on $\mathbb{P}^1$, where $q:=\textup{pr}_{\mathbb{P}^1}$. 
By Lemma \ref{lemma:product}, we have $h^0(E,\sF|_E) \ge h^0(E,\sG) \ge 1$ since $n\ge 4$ by assumption, yielding a contradiction. This finishes the proof of Theorem \ref{thm:suspension_complex_torus}.
\end{proof}

\section{Abundance}\label{section:abundance}

In this section we prove a special case of the abundance conjecture (see Proposition \ref{prop:abundance} below). More precisely, we prove that if the pair $(X,\sF)$ satisfies all the conditions listed in Setup \ref{setup:main}, then $\kappa(X)=\nu(X)$. First, we consider foliations with projectively trivial tangent bundle. The proof of Proposition \ref{prop:semiample_when_projectively_trivial} below
follows an argument that goes back at least as far as \cite[Chapter 9]{brunella}.

\begin{prop}\label{prop:semiample_when_projectively_trivial}
Let $X$ be a complex projective manifold of dimension $n\ge 3$, and let $\sF$ be a regular codimension $1$ foliation. Suppose that $\sF\cong \sL^{\oplus n-1}$ for some line bundle $\sL$. Then $\sL^*$ is semiample with $\nu(\sL^*)\le 1$.
\end{prop}

\begin{proof} 
For the reader's convenience, the proof is subdivided into a number of steps. Notice that 
$\sL^*$ is nef by Lemma \ref{lemma:minimal} and that $\nu(\sL^*)\le 1$ by Proposition \ref{prop:vanishing_chern_classes}. Therefore, in order to prove the 
proposition, it suffices to show that $\kappa(\sL^*)=\nu(\sL^*)$ (see \cite[Proposition 2.1]{kawamata}). Notice also that $\kappa(\sL^*)\le \nu(\sL^*)$ by \cite[Proposition 2.2]{kawamata}. 

\medskip

\noindent\textit{Step 1.} We show that $\kappa(\sL^*)\ge 0$.

\medskip

Suppose first that $q(X)>0$. 

\medskip

If $h^0(X,\sL^*) \ge 1$, then there is nothing to show. Suppose now that $h^0(X,\sL^*)=0$. Let $a\colon X \to A$ be the Albanese morphism, and let $f \colon X \to Y$ be its Stein factorization. Notice that $\dim Y \ge 1$ since $q(X)>0$ by assumption. 
Now the standard pull-back map of K\"ahler differentials $$H^0(X,\Omega_X^1)\otimes \sO_X\cong a^*\Omega_A^1 \to \Omega_X^1$$ identifies with the evaluation map. It follows that $\sF$ is contained in the kernel of the tangent map $$Ta \colon T_X \to a^*T_A$$ since $h^0(X,\sL^*)=0$ by our current assumption. Thus $\dim Y=1$, and $\sF$ is induced by the fibration $f$. Let $F$ be a general fiber of $f$. By \cite[Theorem 0.1]{jahnke_radloff_13}, $F$ is a finite \'etale quotient of an abelian variety. In particular, $\sL|_F$ is a torsion line bundle. This easily implies that there exists a line bundle $\sM$ on $Y$ and
a positive integer $m$ such that $\sL^{\otimes m}\cong f^*\sM$ (see also Proposition \ref{prop:pull-back}). Notice that $\sM^*$ is nef since $\sL^*$ is nef. But then $\kappa(\sM^*)\ge 0$ and hence $\kappa(\sL^*)\ge 0$.

\medskip

Suppose now that $q(X)=0$.

\medskip

If $\sL\equiv 0$, then $\sL$ is a torsion line bundle, and hence $\kappa(\sL^*)=0$. Suppose now that $\sL\not\equiv 0$. From Proposition \ref{prop:vanishing_chern_classes2}, we get $\chi(X,\sO_X)=0$ and hence $h^p(X,\sO_X)>0$ for some $p > 0$. By Hodge symmetry, we then have $h^0(X,\Omega_X^p)>0$. Using the exact sequence 
$$  0 \to \sN^* \to \Omega_X^1 \to \sF^* \to 0$$
we see that $h^0\big(X,(\sN^*)^{\otimes u}\otimes (\sL^*)^{\otimes v}\big)>0$ for some non-negative integers $u$ and $v$ such that $u+v=p$. By Proposition \ref{prop:vanishing_chern_classes}, there exist integers $r>0$ and $s$ such that $\sN^{\otimes r}\cong\sL^{\otimes s}$. It follows that     
$h^0\big(X,(\sL^*)^{\otimes su+rv}\big)>0$. Since $\sL^*$ is nef and $\sL^*\not\equiv 0$, we must have $su+rv \ge 0$. This shows that $\kappa(\sL^*)\ge 0$ in the case where $q(X)=0$.

\medskip

\noindent\textit{Step 2.} Write 
\begin{equation}\label{eq:c_1}
-c_1(\sL) \sim_\mathbb{Q} \sum_{i \in I} b_i B_i
\end{equation}
with $b_i \in \mathbb{Q}_{>0}$. If $\sL\equiv 0$, then $\sL$ is a torsion line bundle since $\kappa(\sL^*)\ge 0$ by Step 1, proving the proposition in this case. Suppose from now on that $\sL\not\equiv 0$. In other words, assume that $I$ is not empty.
Then Proposition \ref{prop:trivial_representation_torus_quotient} applies to show that for any index $i\in I$,
$B_i$ is a finite \'etale quotient of an abelian variety. Moreover, $B_i\cap B_j =\emptyset$ if $i\neq j$ and $\sN_{B_i/X}$ is a flat line bundle. Notice that $\sL|_{B_i}$ is a flat line bundle as well by Equation \eqref{eq:c_1}.

\medskip

\noindent\textit{Step 3.} Suppose now that the composition 
$\sL^{\oplus n-1}|_{B_i}\cong \sF \to T_X|_{B_i} \to \sN_{B_i/X}$ is non-zero. In other words, assume that $B_i$ is not a leaf of $\sF$. Then there exists a foliation by curves $\sG\cong \sL \subset T_X$ such that $B_i$ is everywhere transverse to $\sL$ since both 
$\sL|_{B_i}$ and $\sN_{B_i/X}$ are flat line bundles. In particular, we have $\sN_{B_i/X}\cong \sL|_{B_i}$. Moreover, there is a section $s \in H^0(X,\sO_X(m K_\sG))$ for some integer $m \ge 1$ that vanishes on $B_i$ and nowhere else in some analytic open neighborhood of $B_i$ by Equation \ref{eq:c_1} again. Arguing as in the proof of \cite[Theorem 9.2]{brunella} on page $112$, we conclude that there exists a fibration $f \colon X \to C$ onto a smooth complete curve $C$ such that $B_i$ is a (set-theoretic) fiber of $f$. This easily implies that $\kappa(\sL^*)=\nu(\sL^*)=1$.

\medskip 

\noindent\textit{Step 4.} Suppose finally that $B_i$ is a leaf of $\sF$ for all indices $i \in I$.
Then $T_{B_i} \cong \sF|_{B_i}\cong (\sL|_{B_i})^{\oplus n-1}$ so that Lemma \ref{lemma:characterization_a_v} below applies to show  
that $B_i$ is an abelian variety and $\sL|_{B_i} \cong \sO_{B_i}$. Moreover, $\sN_{B_1/X}$ is a flat vector bundle by Lemma \ref{lemma:BB_vanishing}.

Let $m$ be a positive integer such that $-mc_1(\sL) \sim_\mathbb{Z} \sum_{i \in I} mb_i B_i=:B$. In particular, $mb_i$ is an integer. Let $\gamma\colon X_1 \to X$ be the associated cyclic cover (see \cite[Definition 2.52]{kollar_mori}) which is \'etale away from $B$. Notice that $X_1$ is smooth. Set $B_1:=\gamma^{-1}(B)$. By \cite[Lemma 3.4(1)]{cd1fzerocan}, the foliation $\sF_1$ induced by
$\sF$ on $X_1$ is regular and $\sF_1 \cong \gamma^*\sF$. Using \cite[Lemma 5.12]{uenoLN439}, one readily checks that $\sL^*$ is semiample if and only if so is $\gamma^*\sL^*$. 
Moreover, $\nu(\sL^*)=\nu(\gamma^*\sL^*)$. Therefore, replacing $X$ by $X_1$, we may assume without loss of generality that $-c_1(\sL) \sim_\mathbb{Z} \sum_{i \in I} b_i B_i=:B$ with $b_i \in \mathbb{Z}$.

\medskip

Suppose first that $q(X)>0$. 

\medskip

Let $a\colon X \to A$ be the Albanese morphism, and let $f \colon X \to Y$ be its Stein factorization. By assumption, we have $\dim Y \ge 1$. 

Suppose that $\dim f(B_i) \ge 1$ for some $i \in I$. Then there exists a global $1$-form on $X$ whose restriction to $\sF|_{B_i}$ is non-zero at a general point in $B_i$. This implies $\kappa(\sL^*) = \nu(\sL^*) = 1$. 

Suppose now that $\dim f(B_i) = 0$ for any $i \in I$. Then Lemma \ref{lemma:hodge_index_theorem} applies to show that $\dim Y=1$. Now, by Zariski's Lemma, some non-zero multiple of $B_i$ is a fiber of $f$ since $B_i^2\equiv 0$. This again implies $\kappa(\sL^*) = \nu(\sL^*) = 1$.

\medskip

Suppose finally that $q(X)=0$. 

\medskip

Then $\sO_X(b_iB_i)|_{B_i}\cong\sL^*|_{B_i}\cong \sO_{B_i}$. In other words, 
$\sO_X(B_i)|_{B_i}$ is torsion, say of order $m_i$.
In particular, for any $1 \le k\le m_i-1$, $h^1(B_i,\sO_X(k B_i)|_{B_i})=0$. 
The short exact sequence
$$ 0 \to \sO_X((k-1) B_i) \to \sO_X(k B_i) \to \sO_X(k B_i)|_{B_i} \to 0$$
then gives $h^1(X,\sO_X(k B_i))=0$ for all $1\le k\le m_i-1$. It follows that the restriction map 
$$ H^0(X,\sO_X(m_i B_i))\to H^0(X,\sO_X(m_i B_i)|_{B_i})\cong H^0(X,\sO_{B_i})$$
is surjective. This implies $\kappa(\sL^*) \ge \kappa(\sO_X(B_i)) \ge 1$, and hence 
$\kappa(\sL^*) = \nu(\sL^*) = 1$.
\end{proof}

\begin{lemma}\label{lemma:characterization_a_v}
Let $X$ be a compact Kähler manifold of dimension $n \ge 2$. Suppose that $T_X\cong \sL^{\oplus n}$ for some line bundle $\sL$ on $X$. Then $X$ is a torus.
\end{lemma}

\begin{proof}
By \cite[Proposition 4.3]{GKP_proj_flat_JEP}, $c_1(\sL)=0$. If $q(X)=0$, then $\sL$ is torsion. Otherwise, we have $\sL\cong \sO_X$. In either case, $\sL$ is torsion, say of order $m$. Let $\gamma\colon Y \to X$ be the associated cyclic cover (see \cite[Definition 2.52]{kollar_mori}) which is \'etale with Galois group $G=\langle \tau \rangle$. By construction, $T_Y \cong \gamma^*T_X \cong \sO_Y^{\oplus n}$.
By a result proved by Wang (\cite[Corollary 2]{wang_parallisable}), $Y$ is a complex torus, 
$Y=V/\Lambda$, where $V$ is complex vector space of dimension $n$ and $\Lambda \subset V$ is a lattice.
Set $y_0:=\tau(0)$. There exists a $\mathbb{C}$-linear map $L \colon V \to V$ satisfying $L(\Lambda)\subseteq \Lambda$ such that $\tau = t_{y_0}\circ \wb{L}$, where $\wb{L}$ is the automorphism of $Y$ induced by $L$. Notice that the linear map $L$ can be identified with $\tau_*$ acting on $H^0(Y,T_Y)\cong V$.
Since $T_Y \cong \gamma^*T_X \cong (\gamma^*\sL)^{\oplus n}$, we see that $L=\lambda \textup{Id}$ for some $\lambda\in \mathbb{C}$ with $\lambda^m=1$. By \cite[Lemma 13.1.1]{birkenhake_lange}, the set of fixed points of $\wb{L}$ is a positive-dimensional analytic subvariety since $\tau$ has no fixed point. But this immediately implies that $\lambda=1$. As a consequence, $\tau$ is a translation by a point of finite order. This in turn implies that $X$ is torus and that $\sL\cong \sO_X$.
\end{proof}

The following is the main result of this section.

\begin{prop}\label{prop:abundance}
Setting and notation as in \ref{setup:main}. Suppose in addition that $K_X$ is pseudo-effective. Then $K_X$ is abundant.
\end{prop}

\begin{proof}
Let $\textup{H} \subseteq \textup{PGL}(n-1,\mathbb{C})$ be the Zariski closure of $\rho(\pi_1(X))$. This is a linear algebraic group which has finitely many connected components. Applying Selberg's Lemma and passing to an appropriate finite \'etale cover of $X$, we may assume without loss of generality that $\textup{H}$ is connected and that the image of the induced representation
$$\rho_1 \colon \pi_1(X) \to \textup{H} \to \textup{H}/\textup{Rad}(\textup{H})$$
is torsion free, where $\textup{Rad}(\textup{H})$ denotes the radical of $\textup{H}$. Let 
$$\textup{sh}_{\rho_1}\colon X \dashrightarrow Y$$
be the $\rho_1$-Shafarevich map, where $\textup{sh}_{\rho_1}$ is dominant and
$Y$ is a smooth projective variety. The rational map $\textup{sh}_{\rho_1}$ is almost proper with connected fibers. By \cite[Th\'eor\`eme 1]{CCE15}, we may assume without loss of generality that $Y$ is of general type and that the representation $\rho_1$ factorizes through $\textup{sh}_{\rho_1}$. If $\dim Y = \dim X$, then $X$ is of general type and there is nothing to show.
Suppose from now on that $\dim Y < \dim X$.

Let $F$ be a general (smooth) fiber of $\textup{sh}_{\rho_1}$. By \cite[Proposition 2.23]{druel_proj_flat}, in order to prove that $K_X$ is abundant, it suffices to prove that $K_F$ is abundant. 

\medskip

Suppose first $\dim Y = 0$ and $q(X)=0$. 

\medskip

Since $\rho(\pi_1(X)) \subseteq \textup{Rad}(\textup{H})$, the restriction of $\rho$ to a finite index subgroup of $\pi_1(X)$ is the trivial representation. Thus, there exists a finite \'etale cover $\gamma\colon Y \to X$, as well as a line bundle $\sL$ on $Y$ such that $\gamma^{-1}\sF\cong \sL^{\oplus n-1}$. Then 
Proposition \ref{prop:semiample_when_projectively_trivial} applies to show that $\sL^*$ is semiample. 
This in turn implies that $\sO_X(K_\sF)$ is semiample (\cite[Lemma 5.12]{uenoLN439}).

If $K_\sF\equiv 0$, then our claim follows from Proposition \ref{prop:classification_zero_canonical_class}.

Suppose now that $K_\sF\not\equiv 0$. By Proposition \ref{prop:vanishing_chern_classes} and using the fact that $q(X)=0$, there exist integers $r>0$ and $s$ such that $\sN^{\otimes r}\cong\sO_X(K_\sF)^{\otimes s}$. The claim now follows easily from the adjunction formula $\sO_X(K_X)\cong \sO_X(K_\sF)\otimes \sN^*$.

\medskip

Suppose from now on that either $\dim Y>0$ or $q(X)>0$. 

\medskip

Let $a \colon F \to \textup{A}$ be the Albanese morphism. By construction, we have $\dim \textup{A} = q(F)$. Let $G$ be a general (smooth) fiber of the Stein factorization of $F \to a(F)$. 
By \cite[Theorem 1.1]{hu_log_abundance}, in order to prove that $K_F$ is abundant, it suffices to prove that $K_G$ is abundant. Thus, we may assume without loss of generality that $\dim G >0$. Repeating the process finitely many times, we may also assume that $q(G)=0$. Since $\rho(\pi_1(G)) \subseteq \textup{Rad}(\textup{H})$, the restriction of $\rho$ to a finite index subgroup of $\pi_1(G)$ is the trivial representation. 

Let $U \subseteq X$ be a neighborhood of $G$ in $X$ for the analytic topology. Shrinking $U$, if necessary, we may assume that there exists a smooth projective morphism with connected fibers $f \colon U \to T$ onto a simply connected complex manifold $T$ such that $G$ is a fiber of $f$. Then 
$\sN_{G/U}\cong \sO_{G}^{\oplus \dim X - \dim G}$. Moreover, we have $\dim G < \dim X$ because $\dim Y >0$ or $q(X)>0$ by our current assumption. Since $\rho(\pi_1(G))$ is finite, there exists a finite \'etale cover $g \colon U_1 \to U$ such that the inverse image $G_1$ of $G$ in $U_1$ is connected and such that $\sF_{U_1}:=g^{-1}(\sF|_U)\cong \sL_{U_1}^{\oplus n-1}$ for some line bundle $\sL_{U_1}$ on $U_1$. 

If $G_1$ is a leaf of $\sF_{U_1}$, then $G_1$ is an abelian variety by Lemma \ref{lemma:characterization_a_v}. This in turn implies that $K_G$ is torsion.

Suppose from now on that either $\dim G_1 \le n-2$ or that $\dim G_1 = n-1$ and $G_1$ is not a leaf of $\sF_{U_1}$. Then the composition 
$$\sF_{U_1}|_{G_1}\cong \sL_{U_1}|_{G_1}^{\oplus n-1} \to T_{U_1}|_{G_1} \to \sN_{G_1/U_1}\cong \sO_{G_1}^{\oplus \dim X - \dim G}$$ 
is non-zero, and hence $h^0(G_1,\sL_{G_1}^*)\ge 1$, where 
$\sL_{G_1}:=\sL_{U_1}|_{G_1}$. 

Suppose that $\sL_{G_1}\cong \sO_{G_1}$. If $T_{G_1} \subset \sF_{U_1}|_{G_1}\cong \sO_{G_1}^{\oplus n-1}$ then $\Omega_{G_1}^1$ is generated by its global sections and hence $K_{G_1}$ is abundant. It follows from \cite[Theorem 5.13]{uenoLN439} that $K_G$ is abundant as well. Suppose now that $T_{G_1} \not\subset \sF_{U_1}|_{G_1}$. Then $T_{G_1} \cap (\sF_{U_1}|_{G_1}) \subset T_{G_1}$ is a codimension $1$ foliation with trivial tangent bundle.
Then $K_{G_1}$ is abundant by Proposition \ref{prop:classification_zero_canonical_class}, using the fact that $K_{G_1}=K_X|_{G_1}$ is pseudo-effective. Again, by \cite[Theorem 5.13]{uenoLN439}, we conclude that $K_G$ is abundant.

Suppose finally that $\sL_{G_1}\not\equiv 0$. By Lemma \ref{lemma:minimal}, $\sL_{G_1}^*$ is nef. Moreover, by Proposition \ref{prop:vanishing_chern_classes}, there is a number $\lambda \in \mathbb{Q}$ such that $K_{G_1} \equiv \lambda c_1(\sL_{G_1})$. Observe that $\lambda \le 0$ since $K_G=K_X|_G$ is pseudo-effective. Since $h^0(G_1,\sL_{G_1}^*)\ge 1$, we see that $\kappa(G_1) \ge 0$ by  \cite[Corollary 3.2]{CKP_numerical}. Notice also that
$K_{G_1}$ and hence $K_G$ are nef. Then \cite[Theorem 5.13]{uenoLN439} implies that $\kappa(K_G) \ge 0$.
Write $K_G\sim_\mathbb{Q}\sum_{i \in I} b_i B_i$ with $b_i \in \mathbb{Q}_{>0}$, $B_i$ irreducible and 
$B_i \neq B_j$ if $i\neq j$. By Proposition \ref{prop:trivial_representation_torus_quotient} applied to 
$K_{G_1}=(g|_{G_1})^* K_G\sim_\mathbb{Q}\sum_{i \in I} b_i (g|_{G_1})^*B_i$, $B_i$ is a finite \'etale quotient of an abelian variety and $B_i \cap B_j=\emptyset$ if $i\neq j$. Moreover, we have $B_i^2\equiv 0$. Then Proposition \ref{prop:log_abundance} applies to show that $K_G$ is abundant. This finishes the proof of the proposition.
\end{proof}

\section{Projectively flat foliations with nef conormal bundle}\label{section:nef_conormal}

In this section we address codimension $1$ regular foliations with numerically projectively flat tangent sheaf and nef conormal bundle. The proof of Theorem \ref{thm:case_conormal_nef} below makes use of the following auxiliary result.

\begin{lemma}\label{lemma:isotriviality}
Setting and notation as in \ref{setup:main}. Suppose that $X$ is an abelian scheme over a smooth projective base $Y$ of dimension $m:=\dim Y \le n-1$, and denote by $f \colon X \to Y$ the natural morphism. 
Suppose in addition that $\sN\cong f^*\sM$ for some line bundle $\sM$ on $Y$ and that $T_{X/Y} \not\subseteq \sF$. Then $f$ is isotrivial. 
\end{lemma}

\begin{proof}
Let $\omega\in H^0(X,\Omega_X^1\otimes \sN)$ be a twisted $1$-form defining $\sF$, and let $\Sigma\subset X$ be the neutral section of $f$. 
By assumption, the composition $\sF \to T_X \to f^*T_Y$ is generically surjective.
Notice that the sheaf $f_*T_{X/Y}$ is locally free of rank $n-m$.
Let $F$ be a general fiber of $f$. The morphism $ \sO_F^{\oplus n-m} \cong T_{X/Y}|_F \to \sN|_F\cong \sO_F$ induced by $\omega$
is generically surjective by assumption, and hence surjective. This implies that the saturation
$\sE$ of $f_*(\sF\cap T_{X/Y})$ in $f_*T_{X/Y}$ has rank $n-m-1$. Let $Y^\circ \subseteq Y$ be an open subset with complement of codimension at least $2$ such that $\sE^\circ:=\sE|_{Y^\circ}$ is a subbundle of $f_*T_{X/Y}$. Set $X^\circ:=f^{-1}(Y^\circ)$ and $f^\circ:=f|_{X^\circ}$.
Then 
$(f^\circ)^*\sE^\circ=\sF^\circ\cap T_{X^\circ/Y^\circ}\subset T_{X^\circ/Y^\circ}$ since $\sF^\circ\cap T_{X^\circ/Y^\circ}$ is saturated in 
$T_{X^\circ/Y^\circ}$. In particular, $\sF^\circ\cap T_{X^\circ/Y^\circ}$
is a subbundle of $\sF^\circ$. Set $\sQ^\circ:=\sF^\circ/(\sF^\circ\cap T_{X^\circ/Y^\circ})$. Notice also that $$\sQ^\circ|_F \cong (f^*T_Y)|_F\cong \sO_F^{\oplus m},$$ where $F$ is a general fiber $F$ of $f^\circ$. In particular, $K_\sF|_F \sim_\mathbb{Z} 0$. Then Lemma \ref{lemma:pull_back} applies to show that there exists a line bundle $\sK$ on $Y$ such that $\sO_X(K_\sF)\cong f^*\sK$.

Let now $F$ be any fiber of $f^\circ$. Then the vector bundle $\sF|_F$ is numerically flat since 
$K_\sF|_F \sim_\mathbb{Z} 0$. It follows that $\sQ^\circ|_F$ is numerically flat as well.
Then Lemma \ref{lemma:descent_vb} applies to show that there exists a vector bundle $\sG^\circ$ on $Y^\circ$ such that $\sQ^\circ \cong (f^\circ)^*\sG^\circ$. By construction, there is an exact sequence
$$0 \to (f^\circ)^*\sE^\circ \to \sF^\circ \to (f^\circ)^*\sG^\circ \to 0.$$

Let again $F$ be a general fiber of $f^\circ$. To prove the statement, it suffices to show that the exact sequence
\begin{equation}\label{eq:relative_tangent_sequence}
0 \to T_{X/Y} |_F \to T_X|_F \to (f^*T_Y)|_F \to 0
\end{equation}
splits. This is because complex flows of vector fields on analytic spaces exist.
If the short exact sequence  
\begin{equation}\label{eq:exact_sequence}
0 \to (f^\circ)^*\sE^\circ|_F \to \sF^\circ|_F \to (f^\circ)^*\sG^\circ|_F \to 0
\end{equation}
splits, then the exact sequence \eqref{eq:relative_tangent_sequence} splits as well, proving our claim. Suppose from now on that the exact sequence \eqref{eq:exact_sequence} does not split. Then the image of $\textup{Id}_{\sG^\circ}\in H^0(Y^\circ,\sE nd(\sG^\circ))$ under the connecting homomorphism $$H^0(Y^\circ,\sE nd(\sG^\circ)) \to R^1 f^\circ_*\big((f^\circ)^*\sE^\circ\otimes((f^\circ)^*\sG^\circ)^*\big)
\cong R^1 f^\circ_*\sO_{X^\circ} \otimes \sE^\circ\otimes(\sG^\circ)^*$$
is non-zero.  
Let $C$ be a complete intersection curve of general elements of a very ample linear system on $Y$. By general choice of $C$, we have $C \subseteq Y^\circ$ since $Y \setminus Y^\circ$ has codimension at least $2$. Moreover,
$$\mu^{\textup{max}}\big((R^1 f^\circ_*\sO_{X^\circ} \otimes \sE^\circ\otimes(\sG^\circ)^*)|_C\big) \ge 0.$$
Then
\begin{eqnarray*}
\mu^{\textup{max}}\big((R^1 f^\circ_*\sO_{X^\circ} \otimes \sE^\circ\otimes(\sG^\circ)^*)|_C\big)
& = & \mu^{\textup{max}}\big(R^1 f^\circ_*\sO_{X^\circ}|_C\big)+\mu^{\textup{max}}\big(\sE^\circ|_C\big)+\mu^{\textup{max}}\big((\sG^\circ)^*|_C\big)\\
& = & \mu^{\textup{max}}\big(R^1 f^\circ_*\sO_{X^\circ}|_C\big)+\mu^{\textup{max}}\big(\sE^\circ|_C\big)-\mu^{\textup{min}}\big(\sG^\circ|_C\big),
\end{eqnarray*}
and hence
$$\mu^{\textup{min}}\big(\sG^\circ|_C\big) \le \mu^{\textup{max}}\big(R^1 f^\circ_*\sO_{X^\circ}|_C\big)+\mu^{\textup{max}}\big(\sE^\circ|_C\big).$$
Since $(R^1 f_*\sO_X)^*\cong f_*\Omega^1_{X/Y}$ is nef (see 
\cite[Theorem 5.2]{griffiths_periods_3}), we have
$$\mu^{\textup{max}}\big(R^1 f^\circ_*\sO_{X^\circ}|_C\big) \le 0.$$ 
On the other hand, viewing $C$ as a curve in $\Sigma$, 
there is an exact sequence 
$$0 \to \sE^\circ|_C \to \sF^\circ|_C \to \sG^\circ|_C \to 0.$$ 
It follows that
$$\mu^{\textup{max}}\big(\sE^\circ|_C\big) \le \mu(\sF|_C) \le \mu^{\textup{min}}\big(\sG^\circ|_C\big)$$
since $\sF|_C$ is semistable by assumption. This gives
$$\mu(\sF|_C) \le \mu^{\textup{min}}\big(\sG^\circ|_C\big)\le \mu^{\textup{max}}\big(R^1 f^\circ_*\sO_{X^\circ}|_C\big)+\mu^{\textup{max}}\big(\sE^\circ|_C\big) \le \mu^{\textup{max}}\big(\sE^\circ|_C\big) \le \mu(\sF|_C),$$
and hence 
$$\mu^{\textup{max}}\big(R^1 f^\circ_*\sO_{X^\circ}|_C\big) = 0 \quad \textup{and}\quad \mu^{\textup{max}}\big(\sE^\circ|_C\big) = \mu(\sF|_C) = \mu^{\textup{min}}\big(\sG^\circ|_C\big).$$
This implies that the vector bundle $f_*\Omega^1_{X/Y}|_C$ is numerically flat. In particular,
we have $\deg (f_*\sO(K_{X/Y})|_C) = 0$.
Then \cite[Corollary 2.3]{ueno_iitaka_conj_av} applies to show that $f$ is isotrivial. This finishes the proof of the lemma.
\end{proof}

The following is the main result of this section.

\begin{thm}\label{thm:case_conormal_nef}
Setting and notation as in \ref{setup:main}. Suppose in addition that the conormal bundle $\sN^*$ of $\sF$ is nef.
Then one of the following holds. 
\begin{enumerate}
\item There exists an abelian variety $A$ as well as a finite \'etale cover $\gamma\colon A \to X$ such that 
$\gamma^{-1}\sF$ is a linear foliation. 
\item There exists an abelian scheme $f\colon B \to C$ onto a smooth complete curve of genus at least $2$ as well as a finite \'etale cover $\gamma\colon B \to X$ such that $\gamma^{-1}\sF$ is induced by $f$.
\item The variety $X$ is of general type, and $\nu(K_\sF)=\dim X$. 
\end{enumerate}
\end{thm}

\begin{proof}By Lemma \ref{lemma:minimal}, $\sF^*$ is a nef vector bundle. This in turn implies that $\Omega_X^1$ is nef since $\sN^*$ is nef by assumption. In particular, $K_X$ is nef.
On the other hand, $K_X$ is abundant by Proposition \ref{prop:abundance}. Then 
\cite[Theorem 1.1]{kawamata} applies to show that $K_X$ is semiample. Moreover, by \cite[Theorem 1.2]{hoering_nef_cotan}, replacing $X$ by a finite \'etale cover, if necessary, we may assume that $X$ is an abelian scheme over a smooth projective variety $Y$ with ample canonical divisor (see Lemma \ref{lemma:finite_etale_cover}). Let us denote by $f\colon X \to Y$ the natural morphism and by $\Sigma$ any section of $f$. Set $m:= \dim Y$.

\medskip

Suppose that $m=0$. Then $X$ is an abelian variety. This easily implies 
$\sN^* \cong\sO_X$. As a consequence, $\sF$ is a linear foliation on $X$, and we are in case (1) of Theorem \ref{thm:case_conormal_nef}.

\medskip

Suppose from now on that $m\ge 1$.

\medskip

Set $\sE:=f_*\Omega^1_{X/Y}$ so that 
$\Omega^1_{X/Y} \cong f^*\sE$. By \cite[Theorem 1 and Remark 2]{viehweg_zuo_arakelov}, we have 
\begin{equation}\label{eq:arakelov_inequality}
2 \mu(\sE) \le \mu(\Omega^1_Y),
\end{equation}
where $\mu$ denotes the slope with respect to $K_Y$, which is ample by construction. 
A celebrated theorem proved by Aubin and Yau moreover asserts that $Y$ is K\"ahler-Einstein. As a consequence, the tangent bundle $T_Y$ is polystable with respect to $K_Y$ by a result proved by Kobayashi and L\"ubke independently.

The short exact sequences 
$$0\to \sN^* \to \Omega^1_X \to \sF^* \to 0\quad \textup{and} \quad 0 \to f^*\Omega^1_Y \to \Omega^1_X \to \Omega^1_{X/Y} \to 0$$
give
\begin{equation}\label{eq:slopes}
\frac{\mu(\sN^*|_\Sigma)+(n-1)\mu(\sF^*|_\Sigma)}{n}=\mu(\Omega^1_X|_\Sigma)=\frac{m\mu(\Omega^1_Y)+(n-m)\mu(\sE)}{n}.
\end{equation}

\medskip

Suppose first that the composition $\sN^*|_\Sigma \to \Omega^1_X|_\Sigma \to \Omega^1_{X/Y}|_\Sigma \cong \sE$ vanishes. Then $$\sN^*|_\Sigma \subseteq f^*\Omega^1_Y|_\Sigma\cong \Omega^1_Y$$ and there exists a surjective morphism $$\sF^*|_\Sigma \twoheadrightarrow f^*\sE|_\Sigma\cong \sE.$$ 
Moreover, by Lemma \ref{lemma:numerically_projectively_flat_elementary properties}, the vector bundle $\sF|_\Sigma$ is numerically projectively flat. It follows that
$$\mu(\sN^*|_\Sigma) \le \mu(\Omega^1_Y) \quad \textup{and} \quad \mu(\sF^*|_\Sigma) \le \mu(\sE).$$
Then Equation \eqref{eq:slopes} yields
$$m\mu(\Omega^1_Y)+(n-m)\mu(\sE) = \mu(\sN^*|_\Sigma)+(n-1)\mu(\sF^*|_\Sigma) \le \mu(\Omega^1_Y) + (n-1)\mu(\sE),$$
and hence, using Equation \eqref{eq:arakelov_inequality}, we obtain
$$(m-1)\mu(\Omega^1_Y) \le (m-1)\mu(\sE) \le \frac{m-1}{2}\mu(\Omega^1_Y).$$
This immediately implies $m=1$ since $\mu(\Omega^1_Y)>0$. Moreover, we have $\sN^*=f^*\Omega^1_Y \subset \Omega_X^1$. As a consequence, $\sF$ is induced by $f$, and we are in case (2) of Theorem \ref{thm:case_conormal_nef}. 

\medskip

Suppose finally that the composition $\sN^*|_\Sigma \to \Omega^1_X|_\Sigma \to \Omega^1_{X/Y}|_\Sigma \cong \sE$ is non-zero. Then the composition $\sF|_\Sigma \to T_X|_\Sigma \to f^*T_Y |_\Sigma \cong T_Y$ is generically surjective. Since 
the vector bundle $\sF|_\Sigma$ is numerically projectively flat, we have
\begin{equation}\label{eq:semistability}
\mu(\Omega^1_Y) \le \mu(\sF^*|_\Sigma).
\end{equation}
Together with Equations \eqref{eq:arakelov_inequality} 
and \eqref{eq:slopes} we obtain
\begin{multline*}
2\mu(\sN^*|_\Sigma)+2(n-1)\mu(\Omega^1_Y) \le 2\mu(\sN^*|_\Sigma)+2(n-1)\mu(\sF^*|_\Sigma)=
\\
2m\mu(\Omega^1_Y)+2(n-m)\mu(\sE)\le 2m\mu(\Omega^1_Y)+(n-m)\mu(\Omega^1_Y)=(n+m)\mu(\Omega^1_Y),
\end{multline*}
and thus
$$(n-2-m)\mu(\Omega^1_Y)+2\mu(\sN^*|_\Sigma) \le 0.$$
Notice that $\mu(\sN^*|_\Sigma)\ge 0$ since $\sN^*$ is nef
and that $\mu(\Omega^1_Y)>0$ since $m\ge 1$. In particular, $m \in\{n-2,n-1,n\}$.

\begin{claim}
There exists a line bundle $\sM$ on $Y$ such that $\sN^*\cong f^*\sM$.  
\end{claim}

\begin{proof}
Let $F$ be a general fiber of $f$. The composition $\sN^* \to \Omega^1_X \to \Omega^1_{X/Y}$ is non-zero by our current assumption and induces an injective map $\sN^*|_F \to \Omega^1_{X/Y}|_F\cong \sO_F^{\oplus n-m}$. Since $\sN^*$ is nef, we have $\sN^*|_F\cong\sO_F$. 
The claim now follows from Lemma \ref{lemma:pull_back}.
\end{proof}

Suppose that $m=n$. Then $X=Y$ is of general type and $\nu(K_\sF)=\dim X$ by Lemma \ref{lem:ambient_general_type_foliation_general_type} below. We are in case (3) of Theorem \ref{thm:case_conormal_nef}. 

\medskip

Suppose that $m \in\{n-2,n-1\}$. By Lemma \ref{lemma:isotriviality}, after replacing $X$ by a finite \'etale cover, we may assume that $X \cong Y \times A$ for some abelian variety $A$ of dimension $n-m\in\{1,2\}$ and that $f$ is induced by the projection onto $Y$. Since the composition $\sN^*|_\Sigma \to \Omega^1_X|_\Sigma \to \Omega^1_{X/Y}|_\Sigma \cong \sO_\Sigma^{\oplus n-m}$ is non-zero and $\sN^*$ is nef, we have $\sN\cong \sO_X$.

If $m=n-1$, then $\sF$ is everywhere transverse to $f$. This implies that $T_Y$ is numerically projectively flat. But then $K_Y\equiv 0$ by \cite[Theorem 0.1]{jahnke_radloff_13}, yielding a contradiction.

Suppose finally that $m=n-2$. Then Equations \eqref{eq:arakelov_inequality}, 
\eqref{eq:slopes} and \eqref{eq:semistability} give $\mu(\sF^*|_\Sigma) = \mu(\Omega^1_Y)$ and
$2 \mu(\sE) = \mu(\Omega^1_Y)$. Now this is impossible since $\mu(\sE)=0$ and $\mu(\Omega^1_Y)>0$. This finishes the proof of the theorem.
\end{proof}

\begin{lemma}\label{lem:ambient_general_type_foliation_general_type}
Setting and notation as in \ref{setup:main}. Suppose in addition that $X$ is of general type. Then $\nu(K_\sF) = \dim X$.
\end{lemma}

\begin{proof}
We argue by contradiction and assume that $\nu(K_\sF)\le \dim X - 1$. By Corollary \ref{cor:ambient_general_type_conormal_nef}, $\sN^*$ is nef. It follows that $\Omega_X^1$ is nef since $\sF^*$ is also nef by Lemma \ref{lemma:minimal}. This implies that $X$ contains no rational curve, and hence $K_X$ is an ample divisor. A celebrated theorem proved independently to Aubin and Yau then asserts that $X$ is K\"ahler-Einstein. As a consequence, the tangent bundle $T_X$ is polystable with respect to $K_X$ by a result proved by Kobayashi and L\"ubke independently.

Set $n:=\dim X$. Then $$K_\sF^n=(K_X+c_1(\sN))^n=K_X^n+nK_X^{n-1}\cdot c_1(\sN)=0$$ since 
$K_\sF=K_X+c_1(\sN)$ by the adjunction formula and
$c_1(\sN)^2 = 0$ by Lemma \ref{lemma:BB_vanishing}. 
It follows that $\mu(\sN)=\mu(T_X)$, where $\mu$ denotes the slope with respect to $K_X$, and hence $T_X \cong \sF \oplus \sN$ since $T_X$ is polystable.

By Lemma \ref{lemma:Miyaoka_Yau} (1) below, we then have 
$$(2nc_2(\sF)-(n-1)c_1(\sF)^2)\cdot K_X^{n-2} \ge 0.$$
Since $c_2(\sF)=\frac{n-2}{2(n-1)}c_1(\sF)$ (see Lemma \ref{lemma:chern_classes}), we obtain
$$c_1(\sF)^2\cdot K_X^{n-2}=0.$$
But this contradicts Lemma \ref{lemma:Miyaoka_Yau} (2) below, finishing the proof of the lemma.
\end{proof}

\begin{lemma}\label{lemma:Miyaoka_Yau}
Let $X$ be a complex projective manifold with $K_X$ ample, and let $\Omega^1_X:=\sA\oplus\sB$ be a decomposition of $\Omega^1_X$ into locally free subsheaves. Suppose that $\sA$ has rank $a \ge 2$, and set $n:=\dim X$. 
Then
\begin{enumerate}
\item $(2(a+1)c_2(\sA)-a c_1(\sA)^2)\cdot K_X^{n-2} \ge 0$, and
\item $c_1(\sA)^2\cdot K_X^{n-2} > 0$.
\end{enumerate}
\end{lemma}

\begin{proof}
Let $\sE:=\Omega_X^1\oplus \sO_X$ equipped with the Higgs field $\vartheta$ defined by
$$\Omega_X^1\oplus \sO_X \ni \alpha+f \mapsto (0 + 1)\otimes \alpha \in (\Omega_X^1\oplus \sO_X)\otimes\Omega_X^1.$$
Notice that $\vartheta(\sA\oplus \sO_X) \subset (\sA\oplus \sO_X)\otimes\Omega_X^1$ so that $(\sG:=\sA\oplus \sO_X,\vartheta|_{\sG})$ is a Higgs bundle as well.

Now, a theorem proved independently to Aubin and Yau then asserts that $X$ is K\"ahler-Einstein. As a consequence, the tangent bundle $T_X$ is polystable with respect to $K_X$ by a result proved by Kobayashi and L\"ubke independently. This in turn implies that $\sA$ is polystable with respect to $K_X$ with $\mu(\sA)=\mu(\Omega_X^1)>0$, where $\mu$ denotes the slope with respect to $K_X$.

Next we show that the Higgs bundle $(\sG,\vartheta|_\sG)$ is stable with respect to $K_X$. Let $\sH \subset \sG$ be a saturated subsheaf of rank $0 < r < a+1$ preserved by $\vartheta|_\sG$, and let $p\colon \sH \to \sO_X$ be the composition $\sH \to \sG=\sA \oplus \sO_X \to \sO_X$. 
Observe that $p$ is not the zero map since $\sH$ is preserved by $\vartheta|_\sG$ by assumption. 

If $p$ is injective, then $\mu(\sH) \le 0 < \mu(\sG)=\frac{a}{a+1}\mu(\sA)$. Suppose that $p$ is not injective. Note that we must have $r\ge 2$ and $\mu(\textup{Im}\, p) \le 0$. Moreover, 
$\mu(\textup{Ker}\,p)\le \mu(\sA)$ since $\textup{Ker}\, p \subseteq \sA$ and 
$\sA$ is polystable with respect to $K_X$. Then, we have
$$\mu(\sH) =\frac{(r-1)\mu(\textup{Ker}\,p)+\mu(\textup{Im}\, p)}{r} \le 
\frac{r-1}{r}\mu(\sA)=\frac{r-1}{r}\frac{a+1}{a}\mu(\sG)\le \frac{a^2-1}{a^2}\mu(\sG)<\mu(\sG),$$
proving our claim. 

Item (1) now follows from \cite[Theorem 1 and Proposition 3.4]{simpson_uniformization}. 

Let $b \ge 0$ denote the rank of $\sB$. By \cite[Lemma 3.1]{beauville_split}, we have
$c_1(\sA)^{a+1}\equiv 0$ and $c_1(\sB)^{b+1}\equiv 0$. Then
$$c_1(\sA)^2\cdot K_X^{n-2}  =  c_1(\sA)^2\cdot (c_1(\sA)+c_1(\sB))^{n-2}  = \binom{n-2}{b} c_1(\sA)^a\cdot c_1(\sB)^b,$$
and
$$K_X^n = \binom{n}{b} c_1(\sA)^a\cdot c_1(\sB)^b.$$
It follows that $c_1(\sA)^2)\cdot K_X^{n-2} > 0$ since $K_X^n >0$ by assumption, proving Item (2).
\end{proof}

Kodaira fibrations provide examples of regular foliations by curves $\sF$ on surfaces of general type with $\kappa(\sF)=2$. However, no example of foliations satisfying the conclusion of Theorem \ref{thm:case_conormal_nef}~(3) is known to the author. Let us show that these foliations have no compact leaves.

\begin{prop}
Setting and notation as in \ref{setup:main}. Suppose in addition that $K_X$ is pseudo-effective and that $\sF$ has a compact leaf. Then $\kappa(X)=\nu(X)\in \{0,1\}$.
\end{prop}

\begin{proof}
By Proposition \ref{prop:abundance}, $K_X$ is abundant. Let $B$ denote an algebraic leaf of $\sF$. Then $B$ is a finite \'etale quotient of an abelian variety by \cite[Theorem 0.1]{jahnke_radloff_13}. Moreover, $\sN|_B\cong \sO_X(B)|_B$ is a flat vector bundle by Lemma \ref{lemma:BB_vanishing}. This immediately implies that $B^2 \equiv 0$. Proposition \ref{prop:compact_leaf} below then applies to show that $\nu(X)\le 1$, completing the proof of the proposition.
\end{proof}

We refer to \cite{kollar_mori} and \cite{kmm} for standard references concerning the minimal model program.

\begin{prop}\label{prop:compact_leaf}
Let $X$ be a complex projective manifold of dimension $n \ge 2$ with $\kappa(X)=\nu(X)\ge 0$. Suppose that there exists a smooth prime divisor $B \subset X$ with $K_B \equiv 0$ and $B^2\equiv 0$. Then $\nu(X)\le 1$.
\end{prop}

\begin{proof}
By \cite[Theorem 4.3]{GL}, we may run a minimal model program for $X$ and end with a minimal model. Hence, there exists a birational map $\phi\colon X \map Y$, where $Y$ is a normal $\mathbb{Q}$-factorial projective variety with $K_Y$ nef. Moreover, $Y$ has terminal singularities, and $\phi^{-1}$ does not contract any divisor. Let $p\colon Z \to X$ and $q\colon Z \to Y$ be a common resolution of singularities. Let $C$ be the strict transform of $B$ in $Z$.
If $B$ is contracted by $\phi$, then it is uniruled by \cite[Theorem 1]{kawamata_length}, yielding a contradiction since $K_B\equiv 0$ by assumption. Set $D:=\phi_*B$.

Let $\varepsilon\in\mathbb{Q}$ be a positive number such that $\phi$ is a MMP for $K_X+\varepsilon B$. By \cite[Lemma 3.38]{kollar_mori}, there exists an effective $q$-exceptional $\mathbb{Q}$-divisor $F$ such that 
$$p^*(K_X+\varepsilon B) = q^* (K_Y+\varepsilon D) + F.$$
Moreover, by the negativity lemma, there exists an effective $q$-exceptional $\mathbb{Q}$-divisor $G$ such that
$$p^*B+G=q^*D$$ 
since $B$ is nef. Notice also that $K_X|_B\equiv 0$ by the adjunction formula $K_B=(K_X+B)|_B$.
We then have 
$$
0 \equiv p^*(K_X+\varepsilon B)|_C = q^*(K_Y+\varepsilon D)|_C + F|_C
\equiv (q^* K_Y)|_C + \varepsilon G|_C +F|_C.
$$
Since $G|_C +F|_C$ is effective and $(q^* K_Y)|_C$
is nef, we have 
$$K_Y|_D \equiv 0\quad\textup{and}\quad G \cap C =\emptyset.$$
This immediately implies $D^2\equiv 0$. Now, we have $\kappa(X)=\kappa(Y)$ and $\nu(X)=\nu(Y)$ by \cite[Proposition V.2.7]{nakayama04}, and hence $\kappa(Y)=\nu(Y)$. Then
\cite[Theorem 1.1]{kawamata} applies to show that $K_Y$ is semiample. We denote the Iitaka fibration of $Y$ by 
$f \colon Y \to T$. Since $K_Y|_D \equiv 0$, we have $\dim f(D)=0$. 
If $\dim T =0$, then $\nu(X)= 0$ and there is nothing to show. If $\dim T \ge 1$, then Lemma \ref{lemma:hodge_index_theorem}
applies to show $\nu(X)= 1$.
This finishes the proof of the proposition. 
\end{proof}

\section{The nef reduction map}\label{section:nef_reduction}

In this section we provide a technical tool for the proof of Theorem \ref{thm:case_normal_pseff_non_uniruled} below. We briefly recall the relevant notion first. Let $X$ be a smooth projective variety and let $\sL$ be a nef line bundle on $X$. Then there exists an almost proper, dominant rational map $f \colon X \map Y$ such that
\begin{enumerate}
\item $\sL|_F \equiv 0$ for a general fiber $F$ of $f$,
\item for every general point $x \in X$ and every irreducible curve $C$ passing through $x$ with $\dim f(C) = 1$, we have $\sL \cdot C > 0$. 
\end{enumerate}
The map $f$ is unique up to birational equivalence of $Y$. 
We refer to it as the \textit{nef reduction map} of $\sL$. The dimension of $Y$ is an invariant of $\sL$ called the \textit{nef dimension}. We write $\textup{nd}(\sL):=\dim Y$. We have $\kappa(\sL)\le \nu(\sL)\le \textup{nd}(\sL)$.
We refer to \cite{nef_reduction} for further explanations concerning these notions. Let $\sF$ be a regular foliation on a complex projective manifold and suppose that $\sF$ satisfies all the conditions listed in Setup \ref{setup:main}. Recall from Lemma \ref{lemma:minimal} that $K_\sF$ is nef. Then Proposition \ref{prop:nef_dimension} below gives sufficient conditions that guarantee $\textup{nd}(K_\sF)< \dim X$.

\begin{prop}\label{prop:nef_dimension}
Setting and notation as in \ref{setup:main}. Suppose in addition that $K_X$ is nef and that $\sN$ is pseudo-effective. Set $m:=\min(\nu(K_X),n-2)$. Then 
$\nu(K_X) \le n-1$ and $\textup{nd}(K_\sF)\le m+1\le n-1$.
\end{prop}

\begin{proof}
Let us first show $\nu(K_X) \le n-1$. Let $H$ be an ample divisor on $X$. By \cite[Theorem B]{G_T_22}, we have 
$$\Big(c_2(X)-\frac{n}{2(n+1)}c_1(X)^2\Big)\cdot K_X^m\cdot H^{n-2-m} \ge 0.$$ 
We have
$$c_2(X)= c_2(\sF)+c_1(\sF)\cdot c_1(\sN)\quad\text{and}\quad c_1(X)^2\equiv 
c_1(\sF)^2+2c_1(\sF)\cdot c_1(\sN)$$
since $c_1(\sN)^2\equiv 0$ by Lemma \ref{lemma:BB_vanishing}. In additon, we have 
$$c_2(\sF)\equiv \frac{n-2}{2(n-1)}c_1(\sF)^2$$
by \cite[Proposition 1.1]{jahnke_radloff_13}.
We obtain
\begin{equation}\label{eq:MY}
\Big(-\frac{1}{n^2-1}c_1(\sF)^2+\frac{1}{n+1}c_1(\sF)\cdot c_1(\sN)\Big)\cdot K_X^m\cdot H^{n-2-m} \ge 0.
\end{equation}
On the other hand, $K_\sF=-c_1(\sF)$ is nef by Lemma \ref{lemma:minimal}. This immediately implies
$$-c_1(\sF)^2 \cdot K_X^m\cdot H^{n-2-m} \le 0 $$
and 
$$c_1(\sF)\cdot c_1(\sN)\cdot K_X^m\cdot H^{n-2-m} \le 0$$
since $\sN$ is pseudo-effective by assumption. Equation \eqref{eq:MY} then gives 
\begin{equation}\label{eq:vanshing1}
c_1(\sF)^2 \cdot K_X^m\cdot H^{n-2-m}=c_1(\sF)\cdot c_1(\sN)\cdot K_X^m\cdot H^{n-2-m}=0.
\end{equation}
Therefore, we have $K_X^{m+2}\cdot H^{n-2-m}=0$, and hence $K_X^{n}=0$ since $m=\min(\nu(K_X),n-2)$. This shows $\nu(K_X)\le n-1$. 

\medskip

Let us now show $\textup{nd}(K_\sF) \le m+1$. 
If $\sF$ is algebraically integrable, then $\textup{nd}(K_\sF) \le 1$ by Proposition \ref{prop:special_case_algebraically_integrable}. Suppose from now on that $\sF$ is not algebraically integrable.

By Proposition \ref{prop:abundance} and \cite[Theorem 1.1]{kawamata} together, $K_X$ is semiample. We denote the Iitaka fibration by 
$f\colon X \to Y$. Let $F$ be a general fiber of $f$. Then $F$ is smooth. By construction, there exists an ample divisor $A$ on $Y$ such that $K_X \sim_\mathbb{Q} f^*A$. In particular, we have $K_F \sim_\mathbb{Q} 0$ by the adjunction formula. 

\medskip

Suppose first that $m = \dim Y = \nu(K_X) \le n-2$. 

\medskip

Then $c_2(T_X|_F)\cdot H|_F^{n-2-m}=\frac{1}{A^m} c_2(X)\cdot K_X^m\cdot H^{n-2-m}=0$ by Equation \eqref{eq:vanshing1}. The exact sequence
$$0 \to T_F\to T_X|_F\to \sN_{F/X}\cong \sO_F^{\oplus m }\to 0$$
then gives $c_2(T_F)\cdot H|_F^{n-2-m}=0$. 
As a classical consequence of Yau's theorem on the existence of a K\"ahler-Einstein metric, $F$ is then covered by a complex torus (see \cite[Chapter IV Corollary 4.15]{kobayashi_diff_geom_vb}). Thus, there exist an abelian variety $F_1$ of dimension at least $2$
and a finite \'etale cover $\eta\colon F_1 \to F$. Then \cite[Proposition 3.3]{GKP_proj_flat_JEP}
applies to show that there exists a line bundle $\sL_1$ on $F_1$ as well as a flat vector bundle $\sG_1$ on $F_1$ such that $\eta^*(\sF|_F) \cong \sG_1 \otimes\sL_1$. In addition, $\sG_1$ admits a filtration $0=F_0\sG_1 \subsetneq F_1\sG_1 \subsetneq \cdots \subsetneq F_{n-1}\sG_1=\sG_1$
by flat subbundles $F_i\sG_1$ with $\textup{rank}\, F_i\sG_1 = i$. By Equation \eqref{eq:vanshing1}, we have

\begin{equation}\label{eq:vanshing2}
c_1(\eta^*(\sF|_F))^2\cdot \eta^*(H|_F)^{n-2-m} = 0.
\end{equation}

If $K_\sF|_F\equiv 0$, then $\textup{nd}(K_\sF) \le \dim \textup{Iit}(X) = m \le m+1$. 

\medskip

Suppose that $K_\sF|_F\not\equiv 0$.

\medskip

If $m=0$, then $F=X$ and there exists a generically surjective morphism
$\Omega^1_{F_1} \to \eta^*(\sF|_F)^*$. This immediately implies $h^0(F_1,\sO_{F_1}(\eta^*(K_\sF|_F)))\ge 1$. Then \cite[Th\'eor\`eme VI.5.1]{debarre_cours_spe} applies to show that 
the line bundle $\sO_{F_1}(\eta^*(K_\sF|_F))$ is semiample. Notice that
$\nu(\sO_{F_1}(\eta^*(K_\sF|_F))) = 1$ by Equation \eqref{eq:vanshing2}. This easily implies $\textup{nd}(K_\sF) =1 = m+1$.

Suppose that $m\ge 1$. Since $\sF$ is not algebraically integrable by our current assumption, the composition $\sF|_F \to T_X|_F \to \sN_{F/X}\cong\sO_F^{\oplus m}$ is non-zero. As a consequence, there is an index $1 \le i \le n-1$
such that the line bundle $\sQ_i:=\big((F_{i+1}\sG_1/F_i\sG_1)\otimes \sL_1\big)^*$ has a non-zero global section. Then \cite[Th\'eor\`eme VI.5.1]{debarre_cours_spe} again applies to show  that $\sQ_i$ is semiample. We have $\nu(\sQ_i)=1$ since
$c_1(\sQ_i) \equiv c_1(\sL_1) \equiv \frac{1}{n-1}\eta^*(K_\sF|_F)$. This implies $\textup{nd}(K_\sF) \le \dim Y+1 = m+1$.

\medskip

Suppose finally that $\nu(K_X) = n-1$. In particular, $m=n-2$.

\medskip

Then $F$ is a smooth curve of genus $1$. By Equation \eqref{eq:vanshing1}, we have
$$K_\sF\cdot F = \frac{1}{A^{n-1}} K_\sF\cdot K_X^{n-1}= \frac{1}{A^{n-1}}c_1(\sF)\cdot (c_1(\sF)+c_1(\sN)) \cdot K_X^{n-2}=0.$$
As before, this implies $\textup{nd}(K_\sF) \le n-1 = m+1$, completing the proof of Proposition \ref{prop:nef_dimension}. 
\end{proof}

The following is the main result of this section.

\begin{prop}\label{prop:nef_reduction_map}
Setting and notation as in \ref{setup:main}. 
Suppose in addition that $K_X$ is nef and that $\sN$ is pseudo-effective.
Then, there exist an abelian scheme $f_1 \colon X_1 \to Y_1$ over a smooth projective base and a finite \'etale cover $\gamma \colon X_1 \to X$ such that $f_1$ is the nef reduction map of $\gamma^*K_\sF$. Moreover, there exist line bundles $\sM_1$ and $\sK_1$ on $Y_1$ such that $\gamma^*\sN \cong f_1^*\sM_1$ and $\gamma^*\sO_X(K_\sK) \cong f_1^*\sK_1$. In addition, if $F_1$ is a general fiber of $f_1$, then $(\gamma^{-1}\sF)|_{F_1} \cap T_{F_1}$ is a linear foliation on $F_1$ of codimension at most $1$.  
\end{prop}

\begin{proof} For the reader’s convenience, the proof is subdivided into a number of steps. Let 
$$f \colon X \map Y$$ be the nef reduction map of $K_\sF$. By Proposition \ref{prop:nef_dimension} above, we have $m:=\dim Y \le n-1$.

\medskip

\noindent\textit{Step 1. Fibers of the nef reduction map.} 

\medskip

Let us show that a general fiber $F$ of $f$ is a finite \'etale quotient of an abelian variety and that $K_\sF|_F \sim_\mathbb{Q}0$. Then $\sN|_F$ is a torsion line bundle by the adjunction formula $K_\sF=K_X+c_1(\sN)$.

\medskip

Suppose first that $F$ is a very general fiber. Then $\sN|_F$ is pseudo-effective while $K_X|_F$ is nef and $K_\sF|_F\equiv 0$ by construction. It follows that $K_X|_F\equiv 0$ and $\sN|_F\equiv 0$ since $K_\sF=K_X+c_1(\sN)$ by the adjunction formula. Let now $F$ be a general fiber of $f$. By \cite[Lemma 2.13]{druel_guenancia}, we have $K_X|_F\equiv 0$. By the adjunction formula again, we have $\sN|_F\equiv 0$. 
Now, the vector bundle $\sF|_F$ is numerically flat by Theorem \ref{thm:numerically_projectively_flat} since $K_\sF|_F\equiv 0$. It follows that the composition
$$\sF|_F \to T_X|_F \to \sN_{F/X}\cong \sO_F^{\oplus m}$$ has constant rank, and hence
$\sF|_F\cap T_F$ is a numerically flat subbundle of $T_F$.

If $T_F \subseteq \sF|_F$, then $F$ is a finite \'etale quotient of an abelian variety as a classical consequence of Yau's theorem on the existence of a K\"ahler-Einstein metric (see \cite[Chapter IV Corollary 4.15]{kobayashi_diff_geom_vb}). 
Moreover, $K_\sF|_F \sim_\mathbb{Q}0$ since $\sF|_F/ T_F \subset \sN_{F/X}\cong \sO_F^{\oplus m}$ is numerically flat.

Suppose that $\sF|_F\cap T_F \subsetneq T_F$. Then $F$ is again a finite \'etale quotient of an abelian variety by Proposition \ref{prop:classification_zero_canonical_class} since $K_F\equiv K_X|_F\equiv 0$ by the adjunction formula. Moreover, 
$K_{\sF|_F\cap T_F}\sim_\mathbb{Q} 0$, and hence $K_\sF|_F \sim_\mathbb{Q}0$ since 
$\sF|_F/(\sF|_F\cap T_F) \cong \sN_{F/X}\cong \sO_F^{\oplus m}$. This proves our claim.

\medskip

\noindent\textit{Step 2. Next, we show that $X$ has generically large fundamental group on $F$.} 

\medskip

We argue by contradiction and assume that $X$ does not have generically large fundamental group on $F$. There exists a positive-dimensional finite \'etale quotient of an abelian variety $G \subseteq F$ passing through a general point of $F$ such that the image of the natural morphism $\pi_1(G) \to \pi_1(X)$ is finite. By \cite[3.10]{kollar_sh_inventiones}, the algebraic varieties $G$ fit together to form an almost proper rational map $g \colon X \map Z$ such that $f$ factorizes through $g$. Let $X^\circ \subseteq X$ and $Z^\circ \subseteq Z$ be dense open sets such that the restriction $g^\circ\colon X^\circ \to Z^\circ \subseteq Z$ of $g$ to $X^\circ$ is a smooth projective morphism with connected fibers. Shrinking $Z^\circ$, we may assume that there exists a finite \'etale cover $X_1^\circ \to X^\circ$ such that the fibration $g_1^\circ\colon X_1^\circ \to Z_1^\circ$ obtained as the Stein factorization of the composition $X_1^\circ \to X^\circ \to Z^\circ$ is an abelian scheme equipped with a level three structure (see Lemma \ref{lemma:smooth_deformation_torus_quotient}). We obtain a diagram as follows:
\begin{center}
\begin{tikzcd}
  X_1^\circ \ar[d, "{g_1^\circ}"']\ar[r] &  X^\circ \ar[r, hook]\ar[d, "{g^\circ}"']  & X \ar[d, dashed, "g"'] \ar[dr, dashed, "f", bend left] & \\
  Z_1^\circ \ar[r] & Z^\circ \ar[r, hook] & Z \ar[r, dashed] & Y.  
\end{tikzcd}
\end{center}

Let $G_1$ be a fiber of $g_1^\circ$. There is a split exact sequence of fundamental groups
\begin{center}
\begin{tikzcd}[column sep=small]
  1  \ar[r] & \pi_1(G_1) \ar[r] & \pi_1(X_1^\circ) \ar[r] & \pi_1(Z_1^\circ) \ar[r]\ar[l, bend right=30] & 1.
\end{tikzcd}
\end{center}
Let $K \subseteq \pi_1(G_1)$ denotes the kernel of the representation
$$\pi_1(G_1) \to \pi_1(X_1^\circ) \to \pi_1(X) \to \textup{PGL}(n-1,\mathbb{C})$$
induced by $\rho$. Then the subset $K \cdot \pi_1(Z_1^\circ)$ of 
$\pi_1(X_1^\circ)$ is a (normal) subgroup. Moreover, it has finite index since the image of the composition $\pi_1(G_1) \to \pi_1(X_1^\circ) \to \pi_1(X)$
is finite by our current assumption. Therefore, replacing $X_1^\circ$ by a further finite \'etale cover, if necessary, we may assume that the representation  
$$\pi_1(X_1^\circ) \to \pi_1(X) \to \textup{PGL}(n-1,\mathbb{C})$$
induced by $\rho$ factorizes through $Z_1^\circ$. Shrinking $Z_1^\circ$, if necessary, we may assume in addition that $\sF|_{X_1^\circ}\cong \sL_{X_1^\circ}^{\oplus n-1}$
for some line bundle $\sL_{X_1^\circ}$ on $X_1^\circ$.
Notice that by construction, $\sL_{X_1^\circ}|_{G_1} \equiv 0$. The short exact sequence
$$0  \to T_{G_1} \cong \sO_{G_1}^{\oplus n-m} \to T_{X_1^\circ}|_{G_1} \to \sN_{G_1/X_1^\circ}\cong \sO_{G_1}^{\oplus m} \to 0$$ 
then shows that $\sL_{X_1^\circ}|_{G_1}\cong \sO_{G_1}$. Therefore, shrinking $Z^\circ$, we may assume that $\sF|_{X_1^\circ}\cong \sO_{X_1^\circ}^{\oplus n-1}$.

\begin{claim}\label{claim:triviality}
The fibration $g_1^\circ$ is locally trivial for the analytic topology, and hence $X_1^\circ/Z_1^\circ\cong Z_1^\circ\times A/Z_1^\circ$ for some positive-dimensional abelian variety $A$.
\end{claim}

\begin{proof}[Proof of Claim \ref{claim:triviality}]

$\ $

\medskip

Suppose first that $T_{G}\subseteq \sF|_{G}$ for a general fiber $G$ of $g$. 

\medskip

Then, there exists a regular foliation $\sG^\circ$ on $Z^\circ$ such that $\sF|_{X^\circ}=(g^\circ)^{-1}\sG^\circ$. Let $L_1$ be a leaf of $\sG_1^\circ:=\sG^\circ|_{Z_1^\circ}$ and set $M_1:=(g_1^\circ)^{-1}(L_1)$. Notice that $M_1$ is a leaf of $\sF|_{X_1^\circ}$.
By construction, there is an exact sequence:
$$0 \to T_{M_1/L_1} \to \sF|_{M_1}\cong T_{M_1} \to (g_1^\circ|_{M_1})^*T_{L_1}\to 0.$$
Since $\sF|_{X_1^\circ}\cong \sO_{X_1^\circ}^{\oplus n-1}$, the above exact sequence 
is locally split. A classical result of complex analysis says that complex flows of vector fields on analytic spaces exist. As a consequence, the fibration $g_1^\circ|_{M_1}\colon M_1 \to L_1$ is locally trivial for the analytic topology. Let $\sA(3)$ be the fine moduli space of polarized abelian varieties with a level three structure, and let $Z_1^\circ \to \sA(3)$ be the morphism corresponding to $g_1^\circ$. Then every leaf of $\sG_1^\circ$ is contained in a fiber of $Z_1^\circ \to \sA(3)$. Thus, either $\sG_1^\circ$ has algebraic leaves, or the morphism $Z_1^\circ \to \sA(3)$ is constant. In the former case, 
$\sF$ is algebraically integrable. By Proposition \ref{prop:special_case_algebraically_integrable}, we must have $K_\sF\equiv 0$ and $F=X$ is a finite \'etale quotient of an abelian variety since $\sN$ is pseudo-effective by assumption. In particular, $X$ has generically large fundamental group on $F$, yielding a contradiction. In the latter case, $X_1^\circ/Z_1^\circ \cong Z_1^\circ\times A/Z_1^\circ$ for some positive-dimensional abelian variety $A$, proving the claim in this case.

\medskip

Suppose now that $T_{G}\not\subseteq \sF|_{G}$ for a general fiber $G$ of $g$.

\medskip

Then the composition
$$\sF|_{G_1}\cong \sO_{G_1}^{\oplus n-1} \to T_{X_1^\circ}|_{G_1} \to (g_1^\circ)^*T_{Z_1^\circ}|_{G_1}\cong \sO_{G_1}^{\oplus \dim Z^\circ}$$ is generically surjective, and hence surjective. This immediately implies that the exact sequence 
$$0 \to T_{X_1^\circ/Z_1^\circ} \to T_{X_1^\circ} \to (g_1^\circ)^*T_{Z_1^\circ} \to 0$$
is locally split. As before, it follows that the fibration $g_1^\circ$ is locally trivial for the analytic topology, and hence $X_1^\circ/Z_1^\circ\cong Z_1^\circ\times A/Z_1^\circ$
for some positive-dimensional abelian variety $A$. This finishes the proof of Claim \ref{claim:triviality}.
\end{proof}

Let $W$ be the normalization of the closed subvariety of the Chow variety of $X$ whose general point parametrizes a general fiber of $g$, and let $U$ be the normalization of the universal cycle. By Claim \ref{claim:triviality}, the morphism $p$ is generically isotrivial. Hence, there exists a smooth projective variety $W_1$ as well as a generically finite morphism $\eta\colon W_1 \to W$ and a rational map $W_1\times A \map U$ fitting into the following commutative diagram:
\begin{center}
\begin{tikzcd}[row sep=large, column sep=huge]
  W_1\times A  \ar[d, "{\textup{pr}_{W_1}}"']\ar[r, dashed] & U \ar[d, "p"'] \ar[r, "q"] & X \ar[d, dashed, "{g}"]\\
  W_1\ar[r, "\eta"] & W \ar[r, dashed] & Z.   
\end{tikzcd}
\end{center}

\begin{claim}\label{claim:well-defined}
The map $W_1 \times A \map U$ is a morphism.
\end{claim}

\begin{proof}
Let $p_1\colon Z \to W_1\times A$ be a composition of a finite number of blow-ups with smooth centers such that the induced rational map $q_1\colon Z \to U$ is a morphism. Let $Z_{(w_1,a)}$ be a positive-dimensional fiber of $p_1$ over $(w_1,a)\in W_1\times A$. Then $Z_{(w_1,a)}$ is rationally chain connected, and the composition $$q_1(Z_{(w_1,a)}) \subseteq U_{\eta(w_1)} \to X$$ is finite, where $U_{\eta(w_1)}$ denotes the fiber of $p$ over the point $\eta(w_1)$. Thus, if $\dim q_1(Z_{(w_1,a)}) \ge 1$, then $U_{\eta(w_1)}$ contains a rational curve $C$ such that $\dim q(C) \ge 1$. By Corollary \ref{cor:numerically trivial}, we must have $q(C)\cdot K_\sF=0$. But this contradicts Lemma \ref{lemma:rational_curve_uniruled} since $K_X$ is nef by assumption. Therefore, every fiber of $p_1$ is contracted by $q_1$. Our claim now follows from the rigidity lemma (see \cite[Lemma 1.15]{debarre}). 
\end{proof}

A consequence of Claim \ref{claim:well-defined} is that every fiber of $p$ is irreducible. This implies that the exceptional locus $E$ of $q$ satisfies $E=p^{-1}(p(E))$. 
Set $X^\bullet:=X \setminus q(E)$ and $W^\bullet:=W \setminus p(E)$. Then the rational map $p\circ q^{-1}$ induces an equidimensional morphism $p^\bullet\colon X^\bullet \to W^\bullet$ with irreducible fibers. By construction, $X \setminus X^\bullet$ has codimesnion at leat $2$. Shrinking $W^\bullet$, if necessary, we may assume without loss of generality that $W^\bullet$ is smooth and that there exists a smooth divisor $D^\bullet$ on $W^\bullet$ such that the restriction of $p^\bullet$ to
$(p^\bullet)^{-1}(W^\bullet\setminus D^\bullet)$ is a smooth morphism. Notice that every smooth fiber of $p^\bullet$ is a finite \'etale quotient of an abelian variety by Lemma \ref{lemma:smooth_deformation_torus_quotient}.

\begin{claim}\label{claim:special_fibers}
Let $w\in D^\bullet$ be a general point. Then the fiber $(p^\bullet)^{-1}(w)$ (with its reduced structure) is a finite \'etale quotient of an abelian variety.
\end{claim}

\begin{proof}[Proof of Claim \ref{claim:special_fibers}]
The statement is local on $D^\bullet$, hence we may shrink $W^\bullet$ and assume that $D^\bullet$ is irreducible. Let $m$ denote the multiplicity of $(p^\bullet)^*D^\bullet$, and let $\gamma^\bullet\colon W_1^\bullet \to W^\circ $ be a quasi-finite morphism with $W_1^\bullet$ smooth such that $(\gamma^\bullet)^*D^\bullet = m D_1^\bullet$ with $D_1^\bullet$ smooth (and $D_1^\bullet\neq \emptyset$). We may also assume that 
$\gamma^\bullet$ is \'etale away from $\textup{Supp}\, D_1^\bullet$. Let $X_1^\bullet$ be the normalization of the product $W_1^\bullet \times_{W^\bullet} X^\bullet$. Then an easy local computation shows that the natural morphism $X_1^\bullet \to X^\bullet$ is \'etale away from a closed subset of codimension at least $2$, and hence \'etale by the Nagata-Zariski purity theorem. In particular, $X_1^\bullet$ is smooth. Shrinking $W_1^\bullet$ further, if necessary, we may assume that the natural morphism $g_1^\bullet \colon X_1^\bullet \to W_1^\bullet$ has generically reduced fibers. Observe also that its fibers are irreducible since $p$ has irreducible fibers and $X_1^\bullet \to X^\bullet$ is \'etale.

Let $w \in D_1^\bullet$ be a general point. By \cite[Proposition 18.13]{eisenbud}, the scheme theoretic fiber $G_w:=(g_1^\bullet)^{-1}(w)$ is Cohen-Macaulay, and hence reduced by Serre's criterion. Moreover, by the adjunction formula, $\omega_{G_w}\cong \omega_{X_1^\bullet}|_{G_w}$. In particular, 
$G_w$ is Gorenstein. By miracle flatness, the morphism $g_1^\bullet$ is flat (see \cite[Tag 00R4]{stacks-project}). Because the functions $w \mapsto h^0\big(G_w,\omega_{X_1^\bullet}^{\otimes \pm 1}|_{G_w}\big)$ are upper semicontinuous in the Zariski topology on $W_1^\bullet$ (see \cite[Theorem 12.8]{hartshorne77}), we have $\omega_{G_w}\cong\sO_{G_w}$.
Let $\nu_w\colon \wt{G}_w \to G_w$ be the normalization morphism. By Lemma \ref{lemma:conductor}, there exists an effective Weil divisor $E_w$ on $\wt{G}_w$ such that $\omega_{\wt{G}_w}\cong \nu_w^*\omega_{G_w}(-E_w)\cong \sO_{\wt{G}_w}(-E_w)$. Suppose that $E_w \neq 0$. Then $K_{\wt{G}_w}$ is not pseudo-effective and hence $G_w$ is uniruled by \cite[Corollary 0.3]{bdpp} applied to a resolution of $\wt{G}_w$. Let $C \subseteq G_w$ be a rational curve. By Corollary \ref{cor:numerically trivial}, we have $C\cdot K_\sF=0$. But this contradicts Lemma \ref{lemma:rational_curve_uniruled} since $K_X$ is nef. This shows that $E_w=0$, and hence $G_w$ is normal by Lemma \ref{lemma:conductor}. Then \cite[Corollary 0.3]{bdpp} applies to show that $G_w$ has canonical singularities since $K_{G_w}=0$.

Now, replacing $W_1^\bullet$ with a neighborhood of $w$ for the analytic topology, we may assume without loss of generality the $W_1^\bullet$ is biholomorphic to the open unit ball  
in $\mathbb{C}^{\,\dim W_1^\bullet}$. The same argument used in the proof of \cite[Lemma 5.2.2]{kollar_sh_inventiones}
shows that there is a surjective homomorphism
$$\pi_1(G_{w_1}) \twoheadrightarrow \pi_1(X_1^\bullet),$$
where $\omega_1 \in W_1^\bullet\setminus D_1^\bullet$ is any point.
This implies that the image of the composition
$$\pi_1(X_1^\bullet) \to \pi_1(X^\circ) \to \pi_1(X)$$
is finite by our current assumption. Thus, replacing $X_1^\bullet$ by a finite \'etale cover, if necessary, we may assume that there exists a line bundle $\sL_{X_1^\bullet}$ on $X_1^\bullet$ such that $\sF|_{X_1^\bullet} \cong \sL_{X_1^\bullet}^{\oplus n-1}$. By Step 1, the restriction of $\sL_{X_1^\bullet}$ to a general (smooth) fiber of $g_1^\bullet$ is a torsion line bundle. Moreover, shrinking $(W_1^\bullet,w)$ further, we may assume without loss of generality that $\sL_{X_1^\bullet}$ is a torsion line bundle as well (see Lemma \ref{lemma:pull_back}). Thus, replacing $X_1^\bullet$ by a finite \'etale cover, if necessary, we may therefore assume that 
$\sL_{X_1^\bullet}\cong \sO_{X_1^\bullet}$. 
It follows that the composition
$$\sF|_{G_w} \to T_{X_1^\bullet}|_{G_w} \to \sN_{G_w/X_1^\bullet}\cong \sO_{G_w}^{\oplus m}$$ has constant rank. This in turn implies that the reflexive sheaf $\sF|_{G_w}\cap T_{G_w}$ is a numerically flat locally free sheaf.

Suppose first that $T_{G_w}\subseteq \sF|_{G_w}$ for a general point $w\in D_1^\bullet$. Then $T_{G_w}$ is locally free. By the solution of the Zariski-Lipman conjecture for klt spaces (see \cite[Theorem 6.1]{greb_kebekus_kovacs_peternell10} or \cite[Theorem 1.1]{druel_zl}), we see that $G_w$ is smooth. The claim now follows from Lemma \ref{lemma:smooth_deformation_torus_quotient}.

Suppose finally that $T_{G_w}\not\subseteq \sF|_{G_w}$ for a general point $w\in D_1^\bullet$.
Then the composition
$$\sF|_{G_w} \to T_{X_1^\bullet}|_{G_w} \to (g_1^\bullet)^*T_{W_1^\bullet}|_{G_w}$$
is generically surjective, and hence surjective. This in turn implies that $g_1^\bullet$ is a locally trivial fibration for the analytic topology in a neighborhood of a general point in $D_1^\bullet$. In particular, $G_w$ is smooth. The claim now follows again from Lemma \ref{lemma:smooth_deformation_torus_quotient}. 
\end{proof}

By Claim \ref{claim:special_fibers}, replacing $W^\bullet$ by an open subset with complement of codimension at least $2$, we may assume that every fiber of $p^\bullet$
(with its reduced structure) is a finite \'etale quotient of an abelian variety. By Lemma \ref{lemma:injectivity_fundamental_groups} applied to $p^\bullet\colon X^\bullet \to W^\bullet$, $X^\bullet$ has generically large fundamental group on $G$. On the other hand, the inclusion induces an isomorphism $\pi_1(X^\bullet)\cong \pi_1(X)$ of fundamental groups since $X\setminus X^\bullet$ is a closed subset of codimension at least $2$ by construction. It follows that $X$ has generically large fundamental group on $G$, yielding a contradiction. This shows that $X$ has generically large fundamental group on $F$.

\medskip

\noindent\textit{Step 3. End of proof.} 

\medskip

By \cite[Theorem 6.3]{kollar_sh_inventiones} and Step 2 together, there exists a finite \'etale cover $\gamma\colon X_1 \to X$ such that the fibration $f_1\colon X_1 \map Y_1$ obtained as the Stein factorization of
the almost proper map $X_1 \to X \map Y$ is birational to an abelian group scheme $f_1 \colon X_2 \to Y_2$ over a projective base equipped with a level three structure.

Let $W_1$ be the normalization of the closed subvariety of the Chow variety of $X_1$ whose general point parametrizes a general fiber of $f_1$, and let $U_1$ be the normalization of the universal cycle. Replacing $Y_2$ by a birational modification, if necessary, we may assume that there exist a birational morphism $Y_2 \to W_1$ and a commutative diagram
\begin{center}
\begin{tikzcd}[row sep=large, column sep=huge]
  X_2  \ar[d, "{f_2}"']\ar[r, dashed] & U_1 \ar[d] \ar[r] & X_1 \ar[d, dashed, "f_1"]\\
  Y_2\ar[r] & W_1 \ar[r, dashed, "\textup{birational}"] & Y_1. &   
\end{tikzcd}
\end{center}
Arguing as in the proof of Claim \ref{claim:well-defined}, we see that rational map $X_2 \map U_1$ is a morphism. Then Lemma \ref{lemma:contraction_abelian_scheme} applies to show that $X_1$ is an abelian scheme 
$f_3\colon X_1 \to Y_3$
over a smooth projective base $Y_3$ and that there exists a birational morphism $Y_2 \to Y_3$ such that $X_2/Y_2 \cong Y_2 \times_{Y_3} X_1/Y_2$. This finishes the proof of first statement of Proposition \ref{prop:nef_reduction_map}.

\medskip

Finally, we show that there exist line bundles $\sM_3$ and $\sK_3$ on $Y_3$ such that $\gamma^*\sN \cong f_3^*\sM_3$ and $\gamma^*\sO_X(K_\sK) \cong f_3^*\sK_3$.
By Step 1, $K_\sF|_F \sim_\mathbb{Q}0$ and $\sN|_F$ is a torsion line bundle. Then Lemma \ref{lemma:pull_back} applies to show that there exists a positive integer $k$ such that    
$\gamma^*\sN^{\otimes k} \cong f_3^*\sM_3$ and $\gamma^*\sO_X(k K_\sK) \cong f_3^*\sK_3$, where 
$\sM_3$ and $\sK_3$ are line bundles on $Y_3$. Let $\Sigma_3 \subseteq X_1$ be the neutral section of $f_3$. Then $\sM_3 \cong (\gamma^*\sN|_{\Sigma_3})^{\otimes k}$ viewing 
$(\gamma^*\sN|_{\Sigma_3})^{\otimes k}$ as a line bundle on $Y_3$. Hence
$\gamma^*\sN^{\otimes k}\cong f_3^*(\gamma^*\sN|_{\Sigma_3})^{\otimes k}$. Replacing $X_1$ by the associated cyclic \'etale cover (see \cite[Definition 2.52]{kollar_mori}), we may assume without loss of generality that $k=1$. This completes the proof of the proposition.
\end{proof}

\begin{cor}
Setting and notation as in \ref{setup:main}. Suppose in addition that $X$ is of general type and $\dim X \ge 4$. Then $\sN^*$ is nef.
\end{cor}

\begin{proof}
By Theorem \ref{thm:suspension_complex_torus}, $K_X$ is nef. Then Proposition \ref{prop:nef_dimension} applies to show that $\sN$ is not pseudo-effecive. The claim now follows from Theorem \ref{thm:divisor_zero_square}.
\end{proof}

\section{Projectively flat foliations with pseudo-effective normal bundle}\label{section:pseff_normal}

\begin{thm}\label{thm:case_normal_pseff_non_uniruled}
Setting and notation as in \ref{setup:main}. Suppose in addition that $K_X$ is nef and that
$\sN$ is pseudo-effective. Then $K_\sF\equiv 0$.
\end{thm}

\begin{proof}
Let $f \colon X \map Y$ be the nef reduction map of $K_\sF$, and let $F$ be a general fiber of $f$. Set $m:= \dim Y$. Notice that $m \le n-1$ by Proposition \ref{prop:nef_dimension}. By Proposition \ref{prop:nef_reduction_map}, replacing $X$ by a finite \'etale cover, if necessary, we may assume without loss of generality that $f \colon X \to Y$ is an abelian group scheme over a smooth projective base equipped with a level three structure. We may also assume that $\sF|_F \cap T_F$ is a linear foliation of codimension at most $1$ and that there exist line bundles $\sM$ and $\sK$ on $Y$ such that $\sN\cong f^*\sM$ and $\sO_X(K_\sF)\cong f^*\sK$. Let $\Sigma\subset X$ be the neutral section of $f$. Notice that $\sK \cong \sO_X(K_\sF)|_\Sigma$.

\medskip

Notice that $m=0$ if and only if $K_\sF\equiv 0$. 

\medskip

Let us first show that $T_{X/Y}\not\subseteq \sF$. We argue by contradiction and assume that $T_{X/Y}\subseteq \sF$. In particular, we have $m \ge 1$.

\medskip

If $m =1$, then $\sF=T_{X/Y}$. In other words, $\sF$ is algebraically integrable. By Proposition \ref{prop:special_case_algebraically_integrable}, $X$ is a finite \'etale quotient of an abelian variety and $K_\sF\equiv 0$ since $\sN$ is pseudo-effective by assumption, yielding a contradiction.

Suppose now that $m \ge 2$. There is a regular foliation $\sG$ on $Y$ such that $\sF=f^{-1}\sG$ since $T_{X/Y}\subseteq \sF$ by assumption (see \cite[Lemma 6.7]{fano_fols}). Then $\sM \cong \sN_\sG$. Moreover, there is 
an exact sequence 
$$0 \to T_{X/Y} \to \sF \to f^*\sG \to 0.$$
In addition, the composition
$$T_{X/Y}|_\Sigma \subset \sF|_\Sigma \subset T_X|_\Sigma \to \sN_{\Sigma/X}$$
is an isomorphism, and hence $$\sF|_\Sigma\cong T_{X/Y}|_\Sigma\oplus \sG.$$ 
By Lemma \ref{lemma:numerically_projectively_flat_elementary properties}, the vector bundles $\sF|_\Sigma$ and $\sG$ are numerically projectively flat. Moreover, we have $$\frac{1}{n-1}c_1(\sF|_\Sigma)\equiv\frac{1}{d}c_1(T_{X/Y}|_\Sigma)\equiv 
\frac{1}{m-1}c_1(\sG).$$
Since $f$ is the nef reduction map of $K_\sF$,
we must have $\textup{nd}(\sK)=m$. It follows that the nef reduction map of $K_\sG\equiv\frac{m-1}{n-1}c_1(\sK)$ is the identity map. Notice that $\sN_\sG$ is pseudo-effective because $\sN$ is pseudo-effective by assumption.
\begin{claim}
The divisor $K_Y$ is nef.
\end{claim}

\begin{proof}
We argue by contradiction and assume that $K_Y$ is not nef. Then there is a rational curve $C$ on $Y$ such that $K_Y \cdot C <0$. Let $B \to C$ denotes the normalization of $C$, and set 
$S:=B \times_Y X$ and $\sigma:=B \times_Y \Sigma$. By \cite[Lemma 5.9.3]{kollar_sh_inventiones}, 
we have $(S/B,\sigma) \cong (B \times A/B,B\times 0_A)$
for some abelian variety $A$. Moreover, $K_{X/Y}|_S = K_{S/B} \sim_\mathbb{Z} 0$. This immediately implies that $K_X$ is not nef, yielding a contradiction. This proves the claim.
\end{proof}
Applying Proposition \ref{prop:nef_dimension} to $\sG$, we see that 
$m = 2$. 

\medskip

If $\kappa(Y)=2$, then $\sN_\sG^*$ is pseudo-effective by \cite[Lemme 9]{brunella_aens} and \cite[Corollaire 1]{brunella_aens}, and hence $\sN\equiv 0$ since $\sN \cong f^*\sN_\sG$
is pseudo-effective by assumption. By Theorem \ref{thm:case_conormal_nef} and Proposition \ref{prop:nef_dimension} together, we have $K_\sF\equiv 0$, yielding a contradiction.

From \cite[Section 5]{brunella_aens}, we see that either $Y$ is an elliptic surface or the minimal model of $Y$ is an abelian variety. In the latter case, $X$ is an abelian variety by \cite[Lemma 5.9.3]{kollar_sh_inventiones}. Then Lemma \ref{lemma:abelian_variety} applies to show that $K_\sF\equiv 0$. This again gives a contradiction. Suppose finally that $Y$ is an elliptic surface over a smooth connected curve $B$. By 
\cite[Proposition 2]{brunella_aens}, the fibers (with their reduced structure) of $Y \to B$ are smooth curves of genus $1$. Then, by \cite[Proposition 3]{brunella_aens}, we have $m=\textup{nu}(\sK)\le 1$. This again yields a contradiction.

\medskip

This shows that $T_{X/Y}\not\subseteq \sF$. 

\medskip

Then Lemma \ref{lemma:isotriviality} applies to show that, after replacing $X$ by a finite \'etale cover, we may assume that $X \cong Y \times A$ for some abelian variety $A$ of dimension and that $f$ is induced by the projection onto $Y$.

\medskip

Let us show that $m \le 1$. The composition
$$\sO_X^{\oplus n-m}\cong T_{X/Y} \to T_X \to \sN$$
is generically surjective, and hence $h^0(X,\sN)\ge 1$. Since $\sN\cong f^*\sM$, there exist prime divisors $D_i$ on $Y$ for $i \in I$ and positive integers $m_i \ge 1$ such that $\sN \cong \sO_X(\sum_{i\in I}m_i B_i)$, where $B_i:=f^*D_i$. 

Then $\sO_{B_i}^{\oplus d} \cong T_{X/Y}|_{B_i} \subseteq \sF|_{B_i}$. Since $\sF$ is numerically projectively flat and $K_\sF$ is nef by Lemma \ref{lemma:minimal} we see that $\sF|_{B_i}$ is numerically flat.

If $B_i$ is a leaf of $\sF$, then it is a finite \'etale quotient of an abelian variety by \cite[Theorem 0.1]{jahnke_radloff_13}. Moreover, $\sN|_{B_i}$ is a flat line bundle by Lemma \ref{lemma:BB_vanishing}, and hence $B_i^2\equiv 0$ since $\sN|_{B_i}\cong \sN_{B_i/X}$.

Suppose now that $B_i$ is not a leaf of $\sF$. Then the composition
$$\sF|_{B_i} \to T_X|_{B_i} \to \sN_{B_i/X}$$
is generically surjective. Since $\sF|_{B_i}$ is numerically flat, it follows that $B_i^2$ is numerically equivalent to an effective cycle.

Now, by Lemma \ref{lemma:BB_vanishing}, we have $\sN^2\equiv 0$. This immediately implies that $B_i^2 \equiv 0$ for every $i\in I$ and that $B_i \cap B_j=\emptyset$ if $i \neq j$.

If $\sN\equiv 0$, then $K_\sF\equiv 0$ by Theorem \ref{thm:case_conormal_nef} and Proposition \ref{prop:nef_dimension} together. Then $m=0$. Suppose from now on that $\sN\not\equiv 0$.

Let $i\in I$ such that $B_i$ is not a leaf of $\sF$. Let us show that $B_i$ is smooth with $K_{B_i}\equiv 0$.
The exact sequence
$$0 \to \sF|_{B_i} \to T_X|_{B_i} \to \sN|_{B_i} \to 0 $$
shows that $T_X|_{B_i}$ is numerically flat. In particular, $K_X|_{B_i}\equiv 0$, and hence 
$\omega_{B_i} \equiv 0$ by the adjunction formula. Let $\nu_i\colon \wt{B}_i \to B_i$ be the normalization morphism. By Lemma \ref{lemma:conductor}, there exists an effective Weil divisor $E_i$ on $\wt{B}_i$ such that $\omega_{\wt{B}_i}\cong \nu_i^*\omega_{B_i}(-E_i)$. Suppose that $E_i \neq 0$. Then $K_{\wt{B}_i}$ is not pseudo-effective and hence $B_i$ is uniruled by \cite[Corollary 0.3]{bdpp} applied to a resolution of $\wt{B}_i$. Let $C \subseteq B_i$ be a rational curve. By Corollary \ref{cor:numerically trivial}, we have $C\cdot K_\sF=0$. But this contradicts Lemma \ref{lemma:rational_curve_uniruled} since $K_X$ is nef. 
This shows that $E_i=0$, and hence $B_i$ is normal by Lemma \ref{lemma:conductor}. Then \cite[Corollary 0.3]{bdpp} applies to show that $B_i$ has canonical singularities since $K_{B_i}\equiv 0$.
The natural map $T_X|_{B_i} \to \sN_{B_i/X}$ is generically surjective and hence surjective since 
$T_X|_{B_i}$ and $\sN_{B_i/X}$ are numerically flat vector bundles. It follows that the tangent sheaf $T_{B_i}$ is locally free and numerically flat. By the solution of the Zariski-Lipman conjecture for canonical spaces (see \cite[Theorem 6.1]{greb_kebekus_kovacs_peternell10} or \cite[Theorem 1.1]{druel_zl}), we conclude that $B_i$ is smooth. 

By Proposition \ref{prop:abundance}, we have $\kappa(X)=\nu(X) \ge 0$. Applying
Proposition \ref{prop:log_abundance} to $(X,\sum_{i \in I}B_i)$, we see that 
$K_\sF$ is semiample with $\nu(K_\sF) \le 1$. This proves that $m=1$.

\medskip

By Lemma \ref{lemma:product} and Lemma \ref{lemma:minimal} together, the vector bundle $\sF$ is then numerically flat, and hence $K_\sF=0$, yielding a contradiction. This
finishes the proof of the theorem.
\end{proof}

\section{Proof of Theorem \ref{thm_intro:main}}\label{section:proof}

We are now in position to prove Theorem \ref{thm_intro:main}.

\begin{proof}[Proof of Theorem \ref{thm_intro:main}] Suppose first that $K_X$ is not nef. Then there exists a $\mathbb{P}^1$-bundle structure $\phi\colon X \to Y$ onto a finite \'etale quotient of an abelian variety, and $\sF$ induces a flat Ehresmann connection on $\phi$ by Theorem \ref{thm:suspension_complex_torus}.

Suppose from now on that $K_X$ is nef. 

If $\sN$ is pseudo-effective, then Theorem \ref{thm:case_normal_pseff_non_uniruled} applies to show that $K_\sF\equiv 0$. By Proposition \ref{prop:classification_zero_canonical_class}, there exists an abelian variety $A$ and a finite \'etale cover $\gamma\colon A \to X$ such that $\gamma^{-1}\sF$ is a linear foliation on $A$.

Suppose finally that $\sN$ is not pseudo-effective. Then $\sN^*$ is nef by Theorem \ref{thm_intro:divisor_zero_square}. Then Theorem \ref{thm:case_conormal_nef} applies to show that either there exists a smooth complete curve $B$ of genus at least $2$ as well as an abelian variety $A$, and a finite \'etale cover $\gamma \colon B \times A \to X$ such that $\gamma^{-1}\sF$ is induced by the projection morphism $B \times A \to B$, or $X$ is of general type and $\nu(K_\sF)=\dim X$. In the latter case, $\Omega_X^1$ is nef by Lemma \ref{lemma:minimal}, and hence $K_X$ is ample. In addition, $\kappa(K_\sF)=\dim X$ by Lemma \ref{lem:ambient_general_type_foliation_general_type}. This finishes the proof of theorem. 
\end{proof}

\begin{rem}
If we relax the integrability condition, then the conclusion of Theorem \ref{thm_intro:main} is false. Indeed, let $E$ be an elliptic curve, and let $B$ be an abelian variety of dimension $\dim B =n-1 \ge  1$. Set $A:= E \times B$. Let $\sM$ be a line bundle of degree $2$ on $E$, and set $\sL:=\textup{pr}_E^*\sM^*$. Let $(s_i,t_i) \in H^0(A,\sL^*)^{\times 2}$ for $i \in \{1,\ldots,n-1\}$. Suppose that the zero sets of the $s_i$ and the $t_i$ are pairwise disjoint. Write $T_A=\oplus_{1 \le i \le n}\sO_A v_i$. Then the sections $s_i$ and $t_i$ give an injective map of vector bundles $\sL \subset \sO_A v_i \oplus \sO_A v_{i+1}$. In addition, the induced map
$\sD=\sL^{\oplus n-1} \to T_A$ is an injective map of vector bundles as well. The vector bundle is obviously numerically projectively flat.
\end{rem}

\newcommand{\etalchar}[1]{$^{#1}$}
\providecommand{\bysame}{\leavevmode\hbox to3em{\hrulefill}\thinspace}
\providecommand{\MR}{\relax\ifhmode\unskip\space\fi MR }
\providecommand{\MRhref}[2]{%
  \href{http://www.ams.org/mathscinet-getitem?mr=#1}{#2}
}
\providecommand{\href}[2]{#2}

\end{document}